%% file: main.tex
\documentclass[12pt]{article}

\usepackage[margin=1.3in]{geometry}
\usepackage{setspace}
\usepackage{amsmath}
\usepackage{mathtools}
\usepackage{amsxtra}
\usepackage{amscd}
\usepackage{amsthm}
\allowdisplaybreaks
\usepackage{amssymb}
\usepackage[normalem]{ulem}
\usepackage[all]{xy}
\usepackage{comment}
\usepackage{color}
\usepackage{graphicx}
\usepackage{hyperref}
\usepackage[shortlabels]{enumitem}
\usepackage{array}

\usepackage{tikz}
\usepackage{tikz-cd}
\usepackage{quiver}
\usetikzlibrary{arrows.meta, decorations.markings, calc, positioning}
\usetikzlibrary{decorations.pathreplacing}
\usepackage{ifthen} 

\usepackage[backend=biber, style=alphabetic, giveninits=true]{biblatex}
\theoremstyle{definition}
\newtheorem{theorem}{Theorem}
\newtheorem{lemma}[theorem]{Lemma}
\newtheorem{proposition}[theorem]{Proposition}
\newtheorem{corollary}[theorem]{Corollary}

\newtheorem{definition}[theorem]{Definition}

\newtheorem{example}[theorem]{Example}

\newtheorem{remark}[theorem]{Remark}

\newtheorem{theoremletter}{Theorem}

\numberwithin{theorem}{section}
\newcommand{\Z}{\mathbb{Z} }
\newcommand{\D}{\mathbb{D} }
\newcommand{\N}{\mathbb{N} }

\newcommand{\C}{\mathbb{C} }

\newcommand{\up}{{\mkern.5mu\begin{tikzpicture}[x=.27em,y=.45em,decoration={markings,mark=at position 0.99 with {\arrow[black]{Stealth[length=4.8pt]}}}]
\draw[color=white,postaction={decorate}] (0,-.1) -- (0,.1);
\end{tikzpicture}\mkern.5mu}}
\newcommand{\dn}{{\mkern.5mu\begin{tikzpicture}[x=.27em,y=.45em,decoration={markings,mark=at position 0.99 with {\arrow[black]{Stealth[length=4.8pt]}}}]
\draw[color=white,postaction={decorate}] (0,.1) -- (0,-.1);
\end{tikzpicture}\mkern.5mu}}

\input{arcDiagram.tex}

\title{A quiver presentation of the type D arc algebras and intrinsic formality of $\D_4$}
\author{Eliot Grimont}
\date{}

\begin{document}

\maketitle

\begin{abstract}
We construct a presentation of the Khovanov arc algebras of type $D$ by a quiver with quadratic relations. As an application, we use this presentation and the Koszul dual to compute some Hochschild cohomology groups of $\mathbb{D}_4$, and show that $\mathbb{D}_4$ is intrinsically formal, that is, it admits no non-trivial $A_\infty$-deformations.
\end{abstract}

\tableofcontents
\newpage

\section*{Introduction}
\addcontentsline{toc}{section}{Introduction}

Khovanov arc algebras were originally introduced in the context of categorification, arising from Khovanov’s work~\cite{Khovanov2000} on the categorification of the Jones polynomial. In this framework, diagrammatic algebras built from arc diagrams play a central role in encoding the combinatorial and categorical structures underlying link invariants. Later, Brundan and Stroppel introduced extended versions of these arc algebras, called $K_m^n$ for $m,n \geq 1$, and showed that they are closely related to representation theory : more precisely, the category of finite dimensional $K_m^n$-modules is equivalent to the principal block of the parabolic category $\mathcal{O}$ associated to the parabolic subalgebra corresponding to $\mathfrak{gl}_m(\C) \oplus \mathfrak{gl}_n(\C) \subset \mathfrak{gl}_{m+n}(\C) $, see~\cite{BrundanStroppel2011a} and~\cite{BrundanStroppel2011b}. The construction was subsequently generalized to other Lie types, and in particular to type $D$ by Ehrig and Stroppel, where analogous diagrammatic algebras, called $\D_n$ for $n \geq 1$ appear and are again related to representation theory: indeed, the category of finite dimensional $\D_n$-modules is equivalent to the principal block of the parabolic category $\mathcal{O}_0^{\mathfrak{p}}(\mathfrak{so}(2n))$, where $\mathfrak{p}$ is a maximal parabolic subalgebra of type A, see~\cite{EhrigStroppel2016a}. Understanding the algebraic structure of these arc algebras is therefore important both from the perspective of categorification and from the viewpoint of representation theory.

In this paper, we focus on the type $D$ arc algebras $\D_n$ and study their algebraic structure through presentations by quivers with relations. Our approach is inspired by the work of Barmeier and Wang~\cite{BarmeierWang2026}, who exploited the Koszul property of the algebras $K_m^n$ to construct such a presentation in type $A$. A key observation motivating this work is that the algebras $\D_n$ are also Koszul, see~\cite{EhrigStroppel2016a,Seynnaeve2017,BeilinsonGinzburgSoergel1996}. In particular, they are quadratic algebras, and this allows one to study them through presentations by quivers with relations of degree two. In such a presentation, the vertices correspond to degree-zero diagrams, while the arrows correspond to generators in degree one. We adapt the approach of Barmeier and Wang to the type $D$ setting and use the diagrammatic structure of $\D_n$ to identify these generators and determine the relations between them. The main result of the paper is the following theorem.

\begin{theoremletter}\label{thmA}
    There is a graded $\Bbbk$-algebra isomorphism $\overline{\omega} : \Bbbk Q_n/I_n \to \D_n$, where $Q_n$ is the quiver defined in Section~\ref{sec:5.1}, and $I_n$ is generated by the monomial relations, square relations, and relations at vertices given in Theorem~\ref{thm:5.15}.
\end{theoremletter}

Although this description replaces the original diagrammatic definition with a more algebraic one, it provides effective tools for studying structural properties of the algebras. In particular, the quiver presentation makes it possible to compute the Koszul dual algebra, which provides further information on the homological structure of $\D_n$. As an illustration of the usefulness of this approach, we study in Section~\ref{sec:7} the $A_\infty$-structures associated with $\D_4$. More precisely, we use the quiver presentation and its Koszul dual to investigate whether $\D_4$ admits non-trivial $A_\infty$-deformations. This gives the following result.

\begin{theoremletter}\label{thmB}
    The algebra $\D_4$ is intrinsically formal.
\end{theoremletter}

The paper is organized as follows. In Section~\ref{sec:4}, we recall the definition of the arc algebras $\D_n$ in terms of arc diagrams, and describe their multiplication using the surgery procedure. Section~\ref{sec:5} is devoted to the construction of a presentation of $\D_n$ by a quiver with relations, which leads to Theorem~\ref{thmA}. We first define the quiver $Q_n$ in Section~\ref{sec:5.1}: its vertices are the degree-zero diagrams and its arrows are the generators in degree one. We then determine the relations in Section~\ref{sec:5.2}, by studying the paths of length two in $Q_n$. The last part of the paper, Section~\ref{sec:7}, concerns the case $n=4$. After a quick recall of the necessary background, we present in Section~\ref{mhc} a method due to Barmeier and Wang~\cite{BarmeierWang2020} to compute Hochschild cohomology groups, which are closely tied to $A_\infty$-deformations. This method relies on the Koszul dual of $\D_4$, which we describe in Section~\ref{sec:7.1}, and on a reduction system for the quiver $\overline{Q}_4$, which we construct in Section~\ref{sec:7.2}. We then compute some Hochschild cohomology groups of $\D_4$ in Section~\ref{sec:7.3}, and use it to prove Theorem~\ref{thmB}.

\textbf{Acknowledgments.} This work was initially carried out as part of my Master's thesis. I would like to thank my supervisor, Emily Norton, for her guidance and advice throughout this project. 

\textbf{AI declaration.} During the preparation of this work, I used Claude (Anthropic) to help create a template for drawing diagrams. The tool was used only for this purpose and was not used for data analysis, literature review, or writing the text of the manuscript. I reviewed and edited the output and takes full responsibility for the content of the publication.

\textbf{Convention.} We will use the convention that $\N =\{0,1,2,\ldots\}$.

\newpage

\section{Definition of the arc algebras $\D_n$}\label{sec:4}

In this section, we will define the algebras $\D_n$, following~\cite{EhrigStroppel2016b} and~\cite{EhrigTubbenhaurWilbert2019}, which are a generalization of the algebras $K_m^n$ defined in~\cite{stroppel2009}. We need to define the arc diagrams, obtained by stacking a cup diagram, a weight, and a cap diagram in an admissible way. After defining these concepts, we will explain how to multiply arc diagrams in order to have an algebra structure on $\D_n$.

\begin{definition}\label{def:4.1}
    Let $n \geq 1$ be an integer. We fix a horizontal line in the plane and $n$ distinct points on it. A weight of type $n$ is a sequence of $k \; \up$'s and $(n-k) \; \dn$'s, that we can place on the fixed points on the line. We will denote by $\Delta_n^0$ the set of weights of type $n$ with an even number of $\up$'s, and $\Delta_n^1$ the set of weights of type $n$ with an odd number of $\up$'s. These sets are called blocks.
\end{definition}

\begin{definition}\label{def:4.3}
    Let $n \geq 1$ be an integer, we fix a horizontal line in the plane and $n$ distinct points on it. A cup diagram consists of a diagram with cups and vertical lines under the fixed horizontal line, and attached to the fixed points on it. There cannot be intersections, and each cup and line can be decorated or not, with a dot, following the conditions :
    \begin{itemize}
        \item Every decorated cup or line cannot be to the right of a vertical line. 
        \item Every decorated cup cannot be nested in another cup.
    \end{itemize}
    Visually, these conditions tell us that any dot from a decorated arc can be connected with the left boundary of the diagram without intersecting any other arc, if we see the cup diagram in a rectangle with top edge lying on the fixed line.
\end{definition}

\begin{example}\label{ex:4.4}
    The following visual example will help to understand Definition~\ref{def:4.3}. Let $n = 7$, the following diagram is a cup diagram because, as we said in the definition, each dot can be connected with the left boundary of the rectangle.
    $$\begin{tikzpicture}[scale=0.8]
        \draw[green!80!black,thick] (.5,-0.7) .. controls (.3,-1) .. (0,-1);
        \draw[green!80!black,thick] (6,-1.15) .. controls (5,-1.6) .. (0,-1.6);

        \bigdCUP{0}{1}{1}
        \CUP{3}{4}{1}
        \CUP{2}{5}{.75}
        \bigdDOWNLINE{6}{2}
        \BFILL{7}
        
        \draw[dotted,thick] (0,0) -- (6,0);
        \draw[dotted,thick] (0,-2) -- (6,-2);
        \draw[dotted,thick] (0,0) -- (0,-2);
    \end{tikzpicture}$$
    And the following diagram is not a cup diagram because the decorated cup is on the right of a vertical line.
    $$\begin{tikzpicture}[scale=0.8]
        \draw[red,thick] (5.5,-0.7) .. controls (4.5,-1.3) .. (4,-1.3);

        \CUP{0}{1}{1}
        \CUP{2}{3}{1}
        \bigdCUP{5}{6}{1}
        \DOWNLINE{4}{2}
        \BFILL{7}
        
        \draw[dotted,thick] (0,0) -- (6,0);
        \draw[dotted,thick] (0,-2) -- (6,-2);
        \draw[dotted,thick] (0,0) -- (0,-2);
        \draw[dotted,thick] (6,0) -- (6,-2);
    \end{tikzpicture}$$
\end{example}

\begin{definition}\label{def:4.5}
    A cap diagram is the mirror image of a cup diagram obtained by reflecting all cups and lines across the fixed horizontal line.
\end{definition}

We now have all the ingredients to properly define the arc diagrams obtained by stacking a cup diagram, a weight, and a cap diagram in an admissible way.

\begin{definition}\label{def:4.6}
    Let $n \geq 1$ be an integer and $p \in \{0,1\}$. We fix a horizontal line in the plane and $n$ distinct points on it. An arc diagram associated to the block $\Delta_n^p$, is the glueing of a cup and a cap diagram of size $n$ along a weight belonging to $\Delta_n^p$, satisfying the following conditions.
    \begin{enumerate}
        \item[\textit{(i)}] Any undecorated cup or cap is glued to exactly one $\up$ and one $\dn$.
        \item[\textit{(ii)}] Any decorated cup or cap is glued to either two $\up$'s or two $\dn$'s.
        \item[\textit{(iii)}] Any undecorated half line is glued to one $\dn$.
        \item[\textit{(iv)}] Any decorated half line is glued to one $\up$.
    \end{enumerate}
    Visually each admisssible cup, cap or half line looks like one of the following :
    $$\begin{tikzpicture}[scale=.7]
        \weight{0,1}
        \CUP{0}{1}{1}
        \BFILL{2}

        \begin{scope}[xshift=3cm]
            \weight{1,0}
            \CUP{0}{1}{1}
            \BFILL{2}
        \end{scope}

        \begin{scope}[xshift=6cm]
            \weight{1,1}
            \bigdCUP{0}{1}{1}
            \BFILL{2}
        \end{scope}

        \begin{scope}[xshift=9cm]
            \weight{0,0}
            \bigdCUP{0}{1}{1}
            \BFILL{2}
        \end{scope}

        \begin{scope}[xshift=12cm]
            \weight{0}
            \DOWNLINE{0}{1}
            \BFILL{1}
        \end{scope}

        \begin{scope}[xshift=14cm]
            \weight{1}
            \bigdDOWNLINE{0}{1}
            \BFILL{1}
        \end{scope}

      \begin{scope}[yshift=-3cm]
        \weight{0,1}
        \CAP{0}{1}{1}
        \TFILL{2}

        \begin{scope}[xshift=3cm]
            \weight{1,0}
            \CAP{0}{1}{1}
            \TFILL{2}
        \end{scope}

        \begin{scope}[xshift=6cm]
            \weight{1,1}
            \bigdCAP{0}{1}{1}
            \TFILL{2}
        \end{scope}

        \begin{scope}[xshift=9cm]
            \weight{0,0}
            \bigdCAP{0}{1}{1}
            \TFILL{2}
        \end{scope}

        \begin{scope}[xshift=12cm]
            \weight{0}
            \UPLINE{0}{1}
            \TFILL{1}
        \end{scope}

        \begin{scope}[xshift=14cm]
            \weight{1}
            \bigdUPLINE{0}{1}
            \TFILL{1}
        \end{scope}
      \end{scope}

    \end{tikzpicture}$$
\end{definition}

\begin{example}\label{ex:4.7}
    Let $n = 4$. Here is an example of an arc diagram :
    $$\begin{tikzpicture}[scale=.7]
        \weight{1,1,0,1}
        \CUP{2}{3}{1}
        \CAP{1}{2}{1}
        \bigdCAP{0}{3}{.75}
        \bigdCUP{0}{1}{1}
        \FILL{4}
    \end{tikzpicture}$$
    And here are examples of diagrams that are not arc diagrams :
    $$\begin{tikzpicture}[scale=.7]
        \weight{0,1,1,0}
        \CAP{2}{3}{1}
        \bigdCAP{0}{1}{1}
        \bigdCUP{0}{1}{1}
        \bigdDOWNLINE{2}{1}
        \DOWNLINE{3}{1}
        \FILL{4}

        \begin{scope}[xshift=7cm]
            \weight{1,1,0,0}
            \bigdCAP{0}{1}{1}
            \CUP{1}{2}{1}
            \DOWNLINE{0}{1}
            \UPLINE{2}{1}
            \LINE{3}{1}
            \FILL{4}
        \end{scope}
    \end{tikzpicture}$$
    because they break the conditions $\textit{(ii)}$ and $\textit{(iii)}$, respectively.
\end{example}

\begin{definition}\label{def:4.8}
    \begin{itemize}
        \item A cap or a cup (decorated or not) is said to be oriented clockwise, if it is attached on the right side to a $\dn$, and is said to be oriented counter-clockwise, if it is attached on the right side to a $\up$.
        \item The degree of an arc diagram is given by the number of clockwise caps and clockwise cups it contains.
    \end{itemize}
\end{definition}

\begin{example}\label{ex:4.9}
    We give the table of all possible arc diagrams asssociated to the block $\Delta_3^0$, with their respective degrees
    $$\begin{tabular}{c|c|c|c}
        weight & degree $0$ & degree $1$ & degree $2$ \\ 
        \hline
        $\dn \dn \dn$ & $\begin{tikzpicture}[scale=0.3, baseline={(0,0)}]
    \Bweight{0,0,0}
    \LINE{0}{1}
    \LINE{1}{1}
    \LINE{2}{1}
    \FILL{3}
    \draw[strand, draw=white] (0,-1.1) -- (0,-1.4);
\end{tikzpicture}$ & $\begin{tikzpicture}[scale=0.3, baseline={(0,0)}]
    \Bweight{0,0,0}
    \LINE{2}{1}
    \DOWNLINE{0}{1}
    \DOWNLINE{1}{1}
    \sdCAP{0}{1}{1}
    \FILL{3}
\end{tikzpicture} \qquad \begin{tikzpicture}[scale=0.3, baseline={(0,0)}]
    \Bweight{0,0,0}
    \LINE{2}{1}
    \UPLINE{0}{1}
    \UPLINE{1}{1}
    \sdCUP{0}{1}{1}
    \FILL{3}
    \end{tikzpicture}$ & $\begin{tikzpicture}[scale=0.3, baseline={(0,0)}]
    \Bweight{0,0,0}
    \LINE{2}{1}
    \sdCAP{0}{1}{1}
    \sdCUP{0}{1}{1}
    \FILL{3}
    \end{tikzpicture}$ \\ 
        \hline
        $\up \up \dn $ & $\begin{tikzpicture}[scale=0.3, baseline={(0,0)}]
    \Bweight{1,1,0}
    \LINE{2}{1}
    \sdCAP{0}{1}{1}
    \sdCUP{0}{1}{1}
    \FILL{3}
    \draw[strand, draw=white] (0,-1.1) -- (0,-1.4);
\end{tikzpicture}$ & $\begin{tikzpicture}[scale=0.3, baseline={(0,0)}]
    \Bweight{1,1,0}
    \sdDOWNLINE{0}{1}
    \sdCAP{0}{1}{1}
    \CUP{1}{2}{1}
    \UPLINE{2}{1}
    \FILL{3}
\end{tikzpicture} \qquad \begin{tikzpicture}[scale=0.3, baseline={(0,0)}]
    \Bweight{1,1,0}
    \sdUPLINE{0}{1}
    \sdCUP{0}{1}{1}
    \CAP{1}{2}{1}
    \DOWNLINE{2}{1}
    \FILL{3}
    \end{tikzpicture}$ & $\begin{tikzpicture}[scale=0.3, baseline={(0,0)}]
    \Bweight{1,1,0}
    \CUP{1}{2}{1}
    \CAP{1}{2}{1}
    \sdLINE{0}{1}
    \FILL{3}
    \end{tikzpicture}$ \\ 
        \hline
        $\up \dn \up $ & $\begin{tikzpicture}[scale=0.3, baseline={(0,0)}]
    \Bweight{1,0,1}
    \CUP{1}{2}{1}
    \CAP{1}{2}{1}
    \sdLINE{0}{1}
    \FILL{3}
    \draw[strand, draw=white] (0,-1.1) -- (0,-1.4);
\end{tikzpicture}$ & $\begin{tikzpicture}[scale=0.3, baseline={(0,0)}]
    \Bweight{1,0,1}
    \sdDOWNLINE{0}{1}
    \CAP{0}{1}{1}
    \CUP{1}{2}{1}
    \sdUPLINE{2}{1}
    \FILL{3}
\end{tikzpicture} \qquad \begin{tikzpicture}[scale=0.3, baseline={(0,0)}]
    \Bweight{1,0,1}
    \sdUPLINE{0}{1}
    \CUP{0}{1}{1}
    \CAP{1}{2}{1}
    \sdDOWNLINE{2}{1}
    \FILL{3}
    \end{tikzpicture}$ & $\begin{tikzpicture}[scale=0.3, baseline={(0,0)}]
    \Bweight{1,0,1}
    \CAP{0}{1}{1}
    \CUP{0}{1}{1}
    \sdLINE{2}{1}
    \FILL{3}
    \end{tikzpicture}$ \\ 
        \hline
        $\dn \up \up $ & $\begin{tikzpicture}[scale=0.3, baseline={(0,0)}]
    \Bweight{0,1,1}
    \CAP{0}{1}{1}
    \CUP{0}{1}{1}
    \sdLINE{2}{1}
    \FILL{3}
    \draw[strand, draw=white] (0,-1.1) -- (0,-1.4);
\end{tikzpicture}$ 
\end{tabular}$$
\end{example}

Before explaining how to multiply arc diagrams, we need to introduce one last important concept, which explains how to construct a cup or cap diagram from a fixed weight.

\begin{definition}\label{def:4.10}
    Let $\lambda \in \Delta_n^p$ be a weight. We denote by $\underline{\lambda}$ the unique cup diagram obtained by applying in order the following rules
    \begin{enumerate}
        \item Connect recursively all neighbouring $\dn \cdots \up$ pairs in $\lambda$ with an undecorated cup.
        \item Connect recursively from left to right every pair of neighbouring $\up \cdots \up$ with a decorated cup.
        \item Draw undecorated half-lines extending downwards starting at every remaining $\dn$.
        \item If a single $\up$ remains, draw a decorated half-line extending downwards starting at it.
        \item Delete the weight $\lambda$.
    \end{enumerate}
    We define the cap diagram $\overline{\lambda}$ to be the mirror image of $\underline{\lambda}$ obtained by reflecting all cups and lines across the fixed horizontal line.
\end{definition}

\begin{example}\label{ex:4.11}
    Let $\lambda = \up \dn \up \up \up \dn$, then we detail the process step by step of the construction of $\underline{\lambda}$
    $$\begin{tikzpicture}[x=.4cm,y=.4cm]
    \Bweight{1,0,1,1,1,0}
    \CUP{1}{2}{1}
    \BBFILL{1}
    \BBFILL{2}
    \begin{scope}[xshift=3.5cm]
        \Bweight{1,0,1,1,1,0}
        \CUP{1}{2}{1}
        \dCUP{0}{3}{.75}
        \BBFILL{2}
        \BBFILL{3}
        \BBFILL{1}
        \BBFILL{0}
    \end{scope}
    \begin{scope}[xshift=7cm]
        \Bweight{1,0,1,1,1,0}
        \CUP{1}{2}{1}
        \dCUP{0}{3}{.75}
        \dDOWNLINE{4}{1.5}
        \DOWNLINE{5}{1.5}
        \BFILL{6}
    \end{scope}
    \begin{scope}[xshift=10.5cm]
        \CUP{1}{2}{1}
        \dCUP{0}{3}{.75}
        \dDOWNLINE{4}{1.5}
        \DOWNLINE{5}{1.5}
        \BFILL{6}
    \end{scope}
    \draw[->, very thick] (6,0) -- (7.5,0);
    \draw[->, very thick] (14.8,0) -- (16.3,0);
    \draw[->, very thick] (23.6,0) -- (25.1,0);
\end{tikzpicture}$$
\end{example}

\begin{remark}\label{rmk:4.12}
    Observe that any cup diagram is of the form $\underline{\lambda}$ for some unique weight $\lambda$. Hence, any arc diagram can be written in the form $\underline{\alpha}\lambda\overline{\beta}$ for unique weights $\alpha, \lambda$ and $\beta$.
\end{remark}

\begin{definition}\label{def:4.13}
     Let $\lambda \in \Delta_n^p$ be a weight. Then we define the arc diagram $e_\lambda$ to be $\underline{\lambda}\lambda\overline{\lambda}$.
\end{definition}

\begin{remark}\label{rmk:4.14}
    Note that for each $\lambda \in \Delta_n^p$, the arc diagram $e_\lambda$ is the only arc diagram of degree $0$ and weight $\lambda$. These diagrams will be very important in Section~\ref{sec:5}.
\end{remark}

Now we can finally define the multiplication and the diagrammatic algebra structure. The multiplication process is long, but we will give many examples of diagram multiplications in order to understand it well.

\begin{definition}\label{def:4.15}
    Let $\Bbbk$ be a field of characteristic $0$. We define $\D_n^p$ to be the $\Bbbk$-vector space with basis given by the arc diagrams associated to the block $\Delta_n^p$. We will denote by ${}_\lambda(\D_n^p){}_\mu$ the vector subspace of $\D_n^p$, with basis given by the arc diagrams of the form $\underline{\lambda}\nu\overline{\mu}$.
\end{definition}

\begin{definition}\label{def:4.16}
    Let $\lambda, \mu \in \Delta_n^p$ be weights. We say that the glueing of $\underline{\lambda}$ above $\overline{\mu}$, denoted $\underline{\lambda}\overline{\mu}$, can be oriented, if there exists a weight $\nu \in \Delta_n^p$ such that $\underline{\lambda}\nu\overline{\mu}$ is an arc diagram.
\end{definition}

\begin{definition}\label{def:4.17}
    Let $\lambda, \mu \in \Delta_n^p$ be weights, and assume that $\underline{\lambda}\overline{\mu}$ can be oriented. Then we define $I_{\underline{\lambda}\overline{\mu}}$ to be the ideal in $\Bbbk[X_i \ \vert \ 1 \leq i \leq n ]$ generated by
    \begin{itemize}
        \item $X_i^2$ for all $1 \leq i \leq n$.
        \item $X_i - X_j$ if $i$ and $j$ are connected by an undecorated cup or cap in $\underline{\lambda}\overline{\mu}$.
        \item $X_i + X_j$ if $i$ and $j$ are connected by a decorated cup or cap in $\underline{\lambda}\overline{\mu}$.
        \item $X_i$ if $i$ is the endpoint of a half-line in $\underline{\lambda}\overline{\mu}$.
    \end{itemize}
    Then we also define
    $$\mathcal{M}(\underline{\lambda}\overline{\mu}) = \Bbbk[X_i \ \vert \ 1 \leq i \leq n ]/I_{\underline{\lambda}\overline{\mu}}$$
\end{definition}

\begin{example}\label{ex:4.18}
    Let $\lambda = \up \up \dn \up$ and $\mu = \up \dn \up \up$. Then
    $$\underline{\lambda}\overline{\mu} \ = \ \begin{tikzpicture}[scale=.8, baseline = {(0,0.15)}]
        \CUP{2}{3}{1}
        \CAP{1}{2}{1}
        \bigdCAP{0}{3}{.75}
        \bigdCUP{0}{1}{1}
        \FILL{4}
    \end{tikzpicture}$$
    can be oriented, so that
    $$I_{\underline{\lambda}\overline{\mu}} = (X_1^2, X_2^2, X_3^2, X_4^2, X_1+X_2, X_1+X_4,X_2-X_3,X_3-X_4)$$
    and then 
    $$\mathcal{M}(\underline{\lambda}\overline{\mu}) \cong \Bbbk[X_4]/(X_4^2)$$
\end{example}

\begin{proposition}[{\cite[Prop. 4.5]{EhrigStroppel2016a}}]\label{prop:4.19}
    Let $\lambda, \mu \in \Delta_n^p$ be weights, and assume that $\underline{\lambda}\overline{\mu}$ can be oriented. Then there is an isomorphism of vector spaces
    \begin{align*}
        \varphi_{\underline{\lambda}\overline{\mu}} : {}_\lambda(\D_n^p){}_\mu &\longrightarrow \mathcal{M}(\underline{\lambda}\overline{\mu}) \\
        \underline{\lambda}\nu\overline{\mu} &\mapsto \prod_{C \in \mathcal{C}_{\text{clock}}(\underline{\lambda}\nu\overline{\mu})} X_{t(C)}
    \end{align*}
    where $\mathcal{C}_{\text{clock}}(\underline{\lambda}\nu\overline{\mu})$ is the set of clockwise oriented circles in $\underline{\lambda}\nu\overline{\mu}$, and $t(C)$ is the position of the rightmost arrow of $\nu$ which $C$ goes through.
\end{proposition}

\begin{proof}
    This map is pretty natural. Indeed, let's construct an inverse. Suppose that there are $k$ circles in $\underline{\lambda}\overline{\mu}$ denoted $C_1, \ldots, C_k$. Then observe that any monomial in $\mathcal{M}(\underline{\lambda}\overline{\mu})$ is of the form $\prod_{i=1}^{k} X_{t(C_i)}^{p_i}$, where $p_i \in \{ 0,1 \}$ for $1 \leq i \leq k$. Then we set
    $$(\varphi_{\underline{\lambda}\overline{\mu}})^{-1}\left(\prod_{i=1}^{k} X_{t(C_i)}^{p_i}\right) = \underline{\lambda}\nu\left(\prod_{i=1}^{k} X_{t(C_i)}^{p_i}\right)\overline{\mu}$$
    where $\nu\left(\prod_{i=1}^{k} X_{t(C_i)}^{p_i}\right)$ orients the circle $C_i$ clockwisely if $p_i = 1$ and counterclockwisely if $p_i = 0$, and has no choice for the open arcs. Finally we extend this map linearly.
\end{proof}

\begin{definition}\label{def:4.21}
    We recall that $\D_n^p$ is defined to be the $\Bbbk$-vector space with basis given by the arc diagrams associated to the block $\Delta_n^p$. It can be turned into a $\N$-graded $\Bbbk$-algebra, with grading given by the degrees of the diagrams, and multiplication given as follows. \\
    Let $a = \underline{\alpha}\lambda\overline{\beta}$ and $b = \underline{\gamma}\mu\overline{\delta}$, then we set $ab = 0$ if $\beta \neq \gamma$. Otherwise, if $\beta = \gamma$, the product $ab$ is obtained by drawing $a$ below $b$, without the weights, connecting the open arcs at the same positions, and applying inductively the surgery processes which consist to transform a cup-cap pair into lines in the following way
    $$\begin{tikzpicture}[x=.8cm,y=.8cm]
    \CUP{0}{1}{1}
    \UPLINE{0}{.5}
    \UPLINE{1}{.5}
    \FILL{2}
    \node at (-0.55,0) {$\cdots$};   
    \node at (0.55,0) {$\cdots$};    
    \node at (1.55,0) {$\cdots$};    
    \begin{scope}[yshift=-2cm]
        \CAP{0}{1}{1}
        \DOWNLINE{0}{.5}
        \DOWNLINE{1}{.5}
        \FILL{2}
        \node at (-0.55,0) {$\cdots$};   
        \node at (0.55,0) {$\cdots$};    
        \node at (1.55,0) {$\cdots$};    
    \end{scope}
    \begin{scope}[xshift = 5cm]
    \UPLINE{0}{.5}
    \UPLINE{1}{.5}
    \FILL{2}
    \node at (-0.55,0) {$\cdots$};   
    \node at (0.55,0) {$\cdots$};    
    \node at (1.55,0) {$\cdots$};    
    \begin{scope}[yshift=-2cm]
        \UPLINE{0}{2.265}
        \UPLINE{1}{2.265}
        \DOWNLINE{0}{.5}
        \DOWNLINE{1}{.5}
        \FILL{2}
        \node at (-0.55,0) {$\cdots$};   
        \node at (0.55,0) {$\cdots$};    
        \node at (1.55,0) {$\cdots$};    
    \end{scope}
    \end{scope}
    \draw[->, very thick] (3,-1.4) -- (4,-1.4);
\end{tikzpicture}$$
that we call an undotted surgery, or
$$\begin{tikzpicture}[x=.8cm,y=.8cm]
    \bigdCUP{0}{1}{1}
    \UPLINE{0}{.5}
    \UPLINE{1}{.5}
    \FILL{2}
    \node at (-0.55,0) {$\cdots$};   
    \node at (0.55,0) {$\cdots$};    
    \node at (1.55,0) {$\cdots$};    
    \begin{scope}[yshift=-2cm]
        \bigdCAP{0}{1}{1}
        \DOWNLINE{0}{.5}
        \DOWNLINE{1}{.5}
        \FILL{2}
        \node at (-0.55,0) {$\cdots$};   
        \node at (0.55,0) {$\cdots$};    
        \node at (1.55,0) {$\cdots$};    
    \end{scope}
    \begin{scope}[xshift = 5cm]
    \UPLINE{0}{.5}
    \UPLINE{1}{.5}
    \FILL{2}
    \node at (-0.55,0) {$\cdots$};   
    \node at (0.55,0) {$\cdots$};    
    \node at (1.55,0) {$\cdots$};    
    \begin{scope}[yshift=-2cm]
        \UPLINE{0}{2.265}
        \UPLINE{1}{2.265}
        \DOWNLINE{0}{.5}
        \DOWNLINE{1}{.5}
        \FILL{2}
        \node at (-0.55,0) {$\cdots$};   
        \node at (0.55,0) {$\cdots$};    
        \node at (1.55,0) {$\cdots$};    
    \end{scope}
    \end{scope}
    \draw[->, very thick] (3,-1.4) -- (4,-1.4);
\end{tikzpicture}$$
    that we call a dotted surgery. During all these steps we will define a map $\chi$ in order to reorient the resulting diagram. We start with a map
    $$\chi' : \Bbbk[X_i \ \vert \ 1 \leq i \leq n ] \otimes \Bbbk[X_i \ \vert \ 1 \leq i \leq n ] \to \Bbbk[X_i \ \vert \ 1 \leq i \leq n ] $$
    associated to our surgery procedure. To obtain the image of $f \otimes g \in \Bbbk[X_i \ \vert \ 1 \leq i \leq n ] \otimes \Bbbk[X_i \ \vert \ 1 \leq i \leq n ]$, we will apply the following algorithm :
    \begin{enumerate}
        \item After drawing $a$ below $b$, define $y = f \cdot g \in \Bbbk[X_i \ \vert \ 1 \leq i \leq n ]$.
        \item After connecting all open arcs at same positions, multiply $y$ with $\prod_{C \in \mathcal{C}_{\text{clock}}^{\text{new}}} X_{t(C)}$ where $\mathcal{C}_{\text{clock}}^{\text{new}}$ is the set of all clockwise circles created by this step of connection of open arcs, and $t(C)$ is the position of the rightmost point of the fixed horizontal line which $C$ goes through.
        \item For each surgery, consult the following table 
        
        \textbf{Merge :} If the surgery produces one connected component out of two, leave $y$ as it is.

        \textbf{Split :} If the surgery produces two connected component out of one, then multiply $y$ with
        \[
\begin{dcases}
(X_i + X_j) & \text{if the surgery occurs at positions } i \text{ and } j \text{,}\\
             & \text{and either the surgery is undotted,} \\
             & \text{or dotted but non-nested} \\[1em]
(X_i - X_j) & \text{if the surgery occurs at positions } i \text{ and } j \text{,}\\
             & \text{and the surgery is dotted and nested with} \\
             & \text{position } i \text{ belonging to the inner circle} \\
             & \text{after surgery},
\end{dcases}
\]
        where a nested surgery is a split surgery where both of the resulting components are circles one of which is contained in the other one.

        \textbf{Reconnect :} If the surgery produces two connected components out of two, which only happens if both are open arcs, then either the two lines (before the surgery) are propagating and we leave $y$ as it is, otherwise we multiply $y$ with $0$. We say that an open arc is propagating if it connects the bottom boundary to the top boundary (or vice-versa).
        \item If the resulting diagram cannot be oriented, multiply $y$ with $0$.
        \item We define $\chi'(f \otimes g)$ to be equal to the resulting polynomial $y$.
    \end{enumerate}
    One can show (see~\cite{EhrigStroppel2016a}) that this map $\chi'$ descends to a map
    $$\chi : \mathcal{M}(\underline{\alpha}\overline{\beta}) \otimes \mathcal{M}(\underline{\beta}\overline{\delta}) \to \mathcal{M}(\underline{\alpha}\overline{\delta})$$
    that doesn't depend on our chosen surgery procedure. Then consider the composition
    $$\Phi_{\alpha\beta\delta} : {}_\alpha(\D_n^p)_\beta \otimes {}_\beta(\D_n^p)_\delta \xrightarrow{\varphi_{\underline{\alpha}\overline{\beta}} \otimes \varphi_{\underline{\beta}\overline{\delta}}} \mathcal{M}(\underline{\alpha}\overline{\beta}) \otimes \mathcal{M}(\underline{\beta}\overline{\delta}) \xrightarrow{\chi} \mathcal{M}(\underline{\alpha}\overline{\delta}) \xrightarrow{(\varphi_{\underline{\alpha}\overline{\delta}})^{-1}} {}_\alpha(\D_n^p)_\delta $$
    Finally, we define the multiplication on the basis of $\D_n^p$ by
    $$(\underline{\alpha}\lambda\overline{\beta}) \cdot (\underline{\gamma}\mu\overline{\delta}) = \left\{
    \begin{array}{cc}
        \Phi_{\alpha\beta\delta}((\underline{\alpha}\lambda\overline{\beta}) \otimes (\underline{\gamma}\mu\overline{\delta}))  &\mbox{if } \beta = \gamma \\
        0 &\mbox{else}
    \end{array} \right.$$
    and we extend it bilinearly to obtain a well-defined associative grading preserving multiplication, as shown in~\cite{EhrigStroppel2016a}.
\end{definition}

\begin{remark}\label{rmk:4.22}
    This definition is called the sign adjusted version of the usual multiplication defined in~\cite{EhrigStroppel2016a}. It has been introduced in~\cite{EhrigTubbenhaurWilbert2019} by Ehrig, Tubbenhauer and Wilbert, where they showed that this is an equivalent definition to the original one. However this multiplication is useful because we can choose the order we want for the surgery processes, and there are less changes of signs, so less room for mistakes.
\end{remark}

\begin{example}\label{ex:4.23}
    \begin{itemize}
        \item Consider first that
    $$\underline{\alpha}\lambda\overline{\beta} \ = \ \begin{tikzpicture}[scale=.5, baseline={(0,-0.1)}]
    \Bweight{1,1,0,0}
    \dCAP{0}{1}{1}
    \dCAP{2}{3}{1}
    \CUP{1}{2}{1}
    \CUP{0}{3}{.75}
    \FILL{4}
\end{tikzpicture} \quad \text{ and } \quad \underline{\gamma}\mu\overline{\delta} \ = \ \begin{tikzpicture}[scale=.5, baseline={(0,-0.1)}]
    \Bweight{1,1,0,0}
    \CAP{1}{2}{1}
    \CAP{0}{3}{.75}
    \CUP{1}{2}{1}
    \CUP{0}{3}{.75}
    \FILL{4}
\end{tikzpicture}$$
Then observe that $\beta \neq \gamma$, so
$$(\underline{\alpha}\lambda\overline{\beta}) \cdot (\underline{\gamma}\mu\overline{\delta}) \ = \ \begin{tikzpicture}[scale=.5, baseline={(0,.68)}]
    \Bweight{1,1,0,0}
    \dCAP{0}{1}{1}
    \dCAP{2}{3}{1}
    \CUP{1}{2}{1}
    \CUP{0}{3}{.75}
    \FILL{4}

    \begin{scope}[yshift=3cm]
        \Bweight{1,1,0,0}
        \CAP{1}{2}{1}
        \CAP{0}{3}{.75}
        \CUP{1}{2}{1}
        \CUP{0}{3}{.75}
        \FILL{4}
    \end{scope}
\end{tikzpicture} \ = \ 0$$
\item Now, consider
$$\underline{\alpha}\lambda\overline{\beta} \ = \ \begin{tikzpicture}[scale=.5, baseline={(0,-0.1)}]
    \Bweight{0,1,0,0}
    \CAP{1}{2}{1}
    \CUP{1}{2}{1}
    \UPLINE{0}{1}
    \UPLINE{3}{1}
    \dCUP{0}{3}{.75}
    \FILL{4}
\end{tikzpicture} \quad \text{ and } \quad \underline{\gamma}\mu\overline{\delta} \ = \ \begin{tikzpicture}[scale=.5, baseline={(0,-0.1)}]
    \Bweight{0,0,1,0}
    \CAP{1}{2}{1}
    \dCAP{0}{3}{.75}
    \CUP{1}{2}{1}
    \DOWNLINE{0}{1}
    \DOWNLINE{3}{1}
    \FILL{4}
\end{tikzpicture}$$
This time, $\beta = \gamma$, and we have
$$\psi_{\underline{\alpha}\overline{\beta}}(\underline{\alpha}\lambda\overline{\beta}) = X_3 \ \text{ and } \ \psi_{\underline{\gamma}\overline{\delta}}(\underline{\gamma}\mu\overline{\delta}) = 1$$
So we set $y= X_3 \cdot 1 = X_3$, and we now connect the arcs and perform the surgery procedure, keeping track of $y$ to know the resulting orientation
$$(\underline{\alpha}\lambda\overline{\beta}) \cdot (\underline{\gamma}\mu\overline{\delta}) \ \overset{X_3}{=} \ \begin{tikzpicture}[scale=.5, baseline={(0,.68)}]
    \Bweight{0,1,0,0}
    \CAP{1}{2}{1}
    \CUP{1}{2}{1}
    \UPLINE{0}{1}
    \UPLINE{3}{1}
    \dCUP{0}{3}{.75}
    \FILL{4}

    \begin{scope}[yshift=3cm]
        \Bweight{0,0,1,0}
    \CAP{1}{2}{1}
    \dCAP{0}{3}{.75}
    \CUP{1}{2}{1}
    \DOWNLINE{0}{1}
    \DOWNLINE{3}{1}
    \FILL{4}
    \end{scope}
\end{tikzpicture} \ \underset{\text{Open arc}}{\overset{X_3X_4}{=}} \  \begin{tikzpicture}[scale=.5, baseline={(0,.68)}]
    \Bweight{0,1,0,0}
    \CAP{1}{2}{1}
    \CUP{1}{2}{1}
    \UPLINE{0}{1.5}
    \UPLINE{3}{1.5}
    \dCUP{0}{3}{.75}
    \FILL{4}

    \begin{scope}[yshift=3cm]
        \Bweight{0,0,1,0}
    \CAP{1}{2}{1}
    \dCAP{0}{3}{.75}
    \CUP{1}{2}{1}
    \DOWNLINE{0}{1.5}
    \DOWNLINE{3}{1.5}
    \FILL{4}
    \end{scope}
\end{tikzpicture} \ \underset{\text{Merge}}{\overset{X_3X_4}{=}} \ \begin{tikzpicture}[scale=.5, baseline={(0,.68)}]
    \Bweight{0,1,0,0}
    \UPLINE{1}{1.5}
    \UPLINE{2}{1.5}
    \CUP{1}{2}{1}
    \UPLINE{0}{1.5}
    \UPLINE{3}{1.5}
    \dCUP{0}{3}{.75}
    \FILL{4}

    \begin{scope}[yshift=3cm]
        \Bweight{0,1,0,0}
    \CAP{1}{2}{1}
    \dCAP{0}{3}{.75}
    \DOWNLINE{1}{1.5}
    \DOWNLINE{2}{1.5}
    \DOWNLINE{0}{1.5}
    \DOWNLINE{3}{1.5}
    \FILL{4}
    \end{scope}
\end{tikzpicture} \ = \  \begin{tikzpicture}[scale=.5, baseline={(0,-0.1)}]
    \Bweight{0,1,0,0}
    \CAP{1}{2}{1}
    \dCAP{0}{3}{.75}
    \CUP{1}{2}{1}
    \dCUP{0}{3}{.75}
    \FILL{4}
\end{tikzpicture} $$
\item Finally, consider
$$\underline{\alpha}\lambda\overline{\beta} \ = \ \begin{tikzpicture}[scale=.5, baseline={(0,-0.1)}]
    \Bweight{0,0,1,1}
    \dCAP{0}{1}{1}
    \dCAP{2}{3}{1}
    \CUP{1}{2}{1}
    \CUP{0}{3}{.75}
    \FILL{4}
\end{tikzpicture} \quad \text{ and } \quad \underline{\gamma}\mu\overline{\delta} \ = \ \begin{tikzpicture}[scale=.5, baseline={(0,-0.1)}]
    \Bweight{1,1,0,0}
    \dCUP{0}{1}{1}
    \dCUP{2}{3}{1}
    \CAP{1}{2}{1}
    \CAP{0}{3}{.75}
    \FILL{4}
\end{tikzpicture}$$
Again, we have $\beta = \gamma$, and 
$$\psi_{\underline{\alpha}\overline{\beta}}(\underline{\alpha}\lambda\overline{\beta}) = 1 \ \text{ and } \ \psi_{\underline{\gamma}\overline{\delta}}(\underline{\gamma}\mu\overline{\delta}) = X_4$$
So we set $y= 1 \cdot X_4 = X_4$, and we perform the surgery procedure, keeping track of $y$ to know the resulting orientation
\[
\begin{aligned}
(\underline{\alpha}\lambda\overline{\beta})
\cdot
(\underline{\gamma}\mu\overline{\delta})
\ &\overset{X_4}{=} \
\begin{tikzpicture}[scale=.5, baseline={(0,.68)}]
    \Bweight{0,0,1,1}
    \dCAP{0}{1}{1}
    \dCAP{2}{3}{1}
    \CUP{1}{2}{1}
    \CUP{0}{3}{.75}
    \FILL{4}

    \begin{scope}[yshift=3cm]
    \Bweight{1,1,0,0}
    \dCUP{0}{1}{1}
    \dCUP{2}{3}{1}
    \CAP{1}{2}{1}
    \CAP{0}{3}{.75}
    \FILL{4}
    \end{scope}
\end{tikzpicture}
\ \underset{\text{Merge}}{\overset{X_4}{=}} \
\begin{tikzpicture}[scale=.5, baseline={(0,.68)}]
    \Bweight{1,1,0,0}
    \UPLINE{0}{1.5}
    \UPLINE{1}{1.5}
    \dCAP{2}{3}{1}
    \CUP{1}{2}{1}
    \CUP{0}{3}{.75}
    \FILL{4}

    \begin{scope}[yshift=3cm]
    \Bweight{1,1,0,0}
    \DOWNLINE{0}{1.5}
    \DOWNLINE{1}{1.5}
    \dCUP{2}{3}{1}
    \CAP{1}{2}{1}
    \CAP{0}{3}{.75}
    \FILL{4}
    \end{scope}
\end{tikzpicture}
\ \underset{\text{(Nested) split}}{\overset{(X_3-X_4)X_4}{=}} \
\begin{tikzpicture}[scale=.5, baseline={(0,.68)}]
    \Bweight{1,1,0,0}
    \UPLINE{0}{1.5}
    \UPLINE{1}{1.5}
    \UPLINE{2}{1.5}
    \UPLINE{3}{1.5}
    \CUP{1}{2}{1}
    \CUP{0}{3}{.75}
    \FILL{4}

    \begin{scope}[yshift=3cm]
    \Bweight{1,1,0,0}
    \DOWNLINE{0}{1.5}
    \DOWNLINE{1}{1.5}
    \DOWNLINE{2}{1.5}
    \DOWNLINE{3}{1.5}
    \CAP{1}{2}{1}
    \CAP{0}{3}{.75}
    \FILL{4}
    \end{scope}
\end{tikzpicture}
\\[5ex]
&= \
\begin{tikzpicture}[scale=.5, baseline={(0,-0.1)}]
    \Bweight{1,1,0,0}
    \CAP{1}{2}{1}
    \CAP{0}{3}{.75}
    \CUP{1}{2}{1}
    \CUP{0}{3}{.75}
\end{tikzpicture}
\end{aligned}
\]
because $(X_3-X_4)X_4 = X_3X_4 - X_4^2 = X_3X_4 \in \mathcal{M}(\underline{\alpha}\overline{\delta})$.
\end{itemize}
\end{example}

\begin{definition}\label{def:4.24}
    We will denote $\D_n^0$ by $\D_n$, and from now on we will work exclusively with $\D_n$. This choice is motivated by the fact that $\D_n^0$ is the algebra most studied in the literature. It is shown in~\cite{EhrigStroppel2016a} that the category of finite dimensional $\D_n$-modules is equivalent to the category of perverse sheaves on the isotropic Grassmannians of type $D_n$ or $B_{n-1}$. We note however that $\D_n^1$ is also interesting, and all the results cited in the next sections remain valid for $\D_n^1$ as well.
\end{definition}

\section{Quiver presentation with relations of $\D_n$}\label{sec:5}

In this section we will explain how to present the algebra $\D_n$ as a quiver algebra with quadratic relations. These results are inspired by the work of~\cite{BarmeierWang2026}. In parallel of this work, Ben Mills proved a similar result in~\cite{Mills2026b} for the diagrammatic Hecke category.

\subsection{The quiver $Q_n$}\label{sec:5.1}

We know from~\cite{EhrigStroppel2016a,Seynnaeve2017} that the algebra $\D_n$ is Koszul, and hence quadratic, so we will define a quiver $Q_n$ with vertices given by degree $0$ diagrams, and edges given by degree $1$ diagrams, so that we have a surjective homomorphism of graded algebras
$$\omega : \Bbbk Q_n \twoheadrightarrow \D_n$$
which sends the vertices $e_\lambda$ to the degree $0$ arc diagrams $e_\lambda$ and sends edges to the corresponding degree $1$ arc diagrams. \\
Let's give a more precise description of this quiver $Q_n$.

\begin{definition}\label{def:5.1}
    Let $n \geq 1$, we define the simple undirected graph $\Gamma_n$ as follows. The vertices are given by the degree $0$ arc diagrams in $\D_n$. There is an edge between $e_\lambda$ and $e_\mu$ if the weight of one is obtained from the weight of the other by exchanging either a $\dn \cdots \up$ or a $\up \cdots \up$ pair lying in the same counterclockwise circle (decorated or not).
\end{definition}

\begin{lemma}\label{lem:5.2}
    The quiver $Q_n$ for $\D_n$ is obtained from $\Gamma_n$ by replacing each edge in it by $2$ arrows in opposite directions.
\end{lemma}

\begin{proof}
    Clearly, $Q_n$ and $\Gamma_n$ have the same vertices. \\
    Now, suppose that there is an edge in $\Gamma_n$ between $e_\lambda$ and $e_\mu$, and without loss of generality suppose that $\mu$ is obtained from $\lambda$ by exchanging a $\dn \cdots \up$ or $\up \cdots \up$ pair lying in the same counterclockwise circle of $e_\lambda$. Then observe that $\underline{\lambda}\mu$ is a degree $1$ cup diagram, so that $\underline{\lambda}\mu\overline{\mu}$ is a degree $1$ arc diagram from $e_\lambda$ to $e_\mu$. Similarly, $\underline{\mu}\mu\overline{\lambda}$ is a degree $1$ arc diagram from $e_\mu$ to $e_\lambda$. \\
    Conversely, if $\underline{\lambda}\mu\overline{\mu}$ is a degree $1$ arc diagram from $e_\lambda$ to $e_\mu$, then $\underline{\lambda}\mu$ is a degree $1$ cup diagram, and this exactly means that $\mu$ is obtained from $\lambda$ by exchanging a $\dn \cdots \up$ or $\up \cdots \up$ pair lying in the same counterclockwise circle of $e_\lambda$. Similarly if $\underline{\mu}\mu\overline{\lambda}$ is a degree $1$ arc diagram from $e_\mu$ to $e_\lambda$.
\end{proof}

\begin{definition}\label{def:5.3}
    Let $\lambda \in \Delta_n^0$ be a weight.  We define the floor of $\lambda$, denoted $\vert \lambda \vert_f$, to be the number of $\up$'s in $\lambda$, divided by $2$. We define the local height of $\lambda$, denoted $\vert \lambda \vert_h$, to be the sum of the number of $\up$'s on the left of each $\dn$ in $\lambda$. Finally, we define the height of $\lambda$, denoted $\lVert \lambda \rVert$, to be the sum of $\vert \lambda \vert_f$ and $\vert \lambda \vert_h$.
\end{definition}

\begin{example}\label{ex:5.4}
    We have that 
    $$\vert \up \up \dn \up \dn \dn \up \dn \vert_f = \frac{4}{2} = 2,$$
    and
    $$\vert \up \up \dn \up \dn \dn \up \dn \vert_h = 2+3+3+4 = 12,$$
    so that
    $$\lVert \up \up \dn \up \dn \dn \up \dn \rVert = 2+12 = 14.$$
\end{example}

\begin{lemma}\label{lem:5.5}
    Let $n \geq 1$, then the graph $\Gamma_n$ is bipartite.
\end{lemma}

\begin{proof}
    Suppose that there is an edge in $\Gamma_n$ between $e_\lambda$ and $e_\mu$, and without loss of generality suppose that $\mu$ is obtained from $\lambda$ by exchanging a $\dn \cdots \up$ or $\up \cdots \up$ pair lying in the same counterclockwise circle $C$ of $e_\lambda$. Denote by $k$ the number of circles inside $C$ in $e_\lambda$, so that there are $k \ \dn$'s and $k \ \up$'s inside the pair $\dn \cdots \up$ or $\up \cdots \up$ lying in $C$, denote also by $k'$ the number of $\up$'s on the left of $C$ in $e_\lambda$, and denote also by $k''$ the number of $\dn$'s on the right of $C$ in $e_\lambda$. \\
    If $\mu$ is obtained from $\lambda$ by exchanging a $\dn \cdots \up$ pair lying in $C$, then $\lVert \mu \rVert = \lVert \lambda \rVert + 2k + 1$ because $\vert \mu \vert_f = \vert \lambda \vert_f$ and $\vert \mu \vert_h = \vert \lambda \vert_h + 2k +1$. \\
    Now if $\mu$ is obtained from $\lambda$ by exchanging a $\up \cdots \up$ pair lying in $C$, then $\lVert \mu \rVert = \lVert \lambda \rVert + 2(k'-k'') - 1$ because $\vert \mu \vert_f = \vert \lambda \vert_f - 1$ and $\vert \mu \vert_h = \vert \lambda \vert_h + k - k + 2k' - 2k'' = \vert \lambda \vert_h + 2(k'-k'')$. \\
    Thus, in any case the heights of $\mu$ and $\lambda$ do not have the same parity, so $\Gamma_n$ is bipartite.
\end{proof}

\begin{definition}\label{def:5.6}
    Let $\lambda, \mu \in \Delta_n^0$ be weights. We will denote
    $$\vert \lambda \vert = (\vert \lambda \vert_f, \vert \lambda \vert_h)$$
    And we will write $\vert \lambda \vert < \vert \mu \vert$ if either $\vert \lambda \vert_f > \vert \mu \vert_f$ or $\vert \lambda \vert_f = \vert \mu \vert_f$ and $\vert \lambda \vert_h < \vert \mu \vert_h$. Observe that the order of the floors    are reversed.
\end{definition}

Definition~\ref{def:5.6} gives us a partial order on the weights in $\Delta_n^0$ that allows us to draw in a good-looking way the graphs $Q_n$ and $\Gamma_n$. We will draw the floors in the reverse order (floor $0$ at the top), and then use the local heights inside each floor.

\begin{example}\label{ex:5.7}
    We draw the quiver $Q_4$ with colors indicating from which circle each edge comes, and with dotted lines indicating the floors
    $$\begin{tikzpicture}[scale=.35]
    \CBweight{1,1,0,0}{7,7,9,9}
    \CdCAP{0}{1}{1}{7}
    \CdCUP{0}{1}{1}{7}
    \LINE{2}{1}
    \LINE{3}{1}
    \CFILL{4}{7,7,9,9}

    \begin{scope}[yshift=5cm]
        \Bweight{0,0,0,0}
        \LINE{0}{1}
        \LINE{1}{1}
        \LINE{2}{1}
        \LINE{3}{1}
        \FILL{4}
    \end{scope}

    \begin{scope}[yshift=-5cm]
        \CBweight{1,0,1,0}{9,0,0,9}
        \dLINE{0}{1}
        \LINE{3}{1}
        \CCUP{1}{2}{1}{0}
        \CCAP{1}{2}{1}{0}
        \CFILL{4}{9,0,0,9}
    \end{scope}

    \begin{scope}[yshift=-10cm]
        \begin{scope}[xshift=-4cm]
            \CBweight{0,1,1,0}{1,1,9,9}
            \dLINE{2}{1}
            \LINE{3}{1}
            \CCUP{0}{1}{1}{1}
            \CCAP{0}{1}{1}{1}
            \CFILL{4}{1,1,9,9}
        \end{scope}
        \begin{scope}[xshift=4cm]
            \CBweight{1,0,0,1}{9,9,2,2}
            \dLINE{0}{1}
            \LINE{1}{1}
            \CCUP{2}{3}{1}{2}
            \CCAP{2}{3}{1}{2}
            \CFILL{4}{9,9,2,2}
        \end{scope}
    \end{scope}

    \begin{scope}[yshift=-15cm]
        \CBweight{0,1,0,1}{3,3,4,4}
        \CCUP{0}{1}{1}{3}
        \CCAP{0}{1}{1}{3}
        \CCUP{2}{3}{1}{4}
        \CCAP{2}{3}{1}{4}
        \CFILL{4}{3,3,4,4}
    \end{scope}

    \begin{scope}[yshift=-20cm]
        \CBweight{0,0,1,1}{5,6,6,5}
        \CCUP{0}{3}{.75}{5}
        \CCAP{0}{3}{.75}{5}
        \CCUP{1}{2}{1}{6}
        \CCAP{1}{2}{1}{6}
        \CFILL{4}{5,6,6,5}
    \end{scope}

    \begin{scope}[yshift=-25cm]
        \CBweight{1,1,1,1}{8,8,10,10}
        \CdCUP{0}{1}{1}{8}
        \CdCAP{0}{1}{1}{8}
        \CdCUP{2}{3}{1}{10}
        \CdCAP{2}{3}{1}{10}
        \CFILL{4}{8,8,10,10}
    \end{scope}
    
    \draw[->, very thick, , draw=\colorFor{0}] (1,-1.75) -- (1,-3.25);
    \draw[->, very thick, , draw=\colorFor{0}] (2,-3.25) -- (2,-1.75);
    \draw[->, very thick, , draw=\colorFor{1}] (-1,-6.75) -- (-2,-8.25);
    \draw[->, very thick, , draw=\colorFor{1}] (-1,-8.25) -- (0,-6.75);
    \draw[->, very thick, , draw=\colorFor{2}] (3,-6.75) -- (4,-8.25);
    \draw[->, very thick, , draw=\colorFor{2}] (5,-8.25) -- (4,-6.75);
    \draw[->, very thick, , draw=\colorFor{4}] (-2,-11.75) -- (-1,-13.25);
    \draw[->, very thick, , draw=\colorFor{4}] (0,-13.25) -- (-1,-11.75);
    \draw[->, very thick, , draw=\colorFor{3}] (4,-11.75) -- (3,-13.25);
    \draw[->, very thick, , draw=\colorFor{3}] (4,-13.25) -- (5,-11.75);
    \draw[->, very thick, , draw=\colorFor{6}] (1,-16.75) -- (1,-18.25);
    \draw[->, very thick, , draw=\colorFor{6}] (2,-18.25) -- (2,-16.75);
    \draw[->, very thick, , draw=\colorFor{5}] (-7,-12.5) arc[start angle=180, end angle=90, x radius=6, y radius=7];
    \draw[very thick, , draw=\colorFor{5}] (-7,-12.5) arc[start angle=180, end angle=270, x radius=6, y radius=7];
    \draw[very thick, , draw=\colorFor{5}] (-8,-12.5) arc[start angle=180, end angle=90, x radius=7, y radius=8];
    \draw[->, very thick, , draw=\colorFor{5}] (-8,-12.5) arc[start angle=180, end angle=270, x radius=7, y radius=8];

    \draw[<-, very thick, , draw=\colorFor{7}] (1,1.75) -- (1,3.25);
    \draw[<-, very thick, , draw=\colorFor{7}] (2,3.25) -- (2,1.75);
    \draw[->, very thick, , draw=\colorFor{8}] (1,-21.75) -- (1,-23.25);
    \draw[->, very thick, , draw=\colorFor{8}] (2,-23.25) -- (2,-21.75);
    \draw[very thick, , draw=\colorFor{10}] (10,-12.5) arc[start angle=0, end angle=90, x radius=6, y radius=12];
    \draw[->,very thick, , draw=\colorFor{10}] (10,-12.5) arc[start angle=0, end angle=-90, x radius=6, y radius=12];
    \draw[->,very thick, , draw=\colorFor{10}] (11,-12.5) arc[start angle=0, end angle=90, x radius=7, y radius=13];
    \draw[very thick, , draw=\colorFor{10}] (11,-12.5) arc[start angle=0, end angle=-90, x radius=7, y radius=13];

    \draw[thick, dotted] (-10,2.5) -- (13,2.5);
    \draw[thick, dotted] (-10,-22.5) -- (13,-22.5);
\end{tikzpicture}$$
And if we zoom in on the top two diagrams in this graph, and we denote the vertices and the edges by
$$\begin{tikzpicture}[scale=.8]
    \Bweight{0,0,0,0}
    \LINE{0}{1}
    \LINE{1}{1}
    \LINE{2}{1}
    \LINE{3}{1}
    \FILL{4}

    \begin{scope}[yshift=-5cm]
        \Bweight{1,1,0,0}
        \LINE{2}{1}
        \LINE{3}{1}
        \bigdCUP{0}{1}{1}
        \bigdCAP{0}{1}{1}
        \FILL{4}
    \end{scope}

    \draw[->, very thick] (1,-1.75) -- (1,-3.25);
    \draw[->, very thick] (2,-3.25) -- (2,-1.75);

    \node at (0.3,-2.5) {$y$};
    \node at (2.7,-2.5) {$x$};

    \node at (-2,-0.1) {$e_\mu =$};
    \node at (-2,-5.1) {$e_\lambda =$};
\end{tikzpicture}$$
then, if we follow what we explained in the proof of Lemma~\ref{lem:5.5}, the edges are given diagramatically by
$$x = \underline{\lambda}\mu\overline{\mu} = \begin{tikzpicture}[scale=.8, baseline = {(0,-.1)}]
    \Bweight{0,0,0,0}
    \bigdCUP{0}{1}{1}
    \LINE{2}{1}
    \UPLINE{0}{1}
    \UPLINE{1}{1}
    \LINE{3}{1}
    \FILL{4}
    \end{tikzpicture} \quad \text{ and } \quad y = \underline{\mu}\mu\overline{\lambda} = \begin{tikzpicture}[scale=.8, baseline = {(0,-.1)}]
    \Bweight{0,0,0,0}
    \bigdCAP{0}{1}{1}
    \LINE{2}{1}
    \DOWNLINE{0}{1}
    \DOWNLINE{1}{1}
    \LINE{3}{1}
    \FILL{4}
    \end{tikzpicture} $$
\end{example}

\subsection{Relations in the quiver}\label{sec:5.2}

We will now compute the relations in the quiver algebra. Since we know that $\D_n$ is a Koszul algebra and hence quadratic, we need to understand the length two paths in $Q_n$.

\begin{remark}\label{rmk:5.8}
    From now on, each time we will draw an edge in $\Gamma_n$ or an arrow in $Q_n$, we will draw it in black if it correponds to a type A move (coming from a $\dn \cdots \up$ pair), and in blue if it corresponds to a type D move (coming from a $\up \cdots \up$ pair). In addition, we will sometimes draw dotted lines to indicate that the vertices lie on different floors.
\end{remark}

\begin{lemma}\label{lem:5.9}
    Let $n \geq 1$, and let $\lambda, \mu \in \Delta_n^0$ be distinct weights such that $\vert \lambda \vert \leq \vert \mu \vert$. Then
    \begin{enumerate}
        \item[\textit{(i)}] There is at most one length $2$ path in $\Gamma_n$ between $\lambda$ and $\mu$ passing through $\nu$ such that $\vert \nu \vert < \vert \lambda \vert$.
        \item[\textit{(ii)}] There is at most one length $2$ path in $\Gamma_n$ between $\lambda$ and $\mu$ passing through $\nu$ such that $\vert \nu \vert > \vert \mu \vert$.
    \end{enumerate}
\end{lemma}

\begin{proof}
    Let's prove $\textit{(i)}$. Suppose by contradiction that there exist two distinct weights $\nu$ and $\nu'$, both connected to $\lambda$ and $\mu$ by an edge in $\Gamma_m$, and with $\vert \nu' \vert \leq \vert \nu \vert < \vert \lambda \vert$. We will separate two cases.

    \textbf{If $\lambda$ and $\mu$ are in the same floor :} then we will have several subcases. We denote by $j$ the floor of $\lambda$. Either $\nu$ and $\nu'$ lie on the same floor as $\lambda$ and $\mu$, then all the edges in
    \[\begin{tikzcd}
	\lambda &&&&&& \mu \\
	&&& \nu \\
	&&& {\nu'}
	\arrow[no head,very thick, from=2-4, to=1-1]
	\arrow[no head,very thick, from=2-4, to=1-7]
	\arrow[no head,very thick, from=3-4, to=1-1]
	\arrow[no head,very thick, from=3-4, to=1-7]
\end{tikzcd}\]
    correspond to type A moves. Since $\lambda$ and $\mu$ are distinct, they are obtained from $\nu$ and $\nu'$ by exchanging in each two $\dn \cdots \up$ pairs lying in distinct circles of $e_\nu$ and $e_{\nu'}$. Then these circles must be nested in one of $e_\nu, e_{\nu'}$, and non-nested in the other. Without loss of generality, these circles are nested in $e_\nu$, and non-nested in $e_{\nu'}$
    $$\begin{tikzpicture}[scale=.8, baseline={(0,-0.15)}]
    \weight{0,0,1,1}
    \CAP{1}{2}{1}
    \CAP{0}{3}{.75}
    \CUP{1}{2}{1}
    \CUP{0}{3}{.75}
    \FILL{4}
    \node at (-0.45,0) {$\cdots$};   
    \node at (0.55,0) {$\cdots$};   
    \node at (1.55,0) {$\cdots$};
    \node at (2.55,0) {$\cdots$};   
    \node at (3.55,0) {$\cdots$};
\end{tikzpicture}$$
and
$$\begin{tikzpicture}[scale=.8, baseline={(0,-0.18)}]
    \weight{0,1,0,1}
    \CAP{0}{1}{1}
    \CAP{2}{3}{1}
    \CUP{0}{1}{1}
    \CUP{2}{3}{1}
    \FILL{4}
    \node at (-0.45,0) {$\cdots$};   
    \node at (0.55,0) {$\cdots$};   
    \node at (1.55,0) {$\cdots$};
    \node at (2.55,0) {$\cdots$};   
    \node at (3.55,0) {$\cdots$};
\end{tikzpicture}$$
    But then observe that there would be an edge between $\nu$ and $\nu'$ in $\Gamma_n$, which contradicts the fact that $\Gamma_n$ is bipartite. \\
    Now if both $\nu$ and $\nu'$ lie on the floor $j+1$ (drawn under floor $j$ in $\Gamma_n$), then all the edges in
    \[\begin{tikzcd}
	& \lambda &&&&&& \mu & \\
	{} &&&&&&&& {} \\
	&&&& \nu \\
	&&&& {\nu'}
	\arrow[dashed, no head, from=2-1, to=2-9]
	\arrow[color=\colorFor{0}, no head,very thick, from=3-5, to=1-2]
	\arrow[color=\colorFor{0}, no head,very thick, from=3-5, to=1-8]
	\arrow[color=\colorFor{0}, no head,very thick, from=4-5, to=1-2]
	\arrow[color=\colorFor{0}, no head,very thick, from=4-5, to=1-8]
\end{tikzcd}\]
    correspond to type D moves. But then $\lambda$ and $\mu$ are both obtained from $\nu$ and $\nu'$ by exchanging two $\up \cdots \up$ pairs lying in distinct circles of $e_\nu$ and also of $e_{\nu'}$. There is only one possibility, so the diagrams $e_\nu$ and $e_{\nu'}$ look like
     $$\begin{tikzpicture}[scale=.8, baseline={(0,-0.15)}]
    \weight{1,1,1,1}
    \bigdCAP{0}{1}{1}
    \bigdCAP{2}{3}{1}
    \bigdCUP{0}{1}{1}
    \bigdCUP{2}{3}{1}
    \FILL{4}
    \node at (-0.45,0) {$\cdots$};   
    \node at (0.55,0) {$\cdots$};   
    \node at (1.55,0) {$\cdots$};
    \node at (2.55,0) {$\cdots$};   
    \node at (3.55,0) {$\cdots$};
\end{tikzpicture}$$
    because
    $$\begin{tikzpicture}[scale=.8, baseline={(0,-0.15)}]
    \weight{1,1,1,1}
    \bigdCAP{1}{2}{1}
    \bigdCAP{0}{3}{.75}
    \bigdCUP{1}{2}{1}
    \bigdCUP{0}{3}{.75}
    \FILL{4}
    \node at (-0.45,0) {$\cdots$};   
    \node at (0.55,0) {$\cdots$};   
    \node at (1.55,0) {$\cdots$};
    \node at (2.55,0) {$\cdots$};   
    \node at (3.55,0) {$\cdots$};
\end{tikzpicture}$$
    is not admissible (and corresponds to the same weight). Hence $\nu = \nu'$ and we have a contradiction. \\
    Finally, if $\nu$ lies on the same floor as $\lambda$ and $\mu$, and $\nu'$ lies on the floor $j+1$, then the paths would look like
    \[\begin{tikzcd}
	& \lambda &&&&&& \mu & \\
	&&&& \nu \\
	{} &&&&&&&& {} \\
	&&&& {\nu'}
	\arrow[no head,very thick, from=2-5, to=1-2]
	\arrow[no head,very thick, from=2-5, to=1-8]
	\arrow[dashed, no head, from=3-1, to=3-9]
	\arrow[color=\colorFor{0},no head,very thick, from=4-5, to=1-2]
	\arrow[color=\colorFor{0},no head,very thick, from=4-5, to=1-8]
\end{tikzcd}\]
    As in the previous subcase, we know that $e_{\nu'}$ will look like
    $$\begin{tikzpicture}[scale=.8, baseline={(0,-0.15)}]
    \weight{1,1,1,1}
    \bigdCAP{0}{1}{1}
    \bigdCAP{2}{3}{1}
    \bigdCUP{0}{1}{1}
    \bigdCUP{2}{3}{1}
    \FILL{4}
    \node at (-0.45,0) {$\cdots$};   
    \node at (0.55,0) {$\cdots$};   
    \node at (1.55,0) {$\cdots$};
    \node at (2.55,0) {$\cdots$};   
    \node at (3.55,0) {$\cdots$};
\end{tikzpicture}$$
    and since $\vert\lambda\vert \leq \vert\mu\vert$, we will have
    $$\begin{tikzpicture}[scale=.8, baseline={(0,-0.15)}]
    \weight{0,0,1,1}
    \CAP{1}{2}{1}
    \CAP{0}{3}{.75}
    \CUP{1}{2}{1}
    \CUP{0}{3}{.75}
    \FILL{4}
    \node at (-0.45,0) {$\cdots$};   
    \node at (0.55,0) {$\cdots$};   
    \node at (1.55,0) {$\cdots$};
    \node at (2.55,0) {$\cdots$};   
    \node at (3.55,0) {$\cdots$};

    \node at (1.5,-2) {$e_\lambda$};
\end{tikzpicture} \qquad \qquad \begin{tikzpicture}[scale=.8, baseline={(0,-0.15)}]
    \weight{1,1,0,0}
    \bigdCAP{0}{1}{1}
    \bigdCUP{0}{1}{1}
    \RLINE{2}{1}
    \RLINE{3}{1}
    \FILL{4}
    \node at (-0.45,0) {$\cdots$};   
    \node at (0.55,0) {$\cdots$};   
    \node at (1.55,0) {$\cdots$};
    \node at (2.55,0) {$\cdots$};   
    \node at (3.55,0) {$\cdots$};

    \node at (1.5,-2) {$e_\lambda$};
\end{tikzpicture}$$
    where the bent arcs indicate that the arcs are either a straight line, or belong to a circle in the bending direction.
    Observe that $\nu$ and $\lambda$ are different only on $2$ of the $4$ positions indicated on the diagrams, because they are connected in $\Gamma_n$ and $\nu$ is also connected to $\mu$ so it has to be among these $4$ positions. But $\nu$ and $\lambda$ lie on the same floor, and $\dn \dn \up \up$ is the lowest weight of the floor if we zoom on the $4$ interesting positions. This weight correponds to $\lambda$, but $\vert \nu \vert < \vert \lambda \vert$ so this leads to a contradiction.
    
    \textbf{If $\lambda$ and $\mu$ are in different floors :} then necessarly, the paths would look like
    \[\begin{tikzcd}
	&&&&&&& \mu & \\
	{} &&&&&&&& {} \\
	& \lambda \\
	&&&& \nu \\
	&&&& {\nu'}
	\arrow[dashed, no head, from=2-1, to=2-9]
	\arrow[color=\colorFor{0},no head,very thick, from=4-5, to=1-8]
	\arrow[no head,very thick, from=4-5, to=3-2]
	\arrow[color=\colorFor{0},no head,very thick, from=5-5, to=1-8]
	\arrow[no head,very thick, from=5-5, to=3-2]
\end{tikzcd}\]
    so that $\nu$ and $\nu'$ lie on the same floor as $\lambda$. We know that $\lambda$ is obtained from $\nu$ (respectively $\nu'$) by exchanging a $\dn \cdots \up$ pair lying in a circle of $e_\nu$ (respectively $e_{\nu'})$. Similarly, $\mu$ is obtained from $\nu$ (respectively $\nu'$) by exchanging a $\up \cdots \up$ pair lying in a circle of $e_\nu$ (respectively $e_{\nu'})$. This leads to the three following possibilities for $e_\nu$ and $e_{\nu'}$:
    $$
.$$
    But there are edges between them in $\Gamma_n$ in the following way:
    $$%
$$
    where all the edges come from the non-dotted circle. Since the graph $\Gamma_n$ is bipartite, it's not possible that $\nu$ and $\nu'$ are connected, so we can suppose that
    $$e_\nu \ = \ %
 $$
    But when we exchange the $\dn \cdots \up$ pair lying in the non-dotted counterclockwise circles in each diagram, this does not give the same weight, which should be $\lambda$, so we have a contradiction. \\
    We found a contradiction in each possible subcase, so this proves $\textit{(i)}$. The proof for $\textit{(ii)}$ is similar. 
\end{proof}

Lemma~\ref{lem:5.9} will help us to understand the squares in $\Gamma_n$, so every length $2$ path with at least one other parallel length $2$ path.

\begin{lemma}\label{lem:5.10}
    Any square in $\Gamma_n$ is of one of the following forms
    $$\quad%
  $$
The dots between the arcs indicates that they can be other up and down arrows in between the indicated weights, and these are the same in every diagram within a given square. The bent arcs indicate that the arcs either become a straight line, or belong to a circle in the bent direction. And the dotted lines indicate that the arcs either become a dotted straight line, or belong to a circle in any direction (dotted or not).
\end{lemma}

\begin{proof}
    Observe first that any square in $\Gamma_n$ has a lowest weight $\lambda$ with two ascending arrows starting from it in the square, because otherwise it would create a cycle that would contradicts the order from~\ref{def:5.6}. We will denote by $\nu$ and $\nu'$ the two weights in the square obtained from $\lambda$, and we can suppose without loss of generality that $\vert \nu \vert \geq \vert \nu' \vert$. We will denote by $\mu$ the fourth weight completing the square. \\
    Observe that $\vert \mu \vert \neq \vert \nu \vert, \vert \nu' \vert$ because $\Gamma_n$ is bipartite, by parity of the heights. Observe also that thanks to Lemma~\ref{lem:5.9}, we cannot have that $\vert \mu \vert < \vert \nu' \vert \leq \vert \nu \vert$. Thus a square in $\Gamma_n$ can only have one of the following two shapes
    $$\begin{tikzcd}[column sep=tiny]
	& \mu & \\
	\nu && {\nu'} \\
	& \lambda \\
	& {\vert \nu' \vert \leq \vert \nu \vert < \vert \mu \vert}
	\arrow[no head,very thick, from=2-1, to=1-2]
	\arrow[no head,very thick, from=2-1, to=3-2]
	\arrow[no head,very thick, from=2-3, to=1-2]
	\arrow[no head,very thick, from=3-2, to=2-3]
    \end{tikzcd} \qquad \text{ or } \qquad  \begin{tikzcd}[column sep=tiny, row sep = scriptsize]
	\nu & \\
	& \mu \\
	& {\nu'} \\
	\lambda \\
	{\vert \nu' \vert < \vert \mu \vert < \vert \nu \vert}
	\arrow[no head,very thick, from=2-2, to=1-1]
	\arrow[no head,very thick, from=3-2, to=2-2]
	\arrow[curve={height=-24pt}, no head,very thick, from=4-1, to=1-1]
	\arrow[no head,very thick, from=4-1, to=3-2]
    \end{tikzcd} $$
    In any square, the two edges connected to $\lambda$ correpond to exchanging two $\dn \cdots \up$ or $\up \cdots \up$ pairs lying in distinct circles of $e_\lambda$. If these circles are non-nested, there are $4$ different possibilites given by
    $$\begin{tikzpicture}[scale=.4, baseline={(0,-0.1)}]

        \begin{scope}[xshift=-4cm]
            \Bweight{0,1,1,0}
            \CUP{0}{1}{1}
            \CAP{0}{1}{1}
            \pdLINE{2}{1}
            \RLINE{3}{1}
            \FILL{2}
            \node at (-0.80,0) {$\cdots$};   
            \node at (0.52,0) {$\cdot$};   
            \node at (1.52,0) {$\cdot$};
            \node at (2.52,0) {$\cdot$};   
            \node at (4,0) {$\cdots$};
        \end{scope}
        \begin{scope}[xshift=4cm]
            \Bweight{1,0,0,1}
            \CUP{2}{3}{1}
            \CAP{2}{3}{1}
            \pdLINE{0}{1}
            \RLINE{1}{1}
            \PFILL{2}
            \PFILL{3}
            \node at (-0.80,0) {$\cdots$};   
            \node at (0.52,0) {$\cdot$};   
            \node at (1.52,0) {$\cdot$};
            \node at (2.52,0) {$\cdot$};   
            \node at (4,0) {$\cdots$};
        \end{scope}

    \begin{scope}[yshift=-4cm]
        \Bweight{0,1,0,1}
        \CUP{0}{1}{1}
        \CAP{0}{1}{1}
        \CUP{2}{3}{1}
        \CAP{2}{3}{1}
        \FILL{4}
        \node at (-0.80,0) {$\cdots$};   
        \node at (0.52,0) {$\cdot$};   
        \node at (1.52,0) {$\cdot$};
        \node at (2.52,0) {$\cdot$};   
        \node at (4,0) {$\cdots$};
    \end{scope}

    \draw[strand] (-1,-3) -- (-2,-1.5);
    \draw[strand] (4,-3) -- (5,-1.5);

\end{tikzpicture} \qquad \begin{tikzpicture}[scale=.4, baseline={(0,-0.1)}]

        \begin{scope}[xshift=-4cm]
            \Bweight{1,1,1,0}
            \dCUP{0}{1}{1}
            \dCAP{0}{1}{1}
            \pdLINE{2}{1}
            \RLINE{3}{1}
            \FILL{2}
            \node at (-0.80,0) {$\cdots$};   
            \node at (0.52,0) {$\cdot$};   
            \node at (1.52,0) {$\cdot$};
            \node at (2.52,0) {$\cdot$};   
            \node at (4,0) {$\cdots$};
        \end{scope}
        \begin{scope}[xshift=4cm]
            \Bweight{0,0,0,1}
            \CUP{2}{3}{1}
            \CAP{2}{3}{1}
            \RLINE{0}{1}
            \RLINE{1}{1}
            \PFILL{2}
            \PFILL{3}
            \node at (-0.80,0) {$\cdots$};   
            \node at (0.52,0) {$\cdot$};   
            \node at (1.52,0) {$\cdot$};
            \node at (2.52,0) {$\cdot$};   
            \node at (4,0) {$\cdots$};
        \end{scope}

    \begin{scope}[yshift=-4cm]
        \Bweight{1,1,0,1}
        \dCUP{0}{1}{1}
        \dCAP{0}{1}{1}
        \CUP{2}{3}{1}
        \CAP{2}{3}{1}
        \FILL{4}
        \node at (-0.80,0) {$\cdots$};   
        \node at (0.52,0) {$\cdot$};   
        \node at (1.52,0) {$\cdot$};
        \node at (2.52,0) {$\cdot$};   
        \node at (4,0) {$\cdots$};
    \end{scope}

    \draw[strand] (-1,-3) -- (-2,-1.5);
    \draw[colored strand = {0}] (4,-3) -- (5,-1.5);

\end{tikzpicture}$$
\vspace{.5cm}
$$\begin{tikzpicture}[scale=.4, baseline={(0,-0.1)}]

        \begin{scope}[xshift=-4cm]
            \Bweight{0,1,0,0}
            \CUP{0}{1}{1}
            \CAP{0}{1}{1}
            \RLINE{2}{1}
            \RLINE{3}{1}
            \FILL{2}
            \node at (-0.80,0) {$\cdots$};   
            \node at (0.52,0) {$\cdot$};   
            \node at (1.52,0) {$\cdot$};
            \node at (2.52,0) {$\cdot$};   
            \node at (4,0) {$\cdots$};
        \end{scope}
        \begin{scope}[xshift=4cm]
            \Bweight{1,0,1,1}
            \CUP{1}{2}{1}
            \CAP{1}{2}{1}
            \dCAP{0}{3}{.6}
            \dCUP{0}{3}{.6}
            \FILL{4}
            \node at (-0.80,0) {$\cdots$};   
            \node at (0.52,0) {$\cdot$};   
            \node at (1.52,0) {$\cdot$};
            \node at (2.52,0) {$\cdot$};   
            \node at (4,0) {$\cdots$};
        \end{scope}

    \begin{scope}[yshift=-4cm]
        \Bweight{0,1,1,1}
        \CUP{0}{1}{1}
        \CAP{0}{1}{1}
        \dCUP{2}{3}{1}
        \dCAP{2}{3}{1}
        \FILL{4}
        \node at (-0.80,0) {$\cdots$};   
        \node at (0.52,0) {$\cdot$};   
        \node at (1.52,0) {$\cdot$};
        \node at (2.52,0) {$\cdot$};   
        \node at (4,0) {$\cdots$};
    \end{scope}

    \draw[colored strand = {0}] (-1,-3) -- (-2,-1.5);
    \draw[strand] (4,-3) -- (5,-1.5);

\end{tikzpicture} \qquad \begin{tikzpicture}[scale=.4, baseline={(0,-0.1)}]

        \begin{scope}[xshift=-4cm]
            \Bweight{1,1,0,0}
            \dCUP{0}{1}{1}
            \dCAP{0}{1}{1}
            \RLINE{2}{1}
            \RLINE{3}{1}
            \FILL{2}
            \node at (-0.80,0) {$\cdots$};   
            \node at (0.52,0) {$\cdot$};   
            \node at (1.52,0) {$\cdot$};
            \node at (2.52,0) {$\cdot$};   
            \node at (4,0) {$\cdots$};
        \end{scope}
        \begin{scope}[xshift=4cm]
            \Bweight{0,0,1,1}
            \CUP{1}{2}{1}
            \CAP{1}{2}{1}
            \CAP{0}{3}{.6}
            \CUP{0}{3}{.6}
            \FILL{4}
            \node at (-0.80,0) {$\cdots$};   
            \node at (0.52,0) {$\cdot$};   
            \node at (1.52,0) {$\cdot$};
            \node at (2.52,0) {$\cdot$};   
            \node at (4,0) {$\cdots$};
        \end{scope}

    \begin{scope}[yshift=-4cm]
        \Bweight{1,1,1,1}
        \dCUP{0}{1}{1}
        \dCAP{0}{1}{1}
        \dCUP{2}{3}{1}
        \dCAP{2}{3}{1}
        \FILL{4}
        \node at (-0.80,0) {$\cdots$};   
        \node at (0.52,0) {$\cdot$};   
        \node at (1.52,0) {$\cdot$};
        \node at (2.52,0) {$\cdot$};   
        \node at (4,0) {$\cdots$};
    \end{scope}

    \draw[colored strand = {0}] (-1,-3) -- (-2,-1.5);
    \draw[colored strand = {0}] (4,-3) -- (5,-1.5);

\end{tikzpicture}$$
    If we want now to add $\mu$ to complete one of these figures into a square, we have to change two pairs $\dn \cdots \up$ or $\up \cdots \up$  (ascending arrow) or $\up \cdots \dn$ or $\dn \cdots \dn$ (descending arow) from $\nu$ to obtain $\nu'$. But $\nu$ and $\nu'$ differ only on the $4$ positions drawn in the diagrams, and since exchanging a pair of arrows changes $2$ positions, the two changes from $\nu$ to $\nu'$ have to be done at the $4$ interesting positions, which gives us few possibilities. Observe that in the two different shapes
    $$\begin{tikzcd}
	& \mu & \\
	\nu && {\nu'} \\
	& \lambda 
	\arrow[no head,very thick, from=2-1, to=1-2]
	\arrow[no head,very thick, from=2-1, to=3-2]
	\arrow[no head,very thick, from=2-3, to=1-2]
	\arrow[no head,very thick, from=3-2, to=2-3]
    \end{tikzcd} \qquad \text{ or } \qquad  \begin{tikzcd}[row sep = scriptsize]
	\nu & \\
	& \mu \\
	& {\nu'} \\
	\lambda
	\arrow[no head,very thick, from=2-2, to=1-1]
	\arrow[no head,very thick, from=3-2, to=2-2]
	\arrow[curve={height=-24pt}, no head,very thick, from=4-1, to=1-1]
	\arrow[no head,very thick, from=4-1, to=3-2]
    \end{tikzcd} $$
    there is an ascending arrow starting at $\nu'$ giving rise to $\mu$. So we can create all the possible $\mu$'s thanks to the circles appearing in $e_{\nu'}$ at the $4$ drawn positions (at most $2$), and then complete the picture into a square with an edge between $\lambda$ and $\mu$ (either ascending or descending depending on the shape) by just checking if the circles appearing in $e_\lambda$ or $e_\mu$ give rise to the other one. For example, for the first picture, it is the only one in which we cannot say which diagram is $e_\nu$ and the other is $e_{\nu'}$, so by symmetry we can suppose without loss of generality that $e_\nu$ is on the left, and then $e_{\nu'}$ on the right. There is only one circle in the $4$ interesting positions of $e_{\nu'}$ so the only possibility for $\mu$ is
    \vspace{.5cm}
    $$
$$
    Then there is only one circle in $e_\nu$ and none in $e_\mu$, and it gives rise to
    $$%
$$
    We deduce that the first figure
    $$%
$$
    gives rise uniquely to the square $\textit{(a)}$. \\
    We can give another example for the second picture
    $$%
$$
    This time we know that the diagram on the right is $e_\nu$ because it has been obtained from $e_\lambda$ by a type D move, and $e_{\nu'}$ by a type A move, so $\vert \nu \vert > \vert \nu' \vert$. Again there is only one circle in the $4$ interesting positions of $e_{\nu'}$ so the only possibility for $\mu$ is
    $$%
$$
    Then there is only one circle in $e_\nu$ and none in $e_\mu$, so it gives rise to
    $$%
$$
    We deduce that the second figure
    $$%
$$
    give rise uniquely to the square $\textit{(b)}$. \\
    We apply the exact same process for the third and fourth pictures
    $$%
$$
    but these time there are $2$ circles in $e_{\nu'}$ so it gives two different possible $\mu$'s, then we check all possibilites and we deduce that these figures give rise uniquely to squares $\textit{(c)}$ and $\textit{(d)}$ respectively. \\
    Now if the circles are nested, without any other circle in-between, it gives us two different cases
    $$%
$$
    Similarly, we can easily check that the first figure gives rise uniquely to squares $\textit{(e)}$ and $\textit{(f)}$, and the second figure gives rise uniquely to squares $\textit{(g)}$ and $\textit{(h)}$. \\
    Finally, if the circles are nested, with a circle in-between, it gives us two different cases
    $$%
$$
    And again, we can easily check that they give rise uniquely to squares $\textit{(i)}$ and $\textit{(j)}$ respectively.
\end{proof}

It remains to understand the length $2$ paths in $\Gamma_n$ that do not appear in any square. These type of paths are classified by the following lemma.
\newpage

\begin{lemma}\label{lem:5.11}
     Any length $2$ path in $\Gamma_n$ without a parallel length $2$ path is of one of the following forms
    $$%
$$
    The dots between the arcs indicates that they can be other up and down arrows in between the indicated weights, and these are the same in every diagram within a given path. The (dotted) bent arcs indicate that the arcs either become a (dotted) straight line, or belong to a (dotted or not) circle in the bent direction. And the dotted lines indicate that the arcs either become a dotted straight line, or belong to a circle in any direction (dotted or not).
\end{lemma}

\begin{proof}
    Firstly, by looking at every possible square from Lemma~\ref{lem:5.10}, we can see that these paths $\textit{(a)}$-$\textit{(i)}$ do not appear in any square. \\
    Now we need to prove that there is no other path of length $2$ without any parallel path of length $2$. We consider without loss of generality that such a path connect $\lambda$ and $\mu$ through $\nu$, with $\vert \lambda \vert \leq \vert \mu \vert$. We have that $\vert \nu \vert \neq \vert \lambda \vert, \vert \mu \vert$ because $\Gamma_n$ is bipartite, by parity of the heights. Observe that it is not possible that $\vert \nu \vert < \vert \lambda \vert \leq \vert \mu \vert$ because we proved in Lemma~\ref{lem:5.9} that any such path gives rise to at least one square. \\
    So any path we are looking for is of one of the following shape
    $$\begin{tikzcd}
	\mu \\
	\nu \\
	\lambda \\
	{\vert \lambda \vert < \vert \nu \vert < \vert \mu \vert}
	\arrow[no head, very thick, from=2-1, to=1-1]
	\arrow[no head, very thick, from=3-1, to=2-1]
    \end{tikzcd} \qquad \text{ or } \qquad \begin{tikzcd}[column sep=tiny]
	& \nu & \\
	\lambda && \mu \\
	& {\vert \lambda \vert \leq \vert \mu \vert < \vert \nu \vert}
	\arrow[no head, very thick, from=2-1, to=1-2]
	\arrow[no head, very thick, from=2-3, to=1-2]
    \end{tikzcd}$$
    Suppose that $\vert \lambda \vert < \vert \nu \vert < \vert \mu \vert$, then $\lambda$ and $\mu$ can differ in at most $4$ positions. \\
    In the first case, $\lambda$ and $\mu$ are distinct in only $2$ positions. This means that the $\dn \cdot \up$ or $\up \cdot \up$ pair used in $\nu$ to obtain $\mu$, share a position with the $\dn \cdot \up$ or $\up \cdot \up$ pair used in $\lambda$ to obtain $\nu$. Clearly, this leads to paths $\textit{(a)}$, $\textit{(b)}$, $\textit{(c)}$ and $\textit{(d)}$ only, depending on whether the sharing position is on the right or on the left of the circle used in $e_\lambda$, and if there is a type D move or not in the path. \\
    In the second case, $\lambda$ and $\mu$ are different in $4$ positions, and the positions of the pairs $\dn \cdot \up$ or $\up \cdot \up$ used in $\lambda$ and $\nu$ respectively are non-nested. Then such a path look like
    $$\begin{tikzpicture}[scale=.5, baseline={(0,-0.1)}]
    \Bweight{1,0,1,0}
    \pdLINE{0}{1}
    \RLINE{1}{1}
    \pdLINE{2}{1}
    \RLINE{3}{1}
    \node at (-0.80,0) {$\cdots$};   
    \node at (0.52,0) {$\cdot$};   
    \node at (1.52,0) {$\cdot$};
    \node at (2.52,0) {$\cdot$};
    \node at (4.00,0) {$\cdots$}; 

    \begin{scope}[yshift=-5cm]
        \Bweight{1,0,0,1}
    \CAP{2}{3}{1}
    \CUP{2}{3}{1}
    \pdLINE{0}{1}
    \RLINE{1}{1}
    \PFILL{2}
    \PFILL{3}
    \node at (-0.80,0) {$\cdots$};   
    \node at (0.52,0) {$\cdot$};   
    \node at (1.52,0) {$\cdot$};
    \node at (2.52,0) {$\cdot$};
    \node at (4.00,0) {$\cdots$}; 
    \end{scope}

    \begin{scope}[yshift=-10cm]
        \Bweight{0,1,0,1}
    \CAP{0}{1}{1}
    \CUP{0}{1}{1}
    \CAP{2}{3}{1}
    \CUP{2}{3}{1}
    \FILL{4}
    \node at (-0.80,0) {$\cdots$};   
    \node at (0.52,0) {$\cdot$};   
    \node at (1.52,0) {$\cdot$};
    \node at (2.52,0) {$\cdot$};
    \node at (4.00,0) {$\cdots$}; 
    \end{scope}

    \draw[strand] (1.5,-1.75) -- (1.5,-3.25);
    \draw[strand] (1.5,-6.75) -- (1.5,-8.25);

    \begin{scope}[xshift=10cm]
    \Bweight{1,0,1,0}
    \pdLINE{0}{1}
    \RLINE{1}{1}
    \pdLINE{2}{1}
    \RLINE{3}{1}
    \node at (-0.80,0) {$\cdots$};   
    \node at (0.52,0) {$\cdot$};   
    \node at (1.52,0) {$\cdot$};
    \node at (2.52,0) {$\cdot$};
    \node at (4.00,0) {$\cdots$}; 

    \begin{scope}[yshift=-5cm]
        \Bweight{0,1,1,0}
    \CAP{0}{1}{1}
    \CUP{0}{1}{1}
    \pdLINE{2}{1}
    \RLINE{3}{1}
    \FILL{2}
    \node at (-0.80,0) {$\cdots$};   
    \node at (0.52,0) {$\cdot$};   
    \node at (1.52,0) {$\cdot$};
    \node at (2.52,0) {$\cdot$};
    \node at (4.00,0) {$\cdots$}; 
    \end{scope}

    \begin{scope}[yshift=-10cm]
        \Bweight{0,1,0,1}
    \CAP{0}{1}{1}
    \CUP{0}{1}{1}
    \CAP{2}{3}{1}
    \CUP{2}{3}{1}
    \FILL{4}
    \node at (-0.80,0) {$\cdots$};   
    \node at (0.52,0) {$\cdot$};   
    \node at (1.52,0) {$\cdot$};
    \node at (2.52,0) {$\cdot$};
    \node at (4.00,0) {$\cdots$}; 
    \end{scope}

    \draw[strand] (1.5,-1.75) -- (1.5,-3.25);
    \draw[strand] (1.5,-6.75) -- (1.5,-8.25);
    \end{scope}

    \begin{scope}[xshift=20cm]
    \Bweight{0,0,1,0}
    \RLINE{0}{1}
    \RLINE{1}{1}
    \pdLINE{2}{1}
    \RLINE{3}{1}
    \node at (-0.80,0) {$\cdots$};   
    \node at (0.52,0) {$\cdot$};   
    \node at (1.52,0) {$\cdot$};
    \node at (2.52,0) {$\cdot$};
    \node at (4.00,0) {$\cdots$}; 

    \begin{scope}[yshift=-5cm]
        \Bweight{0,0,0,1}
    \CAP{2}{3}{1}
    \CUP{2}{3}{1}
    \RLINE{0}{1}
    \RLINE{1}{1}
    \PFILL{2}
    \PFILL{3}
    \node at (-0.80,0) {$\cdots$};   
    \node at (0.52,0) {$\cdot$};   
    \node at (1.52,0) {$\cdot$};
    \node at (2.52,0) {$\cdot$};
    \node at (4.00,0) {$\cdots$}; 
    \end{scope}

    \begin{scope}[yshift=-10cm]
        \Bweight{1,1,0,1}
    \dCAP{0}{1}{1}
    \dCUP{0}{1}{1}
    \CAP{2}{3}{1}
    \CUP{2}{3}{1}
    \FILL{4}
    \node at (-0.80,0) {$\cdots$};   
    \node at (0.52,0) {$\cdot$};   
    \node at (1.52,0) {$\cdot$};
    \node at (2.52,0) {$\cdot$};
    \node at (4.00,0) {$\cdots$}; 
    \end{scope}

    \draw[strand] (1.5,-1.75) -- (1.5,-3.25);
    \draw[colored strand = {0}] (1.5,-6.75) -- (1.5,-8.25);
    \end{scope}

    \end{tikzpicture}$$
    \vspace{.5cm}
    $$\begin{tikzpicture}[scale=.5, baseline={(0,-0.1)}]
   \Bweight{0,0,1,0}
    \RLINE{0}{1}
    \RLINE{1}{1}
    \pdLINE{2}{1}
    \RLINE{3}{1}
    \node at (-0.80,0) {$\cdots$};   
    \node at (0.52,0) {$\cdot$};   
    \node at (1.52,0) {$\cdot$};
    \node at (2.52,0) {$\cdot$};
    \node at (4.00,0) {$\cdots$}; 

    \begin{scope}[yshift=-5cm]
        \Bweight{1,1,1,0}
    \dCAP{0}{1}{1}
    \dCUP{0}{1}{1}
    \pdLINE{2}{1}
    \RLINE{3}{1}
    \FILL{2}
    \node at (-0.80,0) {$\cdots$};   
    \node at (0.52,0) {$\cdot$};   
    \node at (1.52,0) {$\cdot$};
    \node at (2.52,0) {$\cdot$};
    \node at (4.00,0) {$\cdots$}; 
    \end{scope}

    \begin{scope}[yshift=-10cm]
        \Bweight{1,1,0,1}
    \dCAP{0}{1}{1}
    \dCUP{0}{1}{1}
    \CAP{2}{3}{1}
    \CUP{2}{3}{1}
    \FILL{4}
    \node at (-0.80,0) {$\cdots$};   
    \node at (0.52,0) {$\cdot$};   
    \node at (1.52,0) {$\cdot$};
    \node at (2.52,0) {$\cdot$};
    \node at (4.00,0) {$\cdots$}; 
    \end{scope}

    \draw[colored strand = {0}] (1.5,-1.75) -- (1.5,-3.25);
    \draw[strand] (1.5,-6.75) -- (1.5,-8.25);

    \begin{scope}[xshift=10cm]
    \Bweight{1,0,0,0}
    \pdLINE{0}{1}
    \RLINE{1}{1}
    \RLINE{2}{1}
    \RLINE{3}{1}
    \node at (-0.80,0) {$\cdots$};   
    \node at (0.52,0) {$\cdot$};   
    \node at (1.52,0) {$\cdot$};
    \node at (2.52,0) {$\cdot$};
    \node at (4.00,0) {$\cdots$}; 

    \begin{scope}[yshift=-5cm]
        \Bweight{0,1,0,0}
    \CAP{0}{1}{1}
    \CUP{0}{1}{1}
    \RLINE{2}{1}
    \RLINE{3}{1}
    \FILL{2}
    \node at (-0.80,0) {$\cdots$};   
    \node at (0.52,0) {$\cdot$};   
    \node at (1.52,0) {$\cdot$};
    \node at (2.52,0) {$\cdot$};
    \node at (4.00,0) {$\cdots$}; 
    \end{scope}

    \begin{scope}[yshift=-10cm]
        \Bweight{0,1,1,1}
    \CAP{0}{1}{1}
    \CUP{0}{1}{1}
    \dCAP{2}{3}{1}
    \dCUP{2}{3}{1}
    \FILL{4}
    \node at (-0.80,0) {$\cdots$};   
    \node at (0.52,0) {$\cdot$};   
    \node at (1.52,0) {$\cdot$};
    \node at (2.52,0) {$\cdot$};
    \node at (4.00,0) {$\cdots$}; 
    \end{scope}

    \draw[strand] (1.5,-1.75) -- (1.5,-3.25);
    \draw[colored strand = {0}] (1.5,-6.75) -- (1.5,-8.25);
    \end{scope}

    \begin{scope}[xshift=20cm]
    \Bweight{0,0,0,0}
    \RLINE{0}{1}
    \RLINE{1}{1}
    \RLINE{2}{1}
    \RLINE{3}{1}
    \node at (-0.80,0) {$\cdots$};   
    \node at (0.52,0) {$\cdot$};   
    \node at (1.52,0) {$\cdot$};
    \node at (2.52,0) {$\cdot$};
    \node at (4.00,0) {$\cdots$}; 

    \begin{scope}[yshift=-5cm]
        \Bweight{1,1,0,0}
    \dCAP{0}{1}{1}
    \dCUP{0}{1}{1}
    \RLINE{2}{1}
    \RLINE{3}{1}
    \FILL{2}
    \node at (-0.80,0) {$\cdots$};   
    \node at (0.52,0) {$\cdot$};   
    \node at (1.52,0) {$\cdot$};
    \node at (2.52,0) {$\cdot$};
    \node at (4.00,0) {$\cdots$}; 
    \end{scope}

    \begin{scope}[yshift=-10cm]
        \Bweight{1,1,1,1}
    \dCAP{0}{1}{1}
    \dCUP{0}{1}{1}
    \dCAP{2}{3}{1}
    \dCUP{2}{3}{1}
    \FILL{4}
    \node at (-0.80,0) {$\cdots$};   
    \node at (0.52,0) {$\cdot$};   
    \node at (1.52,0) {$\cdot$};
    \node at (2.52,0) {$\cdot$};
    \node at (4.00,0) {$\cdots$}; 
    \end{scope}

    \draw[colored strand = {0}] (1.5,-1.75) -- (1.5,-3.25);
    \draw[colored strand = {0}] (1.5,-6.75) -- (1.5,-8.25);
    \end{scope}

    \draw[strand, dotted] (8,1.5) -- (25,1.5);
    \draw[strand, dotted] (8,-11.5) -- (25,-11.5);
    \draw[strand, dotted] (8,1.5) -- (8,-11.5);
    \draw[strand, dotted] (25,1.5) -- (25,-11.5);
    \draw[white] (-2,1.5) -- (-2,-11.5);

    \end{tikzpicture}$$
    The first two paths appear in the square $\textit{(a)}$ from Lemma~\ref{lem:5.10}, the next two paths appear in the square $\textit{(b)}$ from Lemma~\ref{lem:5.10}, and the last two paths (in the rectangle) are paths $\textit{(e)}$ and $\textit{(f)}$ respectively. \\
    In the third case, $\lambda$ and $\mu$ are different in $4$ positions, and the positions of the pairs $\dn \cdot \up$ or $\up \cdot \up$ used in $\lambda$ and $\nu$ respectively are nested, then such a path look like
    $$\begin{tikzpicture}[scale=.5, baseline={(0,-0.1)}]
    \Bweight{1,1,0,0}
    \pdLINE{0}{1}
    \pdLINE{1}{1}
    \RLINE{2}{1}
    \RLINE{3}{1}
    \node at (-0.80,0) {$\cdots$};   
    \node at (0.52,0) {$\cdot$};   
    \node at (1.52,0) {$\cdot$};
    \node at (2.52,0) {$\cdot$};
    \node at (4.00,0) {$\cdots$};

    \begin{scope}[yshift=-5cm]
        \Bweight{1,0,1,0}
        \pdLINE{0}{1}
        \CUP{1}{2}{1}
        \CAP{1}{2}{1}
        \RLINE{3}{1}
        \PFILL{1}
        \PFILL{2}
        \node at (-0.80,0) {$\cdots$};   
        \node at (0.52,0) {$\cdot$};   
        \node at (1.52,0) {$\cdot$};
        \node at (2.52,0) {$\cdot$};
        \node at (4.00,0) {$\cdots$};
    \end{scope}

    \begin{scope}[yshift=-10cm]
        \Bweight{0,0,1,1}
        \CUP{0}{3}{.6}
        \CUP{1}{2}{1}
        \CAP{1}{2}{1}
        \CAP{0}{3}{.6}
        \FILL{4}
        \node at (-0.80,0) {$\cdots$};   
        \node at (0.52,0) {$\cdot$};   
        \node at (1.52,0) {$\cdot$};
        \node at (2.52,0) {$\cdot$};
        \node at (4.00,0) {$\cdots$};
    \end{scope}

    \draw[strand] (1.5,-1.75) -- (1.5,-3.25);
    \draw[strand] (1.5,-6.75) -- (1.5,-8.25);

    \end{tikzpicture} \qquad\qquad\qquad \begin{tikzpicture}[scale=.5, baseline={(0,-0.1)}]
    \Bweight{0,1,0,0}
    \RLINE{0}{1}
    \pdLINE{1}{1}
    \RLINE{2}{1}
    \RLINE{3}{1}
    \node at (-0.80,0) {$\cdots$};   
    \node at (0.52,0) {$\cdot$};   
    \node at (1.52,0) {$\cdot$};
    \node at (2.52,0) {$\cdot$};
    \node at (4.00,0) {$\cdots$};

    \begin{scope}[yshift=-5cm]
        \Bweight{0,0,1,0}
        \RLINE{0}{1}
        \CUP{1}{2}{1}
        \CAP{1}{2}{1}
        \RLINE{3}{1}
        \PFILL{1}
        \PFILL{2}
        \node at (-0.80,0) {$\cdots$};   
        \node at (0.52,0) {$\cdot$};   
        \node at (1.52,0) {$\cdot$};
        \node at (2.52,0) {$\cdot$};
        \node at (4.00,0) {$\cdots$};
    \end{scope}

    \begin{scope}[yshift=-10cm]
        \Bweight{1,0,1,1}
        \dCUP{0}{3}{.6}
        \CUP{1}{2}{1}
        \CAP{1}{2}{1}
        \dCAP{0}{3}{.6}
        \FILL{4}
        \node at (-0.80,0) {$\cdots$};   
        \node at (0.52,0) {$\cdot$};   
        \node at (1.52,0) {$\cdot$};
        \node at (2.52,0) {$\cdot$};
        \node at (4.00,0) {$\cdots$};
    \end{scope}

    \draw[strand] (1.5,-1.75) -- (1.5,-3.25);
    \draw[colored strand = {0}] (1.5,-6.75) -- (1.5,-8.25);

    \end{tikzpicture}$$
    Observe that if the first path have a strand on the left of it that does not appear in a circle lying entirely on the left of it, then it gives rise to paths $\textit{(g)}$ and $\textit{(h)}$, and otherwise the paths appear in squares $\textit{(d)}$ and $\textit{(c)}$ from Lemma~\ref{lem:5.10} respectively. \\
    Finally, the last case would be that the two pairs $\dn \cdot \up$ or $\up \cdot \up$ used in $\lambda$ and $\nu$ respectively, overlap in the following way
    $$\begin{tikzpicture}[baseline = {(0,-0.5)}]
    \filldraw[\colorFor{1}]  (0,0) circle (2pt);
    \filldraw[\colorFor{1}]  (2,0) circle (2pt);
    \filldraw[\colorFor{0}] (1,0) circle (2pt);
    \filldraw[\colorFor{0}] (3,0) circle (2pt);
    \CCUP{0}{2}{.5}{1}
    \CCUP{1}{3}{.5}{0}
    \node at (.8,-1) {$\textcolor{\colorFor{1}}{\text{in } e_\lambda}$};
    \node at (2.2,-1) {$\textcolor{\colorFor{0}}{\text{in } e_\nu}$};
    \end{tikzpicture} \qquad \text{ or } \qquad \begin{tikzpicture}[baseline = {(0,-0.5)}]
    \filldraw[\colorFor{0}]  (0,0) circle (2pt);
    \filldraw[\colorFor{0}]  (2,0) circle (2pt);
    \filldraw[\colorFor{1}] (1,0) circle (2pt);
    \filldraw[\colorFor{1}] (3,0) circle (2pt);
    \CCUP{0}{2}{.5}{0}
    \CCUP{1}{3}{.5}{1}
    \node at (.8,-1) {$\textcolor{\colorFor{0}}{\text{in } e_\nu}$};
    \node at (2.2,-1) {$\textcolor{\colorFor{1}}{\text{in } e_\lambda}$};
    \end{tikzpicture}$$
    but that's not possible. Indeed, suppose that we switch a pair $\dn \cdots \up$ or $\up \cdot \up$ in $e_\lambda$ at positions $1$ and $3$ (the other case is symmetric), it would mean that the arrow at position $2$ lie in a circle in $e_\lambda$, because otherwise it would not be an admissible diagram. This would give us
    $$\begin{tikzpicture}[scale=.5, baseline={(0,-1.3)}]
        \CBweight{1,0,1,0}{1,0,0,1}
        \CpdLINE{0}{1}{1}
        \CCUP{1}{2}{1}{0}
        \CCAP{1}{2}{1}{0}
        \CRLINE{3}{1}{1}
        \CPFILL{1}{0}
        \CPFILL{2}{0}
        \filldraw[\colorFor{0}]  (4,0) circle (4pt);
        \node at (-0.80,0) {$\cdots$};   
        \node at (0.52,0) {$\cdot$};   
        \node at (1.52,0) {$\cdot$};
        \node at (2.52,0) {$\cdot$};
        \node at (3.52,0) {$\cdot$};
        \node at (5.00,0) {$\cdots$};

    \begin{scope}[yshift=-5cm]
        \CBweight{0,0,1,1}{1,0,0,1}
        \CCUP{0}{3}{.6}{1}
        \CCUP{1}{2}{1}{0}
        \CCAP{1}{2}{1}{0}
        \CCAP{0}{3}{.6}{1}
        \CFILL{3}{1,0,0,1}
        \filldraw[\colorFor{0}]  (4,0) circle (4pt);
        \node at (-0.80,0) {$\cdots$};   
        \node at (0.52,0) {$\cdot$};   
        \node at (1.52,0) {$\cdot$};
        \node at (2.52,0) {$\cdot$};
        \node at (3.52,0) {$\cdot$};
        \node at (5.00,0) {$\cdots$};
    \end{scope}

    \draw[strand] (2,-1.75) -- (2,-3.25);
    
    \end{tikzpicture} \qquad \text{ or } \qquad \begin{tikzpicture}[scale=.5, baseline={(0,-1.3)}]
        \CBweight{0,0,1,0}{1,0,0,1}
        \CRLINE{0}{1}{1}
        \CCUP{1}{2}{1}{0}
        \CCAP{1}{2}{1}{0}
        \CRLINE{3}{1}{1}
        \CPFILL{1}{0}
        \CPFILL{2}{0}
        \filldraw[\colorFor{0}]  (4,0) circle (4pt);
        \node at (-0.80,0) {$\cdots$};   
        \node at (0.52,0) {$\cdot$};   
        \node at (1.52,0) {$\cdot$};
        \node at (2.52,0) {$\cdot$};
        \node at (3.52,0) {$\cdot$};
        \node at (5.00,0) {$\cdots$};

    \begin{scope}[yshift=-5cm]
        \CBweight{1,0,1,1}{1,0,0,1}
        \CdCUP{0}{3}{.6}{1}
        \CCUP{1}{2}{1}{0}
        \CCAP{1}{2}{1}{0}
        \CdCAP{0}{3}{.6}{1}
        \CFILL{3}{1,0,0,1}
        \filldraw[\colorFor{0}]  (4,0) circle (4pt);
        \node at (-0.80,0) {$\cdots$};   
        \node at (0.52,0) {$\cdot$};   
        \node at (1.52,0) {$\cdot$};
        \node at (2.52,0) {$\cdot$};
        \node at (3.52,0) {$\cdot$};
        \node at (5.00,0) {$\cdots$};
    \end{scope}

    \draw[colored strand = {0}] (2,-1.75) -- (2,-3.25);
    
    \end{tikzpicture}$$
    But then the blue circle stays intact in $e_\nu$, so we cannot construct $\mu$ thanks to a circle lying in positions $2$ and $4$. \\
    This proves that all paths not in squares of shape
    $$\begin{tikzcd}
	\mu \\
	\nu \\
	\lambda
	\arrow[no head, very thick, from=2-1, to=1-1]
	\arrow[no head, very thick, from=3-1, to=2-1]
    \end{tikzcd}$$
    are $\textit{(a)}$, $\textit{(b)}$, $\textit{(c)}$, $\textit{(d)}$, $\textit{(e)}$, $\textit{(f)}$, $\textit{(g)}$ and $\textit{(h)}$. \\
    For the case $\vert \lambda \vert \leq \vert \mu \vert < \vert \nu \vert$, the idea is the same, and it gives rise to square $\textit{(i)}$.
\end{proof}

In order to compute the relations in the quiver $Q_n$, we need to give names to the arrows in it.

\begin{definition}\label{def:5.13}
    Let $\gamma_\lambda^\mu$ be an edge in $\Gamma_n$ connecting $\lambda$ and $\mu$ with $\vert \lambda \vert < \vert \mu \vert $. Then we denote in $Q_n$ the ascending arrow from $\lambda$ to $\mu$ by $x_\lambda^\mu$, and the descending arrow by $y_\lambda^\mu$. We recall that they correpond to degree $1$ diagrams.
\end{definition}

\begin{definition}\label{def:5.14}
    Let $\lambda \in \Delta_n^0$ be a weight. Then the defect of $\lambda$, denoted $\text{def}(\lambda)$, is the number of circles in $e_\lambda$.
\end{definition}

\begin{remark}
    Let $\lambda \in \Delta_n^0$ be a weight. Observe that $\text{def}(\lambda)$ is also the number of ascending arrows in $Q_n$ starting at $e_\lambda$.
\end{remark}

\begin{theorem}\label{thm:5.15}
    Let $\kappa, \lambda, \mu_i \in \Delta_n^0$ for $1 \leq i \leq \text{def}(\lambda)$, where $\mu_i$ is obtained from $\lambda$ by exchanging a $\dn \cdots \up$ or $\up \cdots \up$ pair lying in a circle $C_i$ of $e_\lambda$, and where $\lambda$ is obtained from $\kappa$ by exchanging a $\dn \cdots \up$ or $\up \cdots \up$ pair lying in a circle $D$ of $e_\kappa$. Then the map $\omega : \Bbbk Q_n \twoheadrightarrow \D_n $ sends the following quadratic relations to $0$.

    \textbf{Monomial relations :} The length $2$ paths from the top to the bottom, and from the bottom to the top of figures $\textit{(a)}$, $\textit{(b)}$, $\textit{(c)}$ and $\textit{(d)}$ from Lemma~\ref{lem:5.11}.

    \textbf{Square relations :} The difference of any parallel legth $2$ paths from any square from Lemma~\ref{lem:5.10}.

    \textbf{Relations at vertices :} The following linear combinations of paths starting and ending at $\lambda$
    $$y_\kappa^\lambda x_\kappa^\lambda - \sum_{i=1}^{\text{def}(\lambda)} d_\kappa^{\mu_i}(\lambda) x_\lambda^{\mu_i} y_\lambda^{\mu_i}$$
    Where the coefficients $d_\kappa^{\mu_i}(\lambda)$ are given by 
    $$
    d_\kappa^{\mu_i}(\lambda) =
    \begin{cases}
        1  &\text{ if } C_i \text{ does not appear in } e_\kappa \text{ and does not enclose another }\\[-1ex]
        &\text{ circle } C_j \text{ that does not appear either in } e_\kappa,\\[1ex]

        -2 &\text{ if } C_i \text{ does not appear in } e_\kappa \text{ and encloses another }\\[-1ex]
        &\text{ circle } C_j \text{ that does not appear either in } e_\kappa,\\[1ex]

        0 &\text{ if } C_i \text{ appears in } e_\kappa \text{ and does not enclose } D, \\[1ex]

        (-1)^{j-1}2 &\text{ if } C_i \text{ appears in } e_\kappa \text{ and is the } j \text{th circle enclosing } D.
    \end{cases}
    $$
\end{theorem}

\begin{proof}
    We need to check all the square relations diagrammatically, but there are many of them, so we are just going to give some examples, but we will detail monomial relations, and the relations at vertices which are a bit harder. \\
    Firstly, all the monomial relations correspond to paths of shape 
    $$\begin{tikzcd}
	\mu \\
	\nu \\
	\lambda
	\arrow[no head, very thick, from=2-1, to=1-1]
	\arrow[no head, very thick, from=3-1, to=2-1]
    \end{tikzcd}$$
    and in every case we can observe that $\underline{\lambda}\overline{\mu}$ and $\underline{\mu}\overline{\lambda}$ are not orientable, so we deduce the monomial relations. \\
    We verify now a square relation corresponding to the square
    $$\begin{tikzpicture}[scale=.4, baseline={(0,-0.1)}]
    \Bweight{0,0,1,0}
    \RLINE{0}{1}
    \RLINE{1}{1}
    \RLINE{3}{1}
    \pdLINE{2}{1}
    \node at (-0.80,0) {$\cdots$};   
    \node at (0.52,0) {$\cdot$};   
    \node at (1.52,0) {$\cdot$};
    \node at (2.52,0) {$\cdot$};   
    \node at (4,0) {$\cdots$};

    \begin{scope}[yshift=-4cm]
        \begin{scope}[xshift=-4cm]
            \Bweight{1,1,1,0}
            \dCUP{0}{1}{1}
            \dCAP{0}{1}{1}
            \pdLINE{2}{1}
            \RLINE{3}{1}
            \FILL{2}
            \node at (-0.80,0) {$\cdots$};   
            \node at (0.52,0) {$\cdot$};   
            \node at (1.52,0) {$\cdot$};
            \node at (2.52,0) {$\cdot$};   
            \node at (4,0) {$\cdots$};
        \end{scope}
        \begin{scope}[xshift=4cm]
            \Bweight{0,0,0,1}
            \CUP{2}{3}{1}
            \CAP{2}{3}{1}
            \RLINE{0}{1}
            \RLINE{1}{1}
            \PFILL{2}
            \PFILL{3}
            \node at (-0.80,0) {$\cdots$};   
            \node at (0.52,0) {$\cdot$};   
            \node at (1.52,0) {$\cdot$};
            \node at (2.52,0) {$\cdot$};   
            \node at (4,0) {$\cdots$};
        \end{scope}
    \end{scope}

    \begin{scope}[yshift=-8cm]
        \Bweight{1,1,0,1}
        \dCUP{0}{1}{1}
        \dCAP{0}{1}{1}
        \CUP{2}{3}{1}
        \CAP{2}{3}{1}
        \FILL{4}
        \node at (-0.80,0) {$\cdots$};   
        \node at (0.52,0) {$\cdot$};   
        \node at (1.52,0) {$\cdot$};
        \node at (2.52,0) {$\cdot$};   
        \node at (4,0) {$\cdots$};
    \end{scope}

    \draw[colored strand = {0}, <-] (-1,-1) -- (-2,-2.5);
    \draw[strand, <-] (4,-1) -- (5,-2.5);
    \draw[strand, ->] (-1,-7) -- (-2,-5.5);
    \draw[colored strand = {0}, ->] (4,-7) -- (5,-5.5);

    \node at (-2.5,-1.75) {$x_\nu^\mu$};
    \node at (5.5,-1.75) {$x_{\nu'}^\mu$};
    \node at (-2.5,-6.75) {$x_\lambda^\nu$};
    \node at (5.5,-6.75) {$x^{\nu'}_\lambda$};

\end{tikzpicture}$$
    by computing
    $$\omega(x_\lambda^\nu x_\nu^\mu) \ = \ \begin{tikzpicture}[scale=.5, baseline={(0,.68)}]
        \Bweight{1,1,1,0}
        \dCUP{0}{1}{1}
        \CUP{2}{3}{1}
        \dCAP{0}{1}{1}
        \TpdLINE{2}{1}
        \TRLINE{3}{1}
        \FILL{4}
        \node at (-0.80,0) {$\cdots$};   
        \node at (0.52,0) {$\cdot$};   
        \node at (1.52,0) {$\cdot$};
        \node at (2.52,0) {$\cdot$};
        \node at (4.00,0) {$\cdots$};

        \begin{scope}[yshift=3cm]
             \Bweight{0,0,1,0}
        \dCUP{0}{1}{1}
        \TRLINE{0}{1}
        \TRLINE{1}{1}
        \pdLINE{2}{1}
        \RLINE{3}{1}
        \FILL{4}
        \node at (-0.80,0) {$\cdots$};   
        \node at (0.52,0) {$\cdot$};   
        \node at (1.52,0) {$\cdot$};
        \node at (2.52,0) {$\cdot$};
        \node at (4.00,0) {$\cdots$};
        \end{scope}
    \end{tikzpicture} \ = \ \begin{tikzpicture}[scale=.5, baseline={(0,.68)}]
        \Bweight{0,0,1,0}
        \dCUP{0}{1}{1}
        \CUP{2}{3}{1}
        \UPLINE{0}{1.5}
        \UPLINE{1}{1.5}
        \UPLINE{2}{1.5}
        \UPLINE{3}{1.5}
        \FILL{4}
        \node at (-0.80,0) {$\cdots$};   
        \node at (0.52,0) {$\cdot$};   
        \node at (1.52,0) {$\cdot$};
        \node at (2.52,0) {$\cdot$};
        \node at (4.00,0) {$\cdots$};

        \begin{scope}[yshift=3cm]
             \Bweight{0,0,1,0}
        \TRLINE{0}{1}
        \TRLINE{1}{1}
        \TpdLINE{2}{1}
        \TRLINE{3}{1}
        \DOWNLINE{0}{1.5}
        \DOWNLINE{1}{1.5}
        \DOWNLINE{2}{1.5}
        \DOWNLINE{3}{1.5}
        \FILL{4}
        \node at (-0.80,0) {$\cdots$};   
        \node at (0.52,0) {$\cdot$};   
        \node at (1.52,0) {$\cdot$};
        \node at (2.52,0) {$\cdot$};
        \node at (4.00,0) {$\cdots$};
        \end{scope}
    \end{tikzpicture} \ = \ \begin{tikzpicture}[scale=.5, baseline={(0,-0.1)}]
            \Bweight{0,0,1,0}
            \TRLINE{0}{1}
        \TRLINE{1}{1}
        \TpdLINE{2}{1}
        \TRLINE{3}{1}
            \dCUP{0}{1}{1}
            \CUP{2}{3}{1}
            \FILL{4}
           \node at (-0.80,0) {$\cdots$};   
        \node at (0.52,0) {$\cdot$};   
        \node at (1.52,0) {$\cdot$};
        \node at (2.52,0) {$\cdot$};
        \node at (4.00,0) {$\cdots$};
    \end{tikzpicture} $$
    and
    $$\omega(x_\lambda^{\nu'} x_{\nu'}^\mu) \ = \ \begin{tikzpicture}[scale=.5, baseline={(0,.68)}]
        \Bweight{0,0,0,1}
        \dCUP{0}{1}{1}
        \CUP{2}{3}{1}
        \CAP{2}{3}{1}
        \TRLINE{0}{1}
        \TRLINE{1}{1}
        \FILL{4}
        \node at (-0.80,0) {$\cdots$};   
        \node at (0.52,0) {$\cdot$};   
        \node at (1.52,0) {$\cdot$};
        \node at (2.52,0) {$\cdot$};
        \node at (4.00,0) {$\cdots$};

        \begin{scope}[yshift=3cm]
             \Bweight{0,0,1,0}
        \CUP{2}{3}{1}
        \RLINE{0}{1}
        \RLINE{1}{1}
        \TpdLINE{2}{1}
        \TRLINE{3}{1}
        \FILL{4}
        \node at (-0.80,0) {$\cdots$};   
        \node at (0.52,0) {$\cdot$};   
        \node at (1.52,0) {$\cdot$};
        \node at (2.52,0) {$\cdot$};
        \node at (4.00,0) {$\cdots$};
        \end{scope}
    \end{tikzpicture} \ = \ \begin{tikzpicture}[scale=.5, baseline={(0,.68)}]
        \Bweight{0,0,1,0}
        \dCUP{0}{1}{1}
        \CUP{2}{3}{1}
        \UPLINE{0}{1.5}
        \UPLINE{1}{1.5}
        \UPLINE{2}{1.5}
        \UPLINE{3}{1.5}
        \FILL{4}
        \node at (-0.80,0) {$\cdots$};   
        \node at (0.52,0) {$\cdot$};   
        \node at (1.52,0) {$\cdot$};
        \node at (2.52,0) {$\cdot$};
        \node at (4.00,0) {$\cdots$};

        \begin{scope}[yshift=3cm]
             \Bweight{0,0,1,0}
        \TRLINE{0}{1}
        \TRLINE{1}{1}
        \TpdLINE{2}{1}
        \TRLINE{3}{1}
        \DOWNLINE{0}{1.5}
        \DOWNLINE{1}{1.5}
        \DOWNLINE{2}{1.5}
        \DOWNLINE{3}{1.5}
        \FILL{4}
        \node at (-0.80,0) {$\cdots$};   
        \node at (0.52,0) {$\cdot$};   
        \node at (1.52,0) {$\cdot$};
        \node at (2.52,0) {$\cdot$};
        \node at (4.00,0) {$\cdots$};
        \end{scope}
    \end{tikzpicture} \ = \ \begin{tikzpicture}[scale=.5, baseline={(0,-0.1)}]
            \Bweight{0,0,1,0}
            \TRLINE{0}{1}
        \TRLINE{1}{1}
        \TpdLINE{2}{1}
        \TRLINE{3}{1}
            \dCUP{0}{1}{1}
            \CUP{2}{3}{1}
            \FILL{4}
           \node at (-0.80,0) {$\cdots$};   
        \node at (0.52,0) {$\cdot$};   
        \node at (1.52,0) {$\cdot$};
        \node at (2.52,0) {$\cdot$};
        \node at (4.00,0) {$\cdots$};
    \end{tikzpicture}$$
    so that $\omega(x_\lambda^\nu x_\nu^\mu - x_\lambda^{\nu'} x_{\nu'}^\mu) = 0$ as wanted. \\
    We also verify a square relation corresponding to the square
    $$\begin{tikzpicture}[scale=.4, baseline={(0,-0.1)}]
    \Bweight{0,0,1,0}
    \CUP{1}{2}{1}
    \CAP{1}{2}{1}
    \RLINE{0}{1}
    \RLINE{3}{1}
    \PFILL{1}
    \PFILL{2}
    \node at (-0.80,0) {$\cdots$};   
    \node at (0.52,0) {$\cdot$};   
    \node at (1.52,0) {$\cdot$};
    \node at (2.52,0) {$\cdot$};   
    \node at (4,0) {$\cdots$};

    \begin{scope}[yshift=-4cm]
        \begin{scope}[xshift=4cm]
            \Bweight{0,0,0,1}
            \CUP{2}{3}{1}
            \CAP{2}{3}{1}
            \RLINE{0}{1}
            \RLINE{1}{1}
            \PFILL{2}
            \PFILL{3}
            \node at (-0.80,0) {$\cdots$};   
            \node at (0.52,0) {$\cdot$};   
            \node at (1.52,0) {$\cdot$};
            \node at (2.52,0) {$\cdot$};   
            \node at (4,0) {$\cdots$};
        \end{scope}
    \end{scope}

    \begin{scope}[yshift=-8cm]
        \begin{scope}[xshift=4cm]
            \Bweight{1,1,0,1}
            \dCUP{0}{1}{1}
            \dCAP{0}{1}{1}
            \CUP{2}{3}{1}
            \CAP{2}{3}{1}
            \FILL{4}
            \node at (-0.80,0) {$\cdots$};   
            \node at (0.52,0) {$\cdot$};   
            \node at (1.52,0) {$\cdot$};
            \node at (2.52,0) {$\cdot$};   
            \node at (4,0) {$\cdots$};
        \end{scope}
    \end{scope}

    \begin{scope}[yshift=-12cm]
        \Bweight{1,0,1,1}
        \CUP{1}{2}{1}
        \CAP{1}{2}{1}
        \dCUP{0}{3}{.7}
        \dCAP{0}{3}{.7}
        \FILL{4}
        \node at (-0.80,0) {$\cdots$};   
        \node at (0.52,0) {$\cdot$};   
        \node at (1.52,0) {$\cdot$};
        \node at (2.52,0) {$\cdot$};   
        \node at (4,0) {$\cdots$};
    \end{scope}

    \draw[strand, <-] (4,-1) -- (5,-2.5);
    \draw[strand, <-] (4,-11) -- (5,-9.5);
    \draw[colored strand = {0}, ->] (5.5,-5.5) -- (5.5,-6.5);
    \draw[colored strand = {0}, <-] (-1,-11) .. controls (-3,-8) and (-3,-4) .. (-1,-1);

    \node at (5.5,-1.75) {$x_\mu^\nu$};
    \node at (5.5,-10.25) {$y^{\nu'}_\lambda$};
    \node at (7, -6) {$y_{\nu'}^\mu$};
    \node at (-3.5, -6) {$y_\lambda^\nu$};

\end{tikzpicture}$$
    by computing
    $$\omega(x_\mu^\nu y_\lambda^\nu) \ = \ \begin{tikzpicture}[scale=.5, baseline={(0,.68)}]
        \Bweight{0,0,1,0}
        \CUP{2}{3}{1}
        \CAP{1}{2}{1}
        \RLINE{0}{1}
        \BRLINE{1}{1}
        \TRLINE{3}{1}
        \node at (-0.80,0) {$\cdots$};   
        \node at (0.52,0) {$\cdot$};   
        \node at (1.52,0) {$\cdot$};
        \node at (2.52,0) {$\cdot$};
        \node at (4.00,0) {$\cdots$};

        \begin{scope}[yshift=3cm]
            \Bweight{0,0,1,0}
            \CUP{1}{2}{1}
            \CAP{1}{2}{1}
            \dCAP{0}{3}{.7}
            \BRLINE{0}{1}
            \BRLINE{3}{1}
            \node at (-0.80,0) {$\cdots$};   
            \node at (0.52,0) {$\cdot$};   
            \node at (1.52,0) {$\cdot$};
            \node at (2.52,0) {$\cdot$};
            \node at (4.00,0) {$\cdots$};
        \end{scope}
    \end{tikzpicture} \ = \ \begin{tikzpicture}[scale=.5, baseline={(0,.68)}]
        \Bweight{0,0,1,0}
        \CUP{2}{3}{1}
        \BRLINE{0}{1}
        \BRLINE{1}{1}
        \UPLINE{0}{1.5}
        \UPLINE{1}{1.5}
        \UPLINE{2}{1.5}
        \UPLINE{3}{1.5}
        \node at (-0.80,0) {$\cdots$};   
        \node at (0.52,0) {$\cdot$};   
        \node at (1.52,0) {$\cdot$};
        \node at (2.52,0) {$\cdot$};
        \node at (4.00,0) {$\cdots$};

        \begin{scope}[yshift=3cm]
            \Bweight{0,0,1,0}
            \CAP{1}{2}{1}
            \dCAP{0}{3}{.7}
            \DOWNLINE{0}{1.5}
            \DOWNLINE{1}{1.5}
            \DOWNLINE{2}{1.5}
            \DOWNLINE{3}{1.5}
            \node at (-0.80,0) {$\cdots$};   
            \node at (0.52,0) {$\cdot$};   
            \node at (1.52,0) {$\cdot$};
            \node at (2.52,0) {$\cdot$};
            \node at (4.00,0) {$\cdots$};
        \end{scope}
    \end{tikzpicture} \ = \ \begin{tikzpicture}[scale=.5, baseline={(0,-0.1)}]
            \Bweight{0,0,1,0}
            \CAP{1}{2}{1}
            \dCAP{0}{3}{.7}
            \CUP{2}{3}{1}
            \BRLINE{0}{1}
            \BRLINE{1}{1}
            \node at (-0.80,0) {$\cdots$};   
            \node at (0.52,0) {$\cdot$};   
            \node at (1.52,0) {$\cdot$};
            \node at (2.52,0) {$\cdot$};
            \node at (4.00,0) {$\cdots$};
    \end{tikzpicture}$$
    and
    $$\omega(y_{\nu'}^\mu y_\lambda^{\nu'}) \ = \ \begin{tikzpicture}[scale=.5, baseline={(0,.68)}]
        \Bweight{0,0,0,1}
        \CUP{2}{3}{1}
        \dCAP{0}{1}{1}
        \CAP{2}{3}{1}
        \BRLINE{0}{1}
        \BRLINE{1}{1}
        \node at (-0.80,0) {$\cdots$};   
        \node at (0.52,0) {$\cdot$};   
        \node at (1.52,0) {$\cdot$};
        \node at (2.52,0) {$\cdot$};
        \node at (4.00,0) {$\cdots$};

        \begin{scope}[yshift=3cm]
            \Bweight{1,1,0,1}
            \dCUP{0}{1}{1}
            \CUP{2}{3}{1}
            \CAP{1}{2}{1}
            \dCAP{0}{3}{.7}
            \node at (-0.80,0) {$\cdots$};   
            \node at (0.52,0) {$\cdot$};   
            \node at (1.52,0) {$\cdot$};
            \node at (2.52,0) {$\cdot$};
            \node at (4.00,0) {$\cdots$};
        \end{scope}
    \end{tikzpicture} \ = \ \begin{tikzpicture}[scale=.5, baseline={(0,.68)}]
        \Bweight{0,0,1,0}
        \CUP{2}{3}{1}
        \BRLINE{0}{1}
        \BRLINE{1}{1}
        \UPLINE{0}{1.5}
        \UPLINE{1}{1.5}
        \UPLINE{2}{1.5}
        \UPLINE{3}{1.5}
        \node at (-0.80,0) {$\cdots$};   
        \node at (0.52,0) {$\cdot$};   
        \node at (1.52,0) {$\cdot$};
        \node at (2.52,0) {$\cdot$};
        \node at (4.00,0) {$\cdots$};

        \begin{scope}[yshift=3cm]
            \Bweight{0,0,1,0}
            \CAP{1}{2}{1}
            \dCAP{0}{3}{.7}
            \DOWNLINE{0}{1.5}
            \DOWNLINE{1}{1.5}
            \DOWNLINE{2}{1.5}
            \DOWNLINE{3}{1.5}
            \node at (-0.80,0) {$\cdots$};   
            \node at (0.52,0) {$\cdot$};   
            \node at (1.52,0) {$\cdot$};
            \node at (2.52,0) {$\cdot$};
            \node at (4.00,0) {$\cdots$};
        \end{scope}
    \end{tikzpicture} \ = \ \begin{tikzpicture}[scale=.5, baseline={(0,-0.1)}]
            \Bweight{0,0,1,0}
            \CAP{1}{2}{1}
            \dCAP{0}{3}{.7}
            \CUP{2}{3}{1}
            \BRLINE{0}{1}
            \BRLINE{1}{1}
            \node at (-0.80,0) {$\cdots$};   
            \node at (0.52,0) {$\cdot$};   
            \node at (1.52,0) {$\cdot$};
            \node at (2.52,0) {$\cdot$};
            \node at (4.00,0) {$\cdots$};
    \end{tikzpicture}$$
    so that $\omega(x_\mu^\nu y_\lambda^\nu - y_{\nu'}^\mu y_\lambda^{\nu'}) = 0$. \\
    Finally we need to prove the relations at vertices. Suppose first that there are circles in $e_\lambda$ appearing in $e_\kappa$ enclosing $D$. Then there are two circles in $e_\lambda$ that do not appear in $e_\kappa$, we denote them $C_1$ and $C_2$, from left to right. Observe that the innermost circle in $e_\kappa$ enclosing $D$ does not appear in $e_\lambda$, and we denote by $C_i$ the $(i-1)$th circle in $e_\kappa$ enclosing $D$ for $3 \leq i \leq k$. Thus $\mu_i$ is obtained from $\lambda$ by exchanging a $\dn \cdots \up$ pair lying in a circle $C_i$ of $e_\lambda$, for $1 \leq i \leq k$. Diagrammatically, $e_\kappa, e_\lambda$ and $e_{\mu_i}$ have the following form
    $$\begin{tikzpicture}[scale=.4, baseline={(0,-0.08)}]
        \FILL{10}
        \CAP{4}{5}{1}
        \CAP{3}{6}{.6}
        \CAP{2}{7}{.5}
        \CAP{1}{8}{.5}
        \CAP{0}{9}{.5}
        \CUP{4}{5}{1}
        \CUP{3}{6}{.6}
        \CUP{2}{7}{.5}
        \CUP{1}{8}{.5}
        \CUP{0}{9}{.5}
        \node at (-0.48,0) {$\cdot$};   
        \node at (0.52,0) {$\cdot$};   
        \node at (1.52,0) {$\cdot$};
        \node at (2.52,0) {$\cdot$};
        \node at (3.52,0) {$\cdot$};   
        \node at (4.52,0) {$\cdot$};
        \node at (5.52,0) {$\cdot$};
        \node at (6.52,0) {$\cdot$};   
        \node at (7.52,0) {$\cdot$};
        \node at (8.52,0) {$\cdot$};
        \node at (9.52,0) {$\cdot$};

        \node at (4.5,4) {$\kappa$};
    \end{tikzpicture} \qquad \begin{tikzpicture}[scale=.4, baseline={(0,-0.08)}]
        \FILL{10}
        \CAP{3}{4}{1}
        \CAP{5}{6}{1}
        \CAP{2}{7}{.5}
        \CAP{1}{8}{.5}
        \CAP{0}{9}{.5}
        \CUP{3}{4}{1}
        \CUP{5}{6}{1}
        \CUP{2}{7}{.5}
        \CUP{1}{8}{.5}
        \CUP{0}{9}{.5}
        \node at (-0.48,0) {$\cdot$};   
        \node at (0.52,0) {$\cdot$};   
        \node at (1.52,0) {$\cdot$};
        \node at (2.52,0) {$\cdot$};
        \node at (3.52,0) {$\cdot$};   
        \node at (4.52,0) {$\cdot$};
        \node at (5.52,0) {$\cdot$};
        \node at (6.52,0) {$\cdot$};   
        \node at (7.52,0) {$\cdot$};
        \node at (8.52,0) {$\cdot$};
        \node at (9.52,0) {$\cdot$};

        \node at (4.5,4) {$\lambda$};
    \end{tikzpicture} \qquad \begin{tikzpicture}[scale=.4, baseline={(0,-0.08)}]
        \FILL{10}
        \CAP{2}{3}{1}
        \CAP{5}{6}{1}
        \CAP{4}{7}{.6}
        \CAP{1}{8}{.5}
        \CAP{0}{9}{.5}
        \CUP{2}{3}{1}
        \CUP{5}{6}{1}
        \CUP{4}{7}{.6}
        \CUP{1}{8}{.5}
        \CUP{0}{9}{.5}
        \node at (-0.48,0) {$\cdot$};   
        \node at (0.52,0) {$\cdot$};   
        \node at (1.52,0) {$\cdot$};
        \node at (2.52,0) {$\cdot$};
        \node at (3.52,0) {$\cdot$};   
        \node at (4.52,0) {$\cdot$};
        \node at (5.52,0) {$\cdot$};
        \node at (6.52,0) {$\cdot$};   
        \node at (7.52,0) {$\cdot$};
        \node at (8.52,0) {$\cdot$};
        \node at (9.52,0) {$\cdot$};

        \node at (4.5,4) {$\mu_1$};
    \end{tikzpicture}$$
    $$\begin{tikzpicture}[scale=.4, baseline={(0,-0.08)}]
        \FILL{10}
        \CAP{3}{4}{1}
        \CAP{6}{7}{1}
        \CAP{2}{5}{.6}
        \CAP{1}{8}{.5}
        \CAP{0}{9}{.5}
        \CUP{3}{4}{1}
        \CUP{6}{7}{1}
        \CUP{2}{5}{.6}
        \CUP{1}{8}{.5}
        \CUP{0}{9}{.5}
        \node at (-0.48,0) {$\cdot$};   
        \node at (0.52,0) {$\cdot$};   
        \node at (1.52,0) {$\cdot$};
        \node at (2.52,0) {$\cdot$};
        \node at (3.52,0) {$\cdot$};   
        \node at (4.52,0) {$\cdot$};
        \node at (5.52,0) {$\cdot$};
        \node at (6.52,0) {$\cdot$};   
        \node at (7.52,0) {$\cdot$};
        \node at (8.52,0) {$\cdot$};
        \node at (9.52,0) {$\cdot$};

        \node at (4.5,-4) {$\mu_2$};
    \end{tikzpicture} \qquad \cdots \qquad \begin{tikzpicture}[scale=.4, baseline={(0,-0.08)}]
        \FILL{10}
        \CAP{3}{4}{1}
        \CAP{5}{6}{1}
        \CAP{2}{7}{.5}
        \CAP{1}{8}{.5}
        \pdLINE{0}{1.7}
        \CUP{3}{4}{1}
        \CUP{5}{6}{1}
        \CUP{2}{7}{.5}
        \CUP{1}{8}{.5}
        \RLINE{9}{2}
        \node at (-0.48,0) {$\cdot$};   
        \node at (0.52,0) {$\cdot$};   
        \node at (1.52,0) {$\cdot$};
        \node at (2.52,0) {$\cdot$};
        \node at (3.52,0) {$\cdot$};   
        \node at (4.52,0) {$\cdot$};
        \node at (5.52,0) {$\cdot$};
        \node at (6.52,0) {$\cdot$};   
        \node at (7.52,0) {$\cdot$};
        \node at (8.52,0) {$\cdot$};
        \node at (9.52,0) {$\cdot$};

        \node at (4.5,-4) {$\mu_k$};
    \end{tikzpicture}$$
    or
    $$\begin{tikzpicture}[scale=.4, baseline={(0,-0.08)}]
        \FILL{10}
        \CAP{4}{5}{1}
        \CAP{3}{6}{.6}
        \CAP{2}{7}{.5}
        \CAP{1}{8}{.5}
        \dCAP{0}{9}{.5}
        \CUP{4}{5}{1}
        \CUP{3}{6}{.6}
        \CUP{2}{7}{.5}
        \CUP{1}{8}{.5}
        \dCUP{0}{9}{.5}
        \node at (-0.48,0) {$\cdot$};   
        \node at (0.52,0) {$\cdot$};   
        \node at (1.52,0) {$\cdot$};
        \node at (2.52,0) {$\cdot$};
        \node at (3.52,0) {$\cdot$};   
        \node at (4.52,0) {$\cdot$};
        \node at (5.52,0) {$\cdot$};
        \node at (6.52,0) {$\cdot$};   
        \node at (7.52,0) {$\cdot$};
        \node at (8.52,0) {$\cdot$};
        \node at (9.52,0) {$\cdot$};

        \node at (4.5,4) {$\kappa$};
    \end{tikzpicture} \qquad \begin{tikzpicture}[scale=.4, baseline={(0,-0.08)}]
        \FILL{10}
        \CAP{3}{4}{1}
        \CAP{5}{6}{1}
        \CAP{2}{7}{.5}
        \CAP{1}{8}{.5}
        \dCAP{0}{9}{.5}
        \CUP{3}{4}{1}
        \CUP{5}{6}{1}
        \CUP{2}{7}{.5}
        \CUP{1}{8}{.5}
        \dCUP{0}{9}{.5}
        \node at (-0.48,0) {$\cdot$};   
        \node at (0.52,0) {$\cdot$};   
        \node at (1.52,0) {$\cdot$};
        \node at (2.52,0) {$\cdot$};
        \node at (3.52,0) {$\cdot$};   
        \node at (4.52,0) {$\cdot$};
        \node at (5.52,0) {$\cdot$};
        \node at (6.52,0) {$\cdot$};   
        \node at (7.52,0) {$\cdot$};
        \node at (8.52,0) {$\cdot$};
        \node at (9.52,0) {$\cdot$};

        \node at (4.5,4) {$\lambda$};
    \end{tikzpicture} \qquad \begin{tikzpicture}[scale=.4, baseline={(0,-0.08)}]
        \FILL{10}
        \CAP{2}{3}{1}
        \CAP{5}{6}{1}
        \CAP{4}{7}{.6}
        \CAP{1}{8}{.5}
        \dCAP{0}{9}{.5}
        \CUP{2}{3}{1}
        \CUP{5}{6}{1}
        \CUP{4}{7}{.6}
        \CUP{1}{8}{.5}
        \dCUP{0}{9}{.5}
        \node at (-0.48,0) {$\cdot$};   
        \node at (0.52,0) {$\cdot$};   
        \node at (1.52,0) {$\cdot$};
        \node at (2.52,0) {$\cdot$};
        \node at (3.52,0) {$\cdot$};   
        \node at (4.52,0) {$\cdot$};
        \node at (5.52,0) {$\cdot$};
        \node at (6.52,0) {$\cdot$};   
        \node at (7.52,0) {$\cdot$};
        \node at (8.52,0) {$\cdot$};
        \node at (9.52,0) {$\cdot$};

        \node at (4.5,4) {$\mu_1$};
    \end{tikzpicture}$$
    $$\begin{tikzpicture}[scale=.4, baseline={(0,-0.08)}]
        \FILL{10}
        \CAP{3}{4}{1}
        \CAP{6}{7}{1}
        \CAP{2}{5}{.6}
        \CAP{1}{8}{.5}
        \dCAP{0}{9}{.5}
        \CUP{3}{4}{1}
        \CUP{6}{7}{1}
        \CUP{2}{5}{.6}
        \CUP{1}{8}{.5}
        \dCUP{0}{9}{.5}
        \node at (-0.48,0) {$\cdot$};   
        \node at (0.52,0) {$\cdot$};   
        \node at (1.52,0) {$\cdot$};
        \node at (2.52,0) {$\cdot$};
        \node at (3.52,0) {$\cdot$};   
        \node at (4.52,0) {$\cdot$};
        \node at (5.52,0) {$\cdot$};
        \node at (6.52,0) {$\cdot$};   
        \node at (7.52,0) {$\cdot$};
        \node at (8.52,0) {$\cdot$};
        \node at (9.52,0) {$\cdot$};

        \node at (4.5,-4) {$\mu_2$};
    \end{tikzpicture} \qquad \cdots \qquad \begin{tikzpicture}[scale=.4, baseline={(0,-0.08)}]
        \FILL{10}
        \CAP{3}{4}{1}
        \CAP{5}{6}{1}
        \CAP{2}{7}{.5}
        \CAP{1}{8}{.5}
        \RLINE{0}{2}
        \CUP{3}{4}{1}
        \CUP{5}{6}{1}
        \CUP{2}{7}{.5}
        \CUP{1}{8}{.5}
        \RLINE{9}{2}
        \node at (-0.48,0) {$\cdot$};   
        \node at (0.52,0) {$\cdot$};   
        \node at (1.52,0) {$\cdot$};
        \node at (2.52,0) {$\cdot$};
        \node at (3.52,0) {$\cdot$};   
        \node at (4.52,0) {$\cdot$};
        \node at (5.52,0) {$\cdot$};
        \node at (6.52,0) {$\cdot$};   
        \node at (7.52,0) {$\cdot$};
        \node at (8.52,0) {$\cdot$};
        \node at (9.52,0) {$\cdot$};

        \node at (4.5,-4) {$\mu_k$};
    \end{tikzpicture}$$
    if $C_{k}$ is a dotted circle. We can compute for the undotted version
    $$
    \omega(y_\kappa^\lambda x_\kappa^\lambda) \ = \ \begin{tikzpicture}[scale=.4, baseline={(0,1.1)}]
        \FILL{10}
        \CAP{4}{5}{1}
        \CAP{3}{6}{.6}
        \CAP{2}{7}{.5}
        \CAP{1}{8}{.5}
        \CAP{0}{9}{.5}
        \CUP{3}{4}{1}
        \CUP{5}{6}{1}
        \CUP{2}{7}{.5}
        \CUP{1}{8}{.5}
        \CUP{0}{9}{.5}
        \node at (-0.48,0) {$\cdot$};   
        \node at (0.52,0) {$\cdot$};   
        \node at (1.52,0) {$\cdot$};
        \node at (2.52,0) {$\cdot$};
        \node at (3.52,0) {$\cdot$};   
        \node at (4.52,0) {$\cdot$};
        \node at (5.52,0) {$\cdot$};
        \node at (6.52,0) {$\cdot$};   
        \node at (7.52,0) {$\cdot$};
        \node at (8.52,0) {$\cdot$};
        \node at (9.52,0) {$\cdot$};

        \begin{scope}[yshift=6cm]
        \FILL{10}
        \CUP{4}{5}{1}
        \CUP{3}{6}{.6}
        \CUP{2}{7}{.5}
        \CUP{1}{8}{.5}
        \CUP{0}{9}{.5}
        \CAP{3}{4}{1}
        \CAP{5}{6}{1}
        \CAP{2}{7}{.5}
        \CAP{1}{8}{.5}
        \CAP{0}{9}{.5}
        \node at (-0.48,0) {$\cdot$};   
        \node at (0.52,0) {$\cdot$};   
        \node at (1.52,0) {$\cdot$};
        \node at (2.52,0) {$\cdot$};
        \node at (3.52,0) {$\cdot$};   
        \node at (4.52,0) {$\cdot$};
        \node at (5.52,0) {$\cdot$};
        \node at (6.52,0) {$\cdot$};   
        \node at (7.52,0) {$\cdot$};
        \node at (8.52,0) {$\cdot$};
        \node at (9.52,0) {$\cdot$};
        \end{scope}

    \end{tikzpicture} \ = \ \begin{tikzpicture}[scale=.4, baseline={(0,1.1)}]
        \FILL{10}
        \UPLINE{0}{3}
        \UPLINE{1}{3}
        \UPLINE{2}{3}
        \UPLINE{3}{3}
        \CAP{4}{5}{1}
        \UPLINE{6}{3}
        \UPLINE{7}{3}
        \UPLINE{8}{3}
        \UPLINE{9}{3}
        \CUP{3}{4}{1}
        \CUP{5}{6}{1}
        \CUP{2}{7}{.5}
        \CUP{1}{8}{.5}
        \CUP{0}{9}{.5}
        \node at (-0.48,0) {$\cdot$};   
        \node at (0.52,0) {$\cdot$};   
        \node at (1.52,0) {$\cdot$};
        \node at (2.52,0) {$\cdot$};
        \node at (3.52,0) {$\cdot$};   
        \node at (4.52,0) {$\cdot$};
        \node at (5.52,0) {$\cdot$};
        \node at (6.52,0) {$\cdot$};   
        \node at (7.52,0) {$\cdot$};
        \node at (8.52,0) {$\cdot$};
        \node at (9.52,0) {$\cdot$};

        \begin{scope}[yshift=6cm]
        \FILL{10}
        \DOWNLINE{0}{3}
        \DOWNLINE{1}{3}
        \DOWNLINE{2}{3}
        \DOWNLINE{3}{3}
        \CUP{4}{5}{1}
        \DOWNLINE{6}{3}
        \DOWNLINE{7}{3}
        \DOWNLINE{8}{3}
        \DOWNLINE{9}{3}
        \CAP{3}{4}{1}
        \CAP{5}{6}{1}
        \CAP{2}{7}{.5}
        \CAP{1}{8}{.5}
        \CAP{0}{9}{.5}
        \node at (-0.48,0) {$\cdot$};   
        \node at (0.52,0) {$\cdot$};   
        \node at (1.52,0) {$\cdot$};
        \node at (2.52,0) {$\cdot$};
        \node at (3.52,0) {$\cdot$};   
        \node at (4.52,0) {$\cdot$};
        \node at (5.52,0) {$\cdot$};
        \node at (6.52,0) {$\cdot$};   
        \node at (7.52,0) {$\cdot$};
        \node at (8.52,0) {$\cdot$};
        \node at (9.52,0) {$\cdot$};
        \end{scope}

    \end{tikzpicture}$$
    \vspace{1cm}
    $$\qquad \qquad = \begin{tikzpicture}[scale=.4, baseline={(0,-0.08)}]
        \FILL{10}
        \CAP{3}{4}{1}
        \CAP{5}{6}{1}
        \CAP{2}{7}{.5}
        \CAP{1}{8}{.5}
        \CAP{0}{9}{.5}
        \CUP{3}{4}{1}
        \CUP{5}{6}{1}
        \CUP{2}{7}{.5}
        \CUP{1}{8}{.5}
        \CUP{0}{9}{.5}
        \node at (-0.48,0) {$\cdot$};   
        \node at (0.52,0) {$\cdot$};   
        \node at (1.52,0) {$\cdot$};
        \node at (2.52,0) {$\cdot$};
        \node at (3.52,0) {$\cdot$};   
        \node at (4.52,0) {$\cdot$};
        \node at (5.52,0) {$\cdot$};
        \node at (6.52,0) {$\cdot$};   
        \node at (7.52,0) {$\cdot$};
        \node at (8.52,0) {$\cdot$};
        \node at (9.52,0) {$\cdot$};
        \begin{scope}[xshift=3cm]
            \Bweight{1,0}
        \end{scope}
    \end{tikzpicture} \ + \ \begin{tikzpicture}[scale=.4, baseline={(0,-0.08)}]
        \FILL{10}
        \CAP{3}{4}{1}
        \CAP{5}{6}{1}
        \CAP{2}{7}{.5}
        \CAP{1}{8}{.5}
        \CAP{0}{9}{.5}
        \CUP{3}{4}{1}
        \CUP{5}{6}{1}
        \CUP{2}{7}{.5}
        \CUP{1}{8}{.5}
        \CUP{0}{9}{.5}
        \node at (-0.48,0) {$\cdot$};   
        \node at (0.52,0) {$\cdot$};   
        \node at (1.52,0) {$\cdot$};
        \node at (2.52,0) {$\cdot$};
        \node at (3.52,0) {$\cdot$};   
        \node at (4.52,0) {$\cdot$};
        \node at (5.52,0) {$\cdot$};
        \node at (6.52,0) {$\cdot$};   
        \node at (7.52,0) {$\cdot$};
        \node at (8.52,0) {$\cdot$};
        \node at (9.52,0) {$\cdot$};
        \begin{scope}[xshift=5cm]
            \Bweight{1,0}
        \end{scope}
    \end{tikzpicture}$$
    so that
    $$\omega(y_\kappa^\lambda x_\kappa^\lambda) = \underline{\lambda}\mu_1\overline{\lambda} + \underline{\lambda}\mu_2\overline{\lambda} $$
    and with similar computations, we find also that
    \begin{align*}
        &\omega(x_\lambda^{\mu_1} y_\lambda^{\mu_1}) = \underline{\lambda}\mu_1\overline{\lambda} + \underline{\lambda}\mu_3\overline{\lambda} \\
        &\omega(x_\lambda^{\mu_i} y_\lambda^{\mu_i}) = \underline{\lambda}\mu_i\overline{\lambda} + \underline{\lambda}\mu_{i+1}\overline{\lambda} \text{ for } 2 \leq i < k \\
        &\omega(x_\lambda^{\mu_k} y_\lambda^{\mu_k}) = \underline{\lambda}\mu_k\overline{\lambda} 
    \end{align*}
    so we deduce that
    \begin{align*}
        \omega(y_\kappa^\lambda x_\kappa^\lambda) &= \omega(x_\lambda^{\mu_1} y_\lambda^{\mu_1}) + \omega(x_\lambda^{\mu_2} y_\lambda^{\mu_2}) - 2\omega(x_\lambda^{\mu_3} y_\lambda^{\mu_3}) + \cdots + (-1)^{k}2 \omega(x_\lambda^{\mu_{k}} y_\lambda^{\mu_{k}}) \\
        &= \sum_{i=1}^{k} d_\kappa^{\mu_i}(\lambda)\omega(x_\lambda^{\mu_i} y_\lambda^{\mu_i})
    \end{align*}
    and hence
    $$\omega\left( y_\kappa^\lambda x_\kappa^\lambda - \sum_{i=1}^{\text{def}(\lambda)} d_\kappa^{\mu_i}(\lambda) x_\lambda^{\mu_i} y_\lambda^{\mu_i} \right)$$
    because $d_\kappa^{\mu_i}(\lambda) = 0$ for all $k+1 \leq i \leq \text{def}(\lambda)$, if they exist. The computations give the same result for the dotted version. Now, if there is no circle in $e_\lambda$ appearing in $e_\kappa$ enclosing $D$, then $\kappa$ and $\lambda$ are of the form of one of the nine following subcases
    $$
$$
    In the first and second subcases, we have if they exist, that $d_\kappa^{\mu_i}(\lambda) = 0$ for all $1 \leq i \leq \text{def}(\lambda)$, and we compute
    $$\omega(y_\kappa^\lambda x_\kappa^\lambda) \ = \ %
 \ = \ 0$$
    as wanted. For the third subcase, if we denote $C_1$ and $C_2$ the circles from left to right on the picture of $e_\lambda$, giving rise to $\mu_1$ and $\mu_2$, we compute
    \[
\begin{aligned}
\omega(y_\kappa^\lambda x_\kappa^\lambda)
&=
%
$$
    so that
    $$\omega\left( y_\kappa^\lambda x_\kappa^\lambda - \sum_{i=1}^{\text{def}(\lambda)} d_\kappa^{\mu_i}(\lambda) x_\lambda^{\mu_i} y_\lambda^{\mu_i} \right)$$
    because $d_\kappa^{\mu_i}(\lambda) = 0$ for all $3 \leq i \leq \text{def}(\lambda)$, if they exist. The fourth subcase is analogous. Then, for the fifth subcase, if we denote $C_1$ the circle on the picture of $e_\lambda$, giving rise to $\mu_1$, we compute
    $$\omega(y_\kappa^\lambda x_\kappa^\lambda) \ = %
 $$
    so that
    $$\omega\left( y_\kappa^\lambda x_\kappa^\lambda - \sum_{i=1}^{\text{def}(\lambda)} d_\kappa^{\mu_i}(\lambda) x_\lambda^{\mu_i} y_\lambda^{\mu_i} \right)$$
    because $d_\kappa^{\mu_i}(\lambda) = 0$ for all $2 \leq i \leq \text{def}(\lambda)$, if they exist. The sixth to eighth subcases are analogous. Finally, for the ninth subcase, if we denote $C_1$ and $C_2$ the inner circle and the outer circle respectively in $e_\lambda$, giving rise to $\mu_1$ and $\mu_2$, we compute
    \[
\begin{aligned}
\omega(y_\kappa^\lambda x_\kappa^\lambda)
&=
%
 $$
    so that
    $$\omega\left( y_\kappa^\lambda x_\kappa^\lambda - \sum_{i=1}^{\text{def}(\lambda)} d_\kappa^{\mu_i}(\lambda) x_\lambda^{\mu_i} y_\lambda^{\mu_i} \right)$$
    because $c_\kappa^{\mu_i}(\lambda) = 0$ for all $3 \leq i \leq \text{def}(\lambda)$, if they exist (and $d_\kappa^{\mu_2}(\lambda) = -2$). The tenth subcase is analogous. These are the only subcases where the coefficient $d_\kappa^{\mu_i}(\lambda) = -2$ occurs.
\end{proof}

Now, if we denote by $I_n$ the two-sided ideal of $\Bbbk Q_n$ generated by the quadratic relations from Theorem~\ref{thm:5.15}, we obtain a surjective graded algebra homomorphism
$$\overline{\omega} : \Bbbk Q_n/I_n \twoheadrightarrow \D_n$$
but it still remains to show that this is an isomorphism, so that we found all the generating quadratic relations. \\
Observe that $\overline{\omega}$ is bijective in degree $0$ and $1$, and since $\D_n$ is Koszul, and thus quadratic, we only have to show that $\overline{\omega}$ is bijective in degree $2$.

\begin{remark}\label{rmk:invD}
    Observe that $\Bbbk Q_n/I_n$ admits an involution
    $$(\bullet)^\ast : \Bbbk Q_n/I_n \longrightarrow \Bbbk Q_n/I_n $$
    defined on generators by $(e_\lambda)^\ast = e_\lambda$, $(x_\lambda^\mu)^\ast = y_\lambda^\mu$ and $(y_\lambda^\mu)^\ast = x_\lambda^\mu$. Note also that $\D_n$ admits an involution
    $$(\bullet)^\ast : \D_n \longrightarrow \D_n$$
    defined by $(\underline{\alpha}\lambda\overline{\beta})^\ast = \underline{\beta}\lambda\overline{\alpha}$. Note that these two involutions are compatible under $\overline{\omega}$.
\end{remark}

\begin{lemma}\label{lem:5.17}
    Let $\lambda, \mu \in \Delta_n^0$, then we have
    $$
    \text{dim}_\Bbbk \left( e_\lambda ( \Bbbk Q_n/I_n )_2 e_\mu \right) \leq
    \begin{cases}
        1  &\text{ if } \lambda \neq \mu,\\
        \text{def}(\lambda) &\text{ if } \lambda = \mu.
    \end{cases}
    $$
\end{lemma}

\begin{proof}
    If $\lambda \neq \mu$, suppose by contradiction that
    $$\text{dim}_\Bbbk \left( e_\lambda ( \Bbbk Q_n/I_n )_2 e_\mu \right) > 1$$
    Then there exist two paths $p$ and $q$ of length $2$ from $e_\lambda$ to $e_\mu$ such that $[p]$ and $[q]$ are linearly independant in $\Bbbk Q_n/I_n$. This means that $p-q \not \in I_n$, but $p$ and $q$ define a square, so we should have $p-q \in I_n$ thanks to Theorem~\ref{thm:5.15}. This contradiction tells us that
    $$\text{dim}_\Bbbk \left( e_\lambda ( \Bbbk Q_n/I_n )_2 e_\mu \right) \leq 1$$
    Now if $\lambda = \mu$, the relations at vertex $e_\lambda$ in $I_n$ tell us that $e_\lambda( \Bbbk Q_n/I_n )_2 e_\lambda$ is spanned by the $2$-cycles of the form $x_\lambda^{\mu_i} y_\lambda^{\mu_i}$ for $1 \leq i \leq \text{def}(\lambda)$. So we have that
    $$\text{dim}_\Bbbk \left( e_\lambda ( \Bbbk Q_n/I_n )_2 e_\lambda \right) \leq \text{def}(\lambda)$$
    as wanted.
\end{proof}

\begin{lemma}\label{lem:5.18}
    Let $\lambda \neq \mu \in \Delta_n^0$, then
    $$\text{dim}_\Bbbk \left( e_\lambda ( \Bbbk Q_n/I_n )_2 e_\mu \right) = \text{dim}_\Bbbk \left( e_\lambda ( \D_n )_2 e_\mu \right).$$
\end{lemma}

\begin{proof}
    Since $\overline{\omega}$ is surjective, and thanks to Lemma~\ref{lem:5.17}, we have that
    $$1 \geq \text{dim}_\Bbbk \left( e_\lambda ( \Bbbk Q_n/I_n )_2 e_\mu \right) \geq \text{dim}_\Bbbk \left( e_\lambda ( \D_n )_2 e_\mu \right)$$
    Observe that if 
    $$\text{dim}_\Bbbk \left( e_\lambda ( \Bbbk Q_n/I_n )_2 e_\mu \right) = 0$$
    we're done. So suppose that
    $$\text{dim}_\Bbbk \left( e_\lambda ( \Bbbk Q_n/I_n )_2 e_\mu \right) = 1$$
    and then there exists a length $2$ path $p$ from $e_\lambda$ to $e_\mu$ such that $[p] \neq 0$ in $e_\lambda ( \Bbbk Q_n/I_n )_2 e_\mu$. \\
    Either $p$ appears in a square, and then computations from the proof of Theorem~\ref{thm:5.15} tell us that $\omega(p) \neq 0$ so that 
    $$\text{dim}_\Bbbk \left( e_\lambda ( \D_n )_2 e_\mu \right) = 1 = \text{dim}_\Bbbk \left( e_\lambda ( \Bbbk Q_n/I_n )_2 e_\mu \right)$$
    Otherwise, thanks to Theorem~\ref{thm:5.15}, we know that $p$ is a path from one of the figures $\textit{(e)}$, $\textit{(f)}$, $\textit{(g)}$, $\textit{(h)}$ or $\textit{(i)}$ from Lemma~\ref{lem:5.11}. If $p$ is the ascending path from figure $\textit{(e)}$, then we compute
    $$\omega(p) \ = \ \begin{tikzpicture}[scale=.5, baseline={(0,.68)}]
        \Bweight{0,1,0,0}
        \CAP{0}{1}{1}
        \CUP{0}{1}{1}
        \dCUP{2}{3}{1}
        \TRLINE{2}{1}
        \TRLINE{3}{1}
        \node at (-0.80,0) {$\cdots$};   
        \node at (0.52,0) {$\cdot$};   
        \node at (1.52,0) {$\cdot$};
        \node at (2.52,0) {$\cdot$};
        \node at (4.00,0) {$\cdots$};

        \begin{scope}[yshift=3cm]
            \Bweight{1,0,0,0}
            \CUP{0}{1}{1}
            \RLINE{2}{1}
            \RLINE{3}{1}
            \TRLINE{1}{1}
            \TpdLINE{0}{1}
            \node at (-0.80,0) {$\cdots$};   
            \node at (0.52,0) {$\cdot$};   
            \node at (1.52,0) {$\cdot$};
            \node at (2.52,0) {$\cdot$};
            \node at (4.00,0) {$\cdots$};
        \end{scope}
    \end{tikzpicture} \ = \  \begin{tikzpicture}[scale=.5, baseline={(0,-0.1)}]
            \Bweight{1,0,0,0}
            \CUP{0}{1}{1}
            \dCUP{2}{3}{1}
            \TRLINE{1}{1}
            \TpdLINE{0}{1}
            \TRLINE{2}{1}
            \TRLINE{3}{1}
            \node at (-0.80,0) {$\cdots$};   
            \node at (0.52,0) {$\cdot$};   
            \node at (1.52,0) {$\cdot$};
            \node at (2.52,0) {$\cdot$};
            \node at (4.00,0) {$\cdots$};
    \end{tikzpicture}$$
    If $p$ is the ascending path from figure $\textit{(f)}$, then we compute
    $$\omega(p) \ = \ \begin{tikzpicture}[scale=.5, baseline={(0,.68)}]
        \Bweight{1,1,0,0}
        \dCAP{0}{1}{1}
        \dCUP{0}{1}{1}
        \dCUP{2}{3}{1}
        \TRLINE{2}{1}
        \TRLINE{3}{1}
        \node at (-0.80,0) {$\cdots$};   
        \node at (0.52,0) {$\cdot$};   
        \node at (1.52,0) {$\cdot$};
        \node at (2.52,0) {$\cdot$};
        \node at (4.00,0) {$\cdots$};

        \begin{scope}[yshift=3cm]
            \Bweight{0,0,0,0}
            \dCUP{0}{1}{1}
            \RLINE{2}{1}
            \RLINE{3}{1}
            \TRLINE{1}{1}
            \TRLINE{0}{1}
            \node at (-0.80,0) {$\cdots$};   
            \node at (0.52,0) {$\cdot$};   
            \node at (1.52,0) {$\cdot$};
            \node at (2.52,0) {$\cdot$};
            \node at (4.00,0) {$\cdots$};
        \end{scope}
    \end{tikzpicture} \ = \  \begin{tikzpicture}[scale=.5, baseline={(0,-0.1)}]
            \Bweight{0,0,0,0}
            \dCUP{0}{1}{1}
            \dCUP{2}{3}{1}
            \TRLINE{1}{1}
            \TRLINE{0}{1}
            \TRLINE{2}{1}
            \TRLINE{3}{1}
            \node at (-0.80,0) {$\cdots$};   
            \node at (0.52,0) {$\cdot$};   
            \node at (1.52,0) {$\cdot$};
            \node at (2.52,0) {$\cdot$};
            \node at (4.00,0) {$\cdots$};
    \end{tikzpicture}$$
    If $p$ is the ascending path from figure $\textit{(g)}$, then we compute
    $$\omega(p) \ = \ \begin{tikzpicture}[scale=.5, baseline={(0,.68)}]
        \Bweight{0,1,0,1,0}
        \CAP{0}{1}{1}
        \CAP{2}{3}{1}
        \CUP{2}{3}{1}
        \CUP{1}{4}{.7}
        \BRLINE{0}{1}
        \TRLINE{4}{1}
        \FILL{5}
        \node at (-0.80,0) {$\cdots$};   
        \node at (0.52,0) {$\cdot$};   
        \node at (1.52,0) {$\cdot$};
        \node at (2.52,0) {$\cdot$};
        \node at (3.52,0) {$\cdot$};
        \node at (5.00,0) {$\cdots$};

        \begin{scope}[yshift=3cm]
            \Bweight{0,1,1,0,0}
            \CUP{0}{1}{1}
            \CUP{2}{3}{1}
            \BRLINE{4}{1}
            \CAP{0}{1}{1}
            \TpdLINE{2}{1}
            \TRLINE{3}{1}
            \TRLINE{4}{1}
            \FILL{4}
            \node at (-0.80,0) {$\cdots$};   
            \node at (0.52,0) {$\cdot$};   
            \node at (1.52,0) {$\cdot$};
            \node at (2.52,0) {$\cdot$};
            \node at (3.52,0) {$\cdot$};
            \node at (5.00,0) {$\cdots$};
        \end{scope}
    \end{tikzpicture} \ = \  \begin{tikzpicture}[scale=.5, baseline={(0,-0.1)}]
            \Bweight{0,1,1,0,0}
            \CUP{2}{3}{1}
            \CUP{1}{4}{.7}
            \BRLINE{0}{1}
            \CAP{0}{1}{1}
            \TpdLINE{2}{1}
            \TRLINE{3}{1}
            \TRLINE{4}{1}
            \FILL{4}
            \node at (-0.80,0) {$\cdots$};   
            \node at (0.52,0) {$\cdot$};   
            \node at (1.52,0) {$\cdot$};
            \node at (2.52,0) {$\cdot$};
            \node at (3.52,0) {$\cdot$};
            \node at (5.00,0) {$\cdots$};
    \end{tikzpicture}$$
    If $p$ is the ascending path from figure $\textit{(h)}$, then we compute
    $$\omega(p) \ = \ \begin{tikzpicture}[scale=.5, baseline={(0,.68)}]
        \Bweight{1,1,0,1,0}
        \dCAP{0}{1}{1}
        \CAP{2}{3}{1}
        \CUP{2}{3}{1}
        \CUP{1}{4}{.7}
        \dBRLINE{0}{1}
        \TRLINE{4}{1}
        \FILL{5}
        \node at (-0.80,0) {$\cdots$};   
        \node at (0.52,0) {$\cdot$};   
        \node at (1.52,0) {$\cdot$};
        \node at (2.52,0) {$\cdot$};
        \node at (3.52,0) {$\cdot$};
        \node at (5.00,0) {$\cdots$};

        \begin{scope}[yshift=3cm]
            \Bweight{1,1,1,0,0}
            \dCUP{0}{1}{1}
            \CUP{2}{3}{1}
            \BRLINE{4}{1}
            \dCAP{0}{1}{1}
            \TpdLINE{2}{1}
            \TRLINE{3}{1}
            \TRLINE{4}{1}
            \FILL{4}
            \node at (-0.80,0) {$\cdots$};   
            \node at (0.52,0) {$\cdot$};   
            \node at (1.52,0) {$\cdot$};
            \node at (2.52,0) {$\cdot$};
            \node at (3.52,0) {$\cdot$};
            \node at (5.00,0) {$\cdots$};
        \end{scope}
    \end{tikzpicture} \ = \  \begin{tikzpicture}[scale=.5, baseline={(0,-0.1)}]
            \Bweight{1,1,1,0,0}
            \CUP{2}{3}{1}
            \CUP{1}{4}{.7}
            \dBRLINE{0}{1}
            \dCAP{0}{1}{1}
            \TpdLINE{2}{1}
            \TRLINE{3}{1}
            \TRLINE{4}{1}
            \FILL{4}
            \node at (-0.80,0) {$\cdots$};   
            \node at (0.52,0) {$\cdot$};   
            \node at (1.52,0) {$\cdot$};
            \node at (2.52,0) {$\cdot$};
            \node at (3.52,0) {$\cdot$};
            \node at (5.00,0) {$\cdots$};
    \end{tikzpicture}$$
    And if $p$ is the path from left to right from figure $\textit{(i)}$, then we compute
    $$\omega(p) \ = \ \begin{tikzpicture}[scale=.5, baseline={(0,.68)}]
        \Bweight{1,0,1,0,1,0}
        \CAP{1}{2}{1}
        \CUP{1}{2}{1}
        \CUP{0}{3}{.7}
        \TpdLINE{0}{1}
        \CAP{3}{4}{1}
        \BpdLINE{4}{1}
        \RLINE{5}{1}
        \node at (-0.80,0) {$\cdots$};   
        \node at (0.52,0) {$\cdot$};   
        \node at (1.52,0) {$\cdot$};
        \node at (2.52,0) {$\cdot$};
        \node at (3.52,0) {$\cdot$};
        \node at (4.52,0) {$\cdot$};
        \node at (6.00,0) {$\cdots$};

        \begin{scope}[yshift=3cm]
            \Bweight{1,0,1,0,1,0}
            \CUP{1}{2}{1}
            \CUP{3}{4}{1}
            \CAP{3}{4}{1}
            \BRLINE{5}{1}
            \pdLINE{0}{1}
            \TRLINE{1}{1}
            \CAP{2}{5}{.7}
            \node at (-0.80,0) {$\cdots$};   
            \node at (0.52,0) {$\cdot$};   
            \node at (1.52,0) {$\cdot$};
            \node at (2.52,0) {$\cdot$};
            \node at (3.52,0) {$\cdot$};
        \node at (4.52,0) {$\cdot$};
        \node at (6.00,0) {$\cdots$};
        \end{scope}
    \end{tikzpicture} \ = \  \begin{tikzpicture}[scale=.5, baseline={(0,-0.1)}]
            \Bweight{1,0,1,0,1,0}
            \CUP{1}{2}{1}
            \CUP{0}{3}{.7}
            \CAP{3}{4}{1}
            \CAP{2}{5}{.7}
            \TpdLINE{0}{1}
            \TRLINE{1}{1}
            \BpdLINE{4}{1}
            \BRLINE{5}{1}
            \node at (-0.80,0) {$\cdots$};   
            \node at (0.52,0) {$\cdot$};   
            \node at (1.52,0) {$\cdot$};
            \node at (2.52,0) {$\cdot$};
            \node at (3.52,0) {$\cdot$};
        \node at (4.52,0) {$\cdot$};
        \node at (6.00,0) {$\cdots$};
    \end{tikzpicture}$$
    Thanks to the involution from Remark~\ref{rmk:invD}, we just had to verify diagrammatically that one direction of each of these figures is non-zero.
\end{proof}

\begin{theorem}\label{thm:5.19}
    The graded algebra homomorphism $\overline{\omega} : \Bbbk Q_n/I_n \twoheadrightarrow \D_n$ is an isomorphism.
\end{theorem}

\begin{proof}
    As we said before, we just need to verify that $\overline{\omega}$ is a bijection in degree $2$. And since $\overline{\omega}$ is surjective, we just need to verify that
    $$\text{dim}_\Bbbk \left( e_\lambda ( \Bbbk Q_n/I_n )_2 e_\mu \right) = \text{dim}_\Bbbk \left( e_\lambda ( \D_n )_2 e_\mu \right)$$
    for every $\lambda, \mu \in \Delta_n^0$. We know from Lemma~\ref{lem:5.18} that it is true if $\lambda \neq \mu$. Now if $\lambda = \mu$, we have from Lemma~\ref{lem:5.17}, and by surjectivity of $\overline{\omega}$ that
    $$\text{def}(\lambda) \geq \text{dim}_\Bbbk \left( e_\lambda ( \Bbbk Q_n/I_n )_2 e_\mu \right) \geq \text{dim}_\Bbbk \left( e_\lambda ( \D_n )_2 e_\mu \right) = \text{def}(\lambda)$$
    where the last equality comes from the fact that for each circle in $e_\lambda$, you can either orient it clockwise or counterclockwise, and each clockwise circle has degree $2$.
\end{proof}

\section{Intrinsic formality of $\D_4$}\label{sec:7}

The goal of this section is to show that $\D_4$ is intrinsically formal, that is $\D_4$ does not admit any non-trivial $A_\infty$-deformation. We will first recall these concepts, and then explain and use a method introduced by Barmeier and Wang in~\cite{BarmeierWang2020} to compute the Hochschild cohomology groups of $\D_4$, that are closely related to its $A_\infty$-deformations.

\subsection{Quick recall}\label{qr}

\begin{definition}
    An $A_\infty$-algebra over $\Bbbk$ is a graded vector space $A = \bigoplus_{k \in \Z} A^k$ endowed with graded maps
    $$m_n : A^{\otimes n} \to A, \ n \geq 1$$
    of degree $2-n$, satisfying the following relations
    $$m_1m_1 = 0, \qquad m_1m_2 = m_2(m_1 \otimes \textbf{1} + \textbf{1} \otimes m_1)$$
    $$m_2(\textbf{1} \otimes m_2 - m_2 \otimes \textbf{1}) = m_1m_3 + m_3(m_1 \otimes \textbf{1} \otimes \textbf{1} + \textbf{1} \otimes m_1 \otimes \textbf{1} + \textbf{1} \otimes \textbf{1} \otimes m_1)$$
    and more generally, for $n \geq 1$
    $$\sum_{r+s+t=n} (-1)^{r+st}m_{r+t+1}(\textbf{1}^{\otimes r} \otimes m_s \otimes \textbf{1}^{\otimes t}) = 0$$
    where $\textbf{1}$ denotes the identity map of $A$.
\end{definition}

\begin{definition}
    Let  $A = \bigoplus_{k \in \Z} A^k$ be an $A_\infty$-algebra. Then we say that $A$ is intrinsically formal if any $A_\infty$-deformations of $A$ is $A_\infty$-quasi-isomorphic to the underlying graded algebra $A$.
\end{definition}

It turns out that that the $A_\infty$-deformations of an algebra are closely related to the Hochschild cohomology groups of it. The following theorems are particularly interesting in our setting.

\begin{theorem}[{\cite[Thm. 4.7]{SeidelThomas2001}}]\label{intformal}
    Let  $A = \bigoplus_{k \in \Z} A^k$ be an $A_\infty$-algebra. If $\text{HH}^2_q(A,A) = 0$ for each $q > 0$, then $A$ is intrinsically formal.
\end{theorem}

\begin{theorem}[{\cite[\S 3.5]{Keller2003}}]\label{intformaltool}
    Let $A = \Bbbk Q/I$ be a Koszul algebra. We view $A$ as bigraded by assigning any arrow $a \in Q$ the bidegree $(0,1)$, we view also the Koszul dual $A^! \cong \Bbbk \overline{Q}/(I_2^\perp)$ as bigraded by assigning any arrow $\overline{a} \in \overline{Q}$ the bidegree $(1,-1)$. Then we have
    $$\text{HH}_q^p(A,A) \cong \text{HH}_q^p(A^!,A^!)$$
    for all $p,q \in \Z$.
\end{theorem}

In the following subsections, we will compute the Hochschild cohomology groups $\text{HH}^2_q(\D_4^!, \D_4^!)$ for $q > 0$ in order to apply Theorems~\ref{intformal} and~\ref{intformaltool}.

\subsection{Method to compute Hochschild cohomology groups}\label{mhc}

The Hochschild cohomology groups can be very difficult to compute, but fortunately in our setting, there is a method introduced by Barmeier and Wang in~\cite{BarmeierWang2020} that is very helpful to compute the groups we are interested in. This method involves reduction systems in the sense of~\cite{Bergman1978}, and we explain it very quickly in this subsection.

\begin{theorem}[{\cite[Thm. 7.1]{BarmeierWang2020}}]
Let $Q$ be any finite quiver, let $I\subset kQ$ be any two-sided ideal and let $R$ be any reduction system satisfying the diamond condition for $I$. Then there is an equivalence of deformation problems between
\begin{enumerate}
    \item deformations of the associative algebra $A=kQ/I$,
    \item deformations of the ideal $I$,
    \item deformations of the reduction system $R$.
\end{enumerate}
\end{theorem}

\begin{definition}\label{defmethod}
A first-order deformation of a reduction system \(R = \{(s,\varphi_s)\}_{s\in S}\) is a reduction system
\[
\widetilde{R} = \{(s,\varphi_s + \widetilde{\varphi}_s\, t)\}_{s\in S}
\]
over \(\Bbbk Q \otimes_\Bbbk \Bbbk [t]/(t^2)\), where \(\widetilde{\varphi}_s \in \Bbbk \mathrm{Irr}_S\), such that \(\widetilde{R}\) is reduction-unique. \\
Equivalently, \(\widetilde{\varphi}\) defines a morphism
\[
\widetilde{\varphi} \in \mathrm{Hom}_{\Bbbk Q_0^e}(\Bbbk S, \Bbbk \mathrm{Irr}_S), 
\quad s \mapsto \widetilde{\varphi}_s.
\]
We say that \(\widetilde{\varphi}\) defines a first-order deformation of \(R\) if it is a \(2\)-cocycle in the cochain complex associated to the Bardzell--Chouhy--Solotar resolution from~\cite{ChouhrSolotar2015}. In other words this means that $\tilde{\varphi}$ is a \(2\)-cocycle if and only if for any overlap ambiguity $pqr$ of length $3$, we have 
$$(p \star q) \star r  = p \star (q \star r) \text{ mod } t^2$$
where
$$u \star v = \left\{
    \begin{array}{cc}
        uv & \mbox{if } uv \not \in S \\
        \varphi_{uv} + t\tilde{\varphi}_{uv} & \mbox{if } uv \in S
    \end{array} \right.$$
Two first-order deformations $\widetilde{R} = \{(s,\varphi_s + \widetilde{\varphi}_s t)\}_{s\in S}$ and $\widetilde{R}' = \{(s,\varphi_s + \widetilde{\varphi}'_s t)\}_{s\in S}$ are equivalent if there exists
\[
\psi \in \mathrm{Hom}_{\Bbbk Q_0^e}(\Bbbk Q_1, \Bbbk \mathrm{Irr}_S)
\]
such that for all \(s \in S\),
\[
\widetilde{\varphi}_s - \widetilde{\varphi}'_s = \delta(\psi)(s).
\]
Here
\[
\delta : \mathrm{Hom}_{\Bbbk Q_0^e}(\Bbbk Q_1, \Bbbk \mathrm{Irr}_S)
\longrightarrow
\mathrm{Hom}_{\Bbbk Q_0^e}(\Bbbk S, \Bbbk \mathrm{Irr}_S)
\]
is given by
\[
\delta(\psi)(s) = T(\psi)(s - \varphi_s),
\]
where \(T(\psi) : \Bbbk Q \to \Bbbk \mathrm{Irr}_S\) is defined on a path \(p = a_1 \cdots a_m\) by
\[
T(\psi)(p)
=
\sum_{i=1}^m
\sigma\pi(a_1 \cdots a_{i-1}\,\psi(a_i)\,a_{i+1} \cdots a_m).
\]
Here \(a_1,\ldots,a_m\) are arrows in \(Q\), \(\pi : \Bbbk Q \to \Bbbk Q/I\) is the natural projection, and \(\sigma : \Bbbk Q/I \to \Bbbk\mathrm{Irr}_S\) is the inverse of \(\pi|_{\mathrm{Irr}_S}\). In other words, \(\sigma\pi(p)\) is obtained by reducing \(p\) with respect to \(R\) until it becomes irreducible.
\end{definition}

\begin{corollary}[{\cite[Cor. 7.44]{BarmeierWang2020}}]\label{method}
Let $A=kQ/I$ and $R$ be any reduction system satisfying the diamond condition for $I$. Then $\mathrm{HH}^2(A,A)$ is isomorphic to the space of first-order deformations of $R$ modulo equivalence.
\end{corollary}

This method may seem somewhat unclear at this point, but it will become much clearer as we apply it several times in Section~\ref{sec:7.3}.

\subsection{Koszul dual}\label{sec:7.1}

We will not prove or disprove if $\D_n$ is intrinsically formal for $n \geq 5$ , but we can still compute the Koszul dual of $\D_n$ explicitely.

Recall that the Koszul dual of $\D_n$ is defined by
$$(\D_n)^{!} = T_\Bbbk (\D_n)_1 / (I_n)_2^{\perp}$$
and it is known that $(\D_n)^{!}$ is also Koszul, and that $((\D_n)^{!})^{!} \cong \D_n$. \\
Since we proved in Theorem~\ref{thm:5.19} that $\D_n \cong \Bbbk Q_n/I_n$ can be viewed as a quiver algebra with quadratic relations, it is well known that there is an isomorphism
$$\overline{\D}_n := (\D_n)^{!} \cong \Bbbk \overline{Q}_n/\overline{I}_n$$
where $\overline{Q}_n$ denotes the opposite quiver of $Q_n$, and $\overline{I}_n := (I_n)_2^{\perp}$ is defined to be the perpendicular ideal of $(I_n)_2$ with respect to the nondegenerate $\Bbbk$-bilinear pairing $\langle \bullet, \bullet \rangle : \Bbbk Q_n \times \Bbbk \overline{Q}_n \to \Bbbk$ which is uniquely determined by
$$\langle a_1 \cdots a_n, \overline{b}_n \cdots \overline{b}_1 \rangle = \delta_{a_1,b_1}\cdots\delta_{a_n,b_n}$$
for paths $a_1\cdots a_n, b_1 \cdots b_n$ in $Q_n$. Concretely, the quadratic relations for $(\D_n)^{!}$ are given by
$$\overline{I}_n = (I_n)_2^{\perp} = \{ p \in \Bbbk \overline{Q}_n \vert \langle (I_n)_2, p \rangle = 0 \} $$
We will write $\overline{x}_\lambda^\mu$ and $\overline{y}_\lambda^\mu$ for the Koszul duals of the generators $x_\lambda^\mu$ and $y_\lambda^\mu$ of $\D_n$. Therefore, $\overline{x}_\lambda^\mu$ is a descending arrow, and $\overline{y}_\lambda^\mu$ is an ascending arrow. \\
Now let's describe the ideal $\overline{I}_n$.

\begin{theorem}\label{thm:7.1}
    The ideal $\overline{I}_n$ is generated by the following quadratic relations

    \textbf{Dual monomial relations :} The paths from the top to the bottom, and from the bottom to the top of figures $\textit{(e)}$, $\textit{(f)}$, $\textit{(g)}$ and $\textit{(h)}$ from Lemma~\ref{lem:5.11}, and the paths from the left to the right, and from the right to the left of figure $\textit{(i)}$ from Lemma~\ref{lem:5.11}.

    \textbf{Dual parallel relations :} If there is more than one parallel path of length $2$ (excluding cycles), then the sum of all such parallel paths is in $\overline{I}_n$.

    \textbf{Dual relations at vertices :} Fix $\lambda \in \Delta_n^0$, and let $\overline{y}_{\kappa_j}^\lambda$ be the ascending arrows in $\overline{Q}_n$ ending at $e_\lambda$. Let also $\overline{y}_\lambda^\mu$ be any ascending arrow starting at $e_\lambda$. Then the following element belongs to $\overline{I}_n$
    $$\overline{y}_\lambda^\mu\overline{x}_\lambda^\mu + \sum_{j} d_{\kappa_j}^\mu(\lambda) \overline{x}_{\kappa_j}^\lambda \overline{y}_{\kappa_j}^\lambda $$
    where the coefficients $d_{\kappa_j}^\mu(\lambda)$ are the same as in Theorem~\ref{thm:5.15}.
\end{theorem}

\begin{proof}
    We need to show that the above relations are orthogonal to the relations from Theorem~\ref{thm:5.15}. The dual monomial relations are clearly orthogonal to the square and vertices relations, they're also orthogonal to the monomial relations because they correspond to different paths from Lemma~\ref{lem:5.11}. \\
    The dual parallel relations are clearly orthogonal to the monomials and vertices relations, they're also orthogonal to the square relations because the signs compensate. \\
    Finally, the dual relations at vertices are clearly orthogonal to the monomial and square relations, but it's less obvious for the relations at vertices. Fix $\lambda \in \Delta_n^0$, and let $\overline{y}_{\kappa_j}^\lambda$ be the ascending arrows in $\overline{Q}_n$ ending at $e_\lambda$, let also $\mu_i$ for $1 \leq i \leq \text{def}(\lambda)$ be the weights obtained from $\lambda$ by exchanging a $\dn \cdots \up$ or a $\up \cdots \up$ pair lying in a circle of $e_\lambda$. Then we have
\begin{align*}
    \left\langle
y^{\lambda}_{\kappa_k} x^{\lambda}_{\kappa_k}
-
\sum_i d^{\mu_i}_{\kappa_k}(\lambda) x^{\mu_i}_{\lambda} y^{\mu_i}_{\lambda},
\bar y^{\mu_l}_{\lambda} \bar x^{\mu_l}_{\lambda}
+
\sum_j d^{\mu_l}_{\kappa_j}(\lambda) \bar x^{\lambda}_{\kappa_j} \bar y^{\lambda}_{\kappa_j}
\right\rangle
&=
\sum_j \delta_{\kappa_k,\kappa_j} d^{\mu_l}_{\kappa_j}
-
\sum_i \delta_{\mu_i,\mu_l} d^{\mu_i}_{\kappa_k} \\
&=0
\end{align*}
where we used that
\[
\langle
y^{\lambda}_{\kappa_k} x^{\lambda}_{\kappa_k},
\;
\bar x^{\lambda}_{\kappa_j} \bar y^{\lambda}_{\kappa_j}
\rangle
=
\delta_{\kappa_k,\kappa_j}
\quad \text{ and } \quad
\langle
x^{\mu_i}_{\lambda} y^{\mu_i}_{\lambda},
\;
\bar y^{\mu_l}_{\lambda} \bar x^{\mu_l}_{\lambda}
\rangle
=
\delta_{\mu_i,\mu_l}
\]
This verifies that the dual relations at vertices are indeed orthogonal to the relations at vertices.
\end{proof}

\subsection{Reduction system of $\overline{Q}_4$}\label{sec:7.2}

Now that we have an explicit description of $\overline{\D}_4$, we can define an explicit reduction system that might satisfy the diamond condition for $\overline{I}_4$. But first let's give a visual description of $\overline{Q}_4$, with explicit labels to the arrows.

$$\begin{tikzpicture}[scale=.4]
    \CBweight{1,1,0,0}{9,9,9,9}
    \dCAP{0}{1}{1}
    \dCUP{0}{1}{1}
    \LINE{2}{1}
    \LINE{3}{1}
    \CFILL{4}{9,9,9,9}

    \begin{scope}[yshift=5cm]
        \CBweight{0,0,0,0}{9,9,9,9}
        \LINE{0}{1}
        \LINE{1}{1}
        \LINE{2}{1}
        \LINE{3}{1}
        \CFILL{4}{9,9,9,9}
    \end{scope}

    \begin{scope}[yshift=-5cm]
        \CBweight{1,0,1,0}{9,9,9,9}
        \dLINE{0}{1}
        \LINE{3}{1}
        \CCUP{1}{2}{1}{9}
        \CCAP{1}{2}{1}{9}
        \CFILL{4}{9,9,9,9}
    \end{scope}
    \begin{scope}[yshift=-10cm]
        \begin{scope}[xshift=-4cm]
            \CBweight{0,1,1,0}{9,9,9,9}
            \dLINE{2}{1}
            \LINE{3}{1}
            \CCUP{0}{1}{1}{9}
            \CCAP{0}{1}{1}{9}
            \CFILL{4}{9,9,9,9}
        \end{scope}
        \begin{scope}[xshift=4cm]
            \CBweight{1,0,0,1}{9,9,9,9}
            \dLINE{0}{1}
            \LINE{1}{1}
            \CCUP{2}{3}{1}{9}
            \CCAP{2}{3}{1}{9}
            \CFILL{4}{9,9,9,9}
        \end{scope}
    \end{scope}
    \begin{scope}[yshift=-15cm]
        \CBweight{0,1,0,1}{9,9,9,9}
        \CCUP{0}{1}{1}{9}
        \CCAP{0}{1}{1}{9}
        \CCUP{2}{3}{1}{9}
        \CCAP{2}{3}{1}{9}
        \CFILL{4}{9,9,9,9}
    \end{scope}
    \begin{scope}[yshift=-20cm]
        \CBweight{0,0,1,1}{9,9,9,9}
        \CCUP{0}{3}{.75}{9}
        \CCAP{0}{3}{.75}{9}
        \CCUP{1}{2}{1}{9}
        \CCAP{1}{2}{1}{9}
        \CFILL{4}{9,9,9,9}
    \end{scope}

    \begin{scope}[yshift=-25cm]
        \CBweight{1,1,1,1}{9,9,9,9}
        \CdCUP{0}{1}{1}{9}
        \CdCAP{0}{1}{1}{9}
        \CdCUP{2}{3}{1}{9}
        \CdCAP{2}{3}{1}{9}
        \CFILL{4}{9,9,9,9}
    \end{scope}
    
    \draw[<-, very thick, , draw=\colorFor{9}] (1,-1.75) -- (1,-3.25);
\draw[<-, very thick, , draw=\colorFor{9}] (2,-3.25) -- (2,-1.75);
\draw[<-, very thick, , draw=\colorFor{9}] (-1,-6.75) -- (-2,-8.25);
\draw[<-, very thick, , draw=\colorFor{9}] (-1,-8.25) -- (0,-6.75);
\draw[<-, very thick, , draw=\colorFor{9}] (3,-6.75) -- (4,-8.25);
\draw[<-, very thick, , draw=\colorFor{9}] (5,-8.25) -- (4,-6.75);
\draw[<-, very thick, , draw=\colorFor{9}] (-2,-11.75) -- (-1,-13.25);
\draw[<-, very thick, , draw=\colorFor{9}] (0,-13.25) -- (-1,-11.75);
\draw[<-, very thick, , draw=\colorFor{9}] (4,-11.75) -- (3,-13.25);
\draw[<-, very thick, , draw=\colorFor{9}] (4,-13.25) -- (5,-11.75);
\draw[<-, very thick, , draw=\colorFor{9}] (1,-16.75) -- (1,-18.25);
\draw[<-, very thick, , draw=\colorFor{9}] (2,-18.25) -- (2,-16.75);
\draw[ very thick, , draw=\colorFor{9}] (-7,-12.5) arc[start angle=180, end angle=90, x radius=6, y radius=7];
\draw[->,very thick, , draw=\colorFor{9}] (-7,-12.5) arc[start angle=180, end angle=270, x radius=6, y radius=7];
\draw[->,very thick, , draw=\colorFor{9}] (-8,-12.5) arc[start angle=180, end angle=90, x radius=7, y radius=8];
\draw[ very thick, , draw=\colorFor{9}] (-8,-12.5) arc[start angle=180, end angle=270, x radius=7, y radius=8];

\draw[<-, very thick, , draw=\colorFor{9}] (1,3.25) -- (1,1.75);
\draw[<-, very thick, , draw=\colorFor{9}] (2,1.75) -- (2,3.25);
\draw[<-, very thick, , draw=\colorFor{9}] (1,-21.75) -- (1,-23.25);
\draw[<-, very thick, , draw=\colorFor{9}] (2,-23.25) -- (2,-21.75);
\draw[->,very thick, , draw=\colorFor{9}] (10,-12.5) arc[start angle=0, end angle=90, x radius=6, y radius=12];
\draw[very thick, , draw=\colorFor{9}] (10,-12.5) arc[start angle=0, end angle=-90, x radius=6, y radius=12];
\draw[very thick, , draw=\colorFor{9}] (11,-12.5) arc[start angle=0, end angle=90, x radius=7, y radius=13];
\draw[->,very thick, , draw=\colorFor{9}] (11,-12.5) arc[start angle=0, end angle=-90, x radius=7, y radius=13];

\node at (0,-2.5) {$\overline{y}_2$};
\node at (3,-2.5) {$\overline{x}_2$};

\node at (0,-17.5) {$\overline{y}_7$};
\node at (3,-17.5) {$\overline{x}_7$};

\node at (-2.3,-7) {$\overline{y}_3$};
\node at (.3,-8) {$\overline{x}_3$};

\node at (2.7,-8) {$\overline{y}_4$};
\node at (5.3,-7) {$\overline{x}_4$};

\node at (-2.3,-13) {$\overline{y}_5$};
\node at (.3,-12) {$\overline{x}_5$};

\node at (2.7,-12) {$\overline{y}_6$};
\node at (5.3,-13) {$\overline{x}_6$};

\node at (-9,-12.5) {$\overline{y}_8$};
\node at (-6,-12.5) {$\overline{x}_8$};

\node at (0,2.5) {$\overline{y}_1$};
\node at (3,2.5) {$\overline{x}_1$};

\node at (0,-22.5) {$\overline{y}_9$};
\node at (3,-22.5) {$\overline{x}_9$};

\node at (9,-12.5) {$\overline{y}_{10}$};
\node at (12,-12.5) {$\overline{x}_{10}$};

\end{tikzpicture}$$

\begin{definition}\label{def:7.2}
    Let $\overline{S}_{4}$ be the set of all paths in $\overline{Q}_{4}$ of the following four types:
\begin{enumerate}[(I)]
    \item length 2 paths of the form $\overline{y}_{\lambda}^{\nu}\overline{x}_{\mu}^{\nu}$ for some weights $\lambda, \mu, \nu \in \Delta_{4}^{0}$
    \item length 2 paths of the form $\overline{y}_{\lambda}^{\nu}\overline{y}_{\nu}^{\mu}$ or $\overline{x}_{\nu}^{\mu}\overline{x}_{\lambda}^{\nu}$ in path $\textit{(f)}$ from Lemma~\ref{lem:5.11}
    \item length 2 paths of the form $\overline{y}_{\lambda}^{\nu}\overline{y}_{\nu}^{\mu}$ or $\overline{x}_{\nu}^{\mu}\overline{x}_{\lambda}^{\nu}$ in squares $\textit{(a)}$ and $\textit{(d)}$ from Lemma~\ref{lem:5.10}, where $\nu$ is the right vertex in square $\textit{(a)}$, and $\nu$ is the third highest vertex in square $\textit{(d)}$
\end{enumerate}
For each $s \in \overline{S}_4$ let $\varphi_s \in \Bbbk\mathrm{Irr}_{\overline{S}_4}$ denote the linear combination of the irreducible paths so that $s - \varphi_s \in \overline{I}_4$ (Theorem~\ref{thm:7.1} ensure that this is possible). Then we define $\overline{R}_4 = \{(s, \varphi_s) \mid s \in \overline{S}_4\}$. 
The reduction system $\overline{R}_4$ is given explicitely by
\[
\begin{gathered}
(\overline{y}_1\overline{x}_1,\,-\overline{x}_2\overline{y}_2)
\qquad
(\overline{y}_2\overline{x}_2,\,-\overline{x}_3\overline{y}_3-\overline{x}_4\overline{y}_4)
\qquad
(\overline{y}_3\overline{x}_3,\,0)
\qquad
(\overline{y}_4\overline{x}_4,\,0)
\\
(\overline{y}_5\overline{x}_5,\,-\overline{x}_7\overline{y}_7)
\qquad
(\overline{y}_6\overline{x}_6,\,-\overline{x}_7\overline{y}_7)
\qquad
(\overline{y}_7\overline{x}_7,\,-\overline{x}_9\overline{y}_9)
\qquad
(\overline{y}_8\overline{x}_8,\,2\overline{x}_9\overline{y}_9)
\\
(\overline{y}_9\overline{x}_9,\,0)
\qquad
(\overline{y}_{10}\overline{x}_{10},\,0)
\qquad
(\overline{y}_2\overline{x}_{10},\,-\overline{x}_8\overline{x}_9)
\qquad
(\overline{y}_{10}\overline{x}_2,\,-\overline{y}_9\overline{y}_8)
\\
(\overline{y}_3\overline{x}_4,\,-\overline{x}_5\overline{y}_6)
\qquad
(\overline{y}_4\overline{x}_3,\,-\overline{x}_6\overline{y}_5)
\qquad
(\overline{y}_8\overline{x}_4,\,-\overline{y}_7\overline{y}_6)
\qquad
(\overline{y}_4\overline{x}_8,\,-\overline{x}_6\overline{x}_7)
\\
(\overline{y}_8\overline{x}_3,\,-\overline{y}_7\overline{y}_5)
\qquad
(\overline{y}_3\overline{x}_8,\,-\overline{x}_5\overline{x}_7)
\qquad
(\overline{y}_{10}\overline{y}_1,\,0)
\qquad
(\overline{x}_1\overline{x}_{10},\,0)
\\
(\overline{y}_6\overline{y}_4,\,-\overline{y}_5\overline{y}_3-\overline{x}_7\overline{y}_8)
\qquad
(\overline{x}_4\overline{x}_6,\,-\overline{x}_3\overline{x}_5-\overline{x}_8\overline{y}_7)
\qquad
(\overline{y}_8\overline{y}_2,\,-\overline{x}_9\overline{y}_{10})
\qquad
(\overline{x}_2\overline{x}_8,\,-\overline{x}_{10}\overline{y}_{9})
\end{gathered}
\]
\end{definition}

\begin{lemma}\label{lem:7.3}
    The reduction system $\overline{R}_4$ is reduction-finite.
\end{lemma}

\begin{proof}
    The proof of~\cite[Lemma~3.14]{BarmeierWang2026} for $\overline{R}_m^n$ works also for $\overline{R}_4$.
\end{proof}

\begin{lemma}\label{lem:7.4}
    The reduction system $\overline{R}_4$ satisfies the diamond condition for $\overline{I}_4$.
\end{lemma}

\begin{proof}
    Thanks to Lemma~\ref{lem:7.3} and~\cite[Theorem 1.2]{Bergman1978}, it suffices to show that all overlap ambiguities of $\overline{S}_4$ are resolvable with respect to $\overline{R}_4$. But there are only ten overlaps given by
    $$
    \overline{y}_{10}\overline{y}_1\overline{x}_1,\quad
    \overline{y}_6\overline{y}_4\overline{x}_4,\quad
    \overline{y}_6\overline{y}_4\overline{x}_3,\quad
    \overline{y}_6\overline{y}_4\overline{x}_8,\quad
    \overline{y}_8\overline{y}_2\overline{x}_2,
    $$
    $$
    \overline{y}_1\overline{x}_1\overline{x}_{10},\quad
    \overline{y}_4\overline{x}_4\overline{x}_6,\quad
    \overline{y}_3\overline{x}_4\overline{x}_6,\quad
    \overline{y}_8\overline{x}_4\overline{x}_6,\quad
    \overline{y}_2\overline{x}_2\overline{x}_8,
    $$
    and it can be easily checked that they are indeed resolvable with respect to $\overline{R}_4$. For instance, by applying succesive basic reductions we obtain
    $$\textcolor{red}{\overline{y}_{10}\overline{y}_1}\overline{x}_1 \to 0 \quad \text{ and } \quad \overline{y}_{10}\textcolor{red}{\overline{y}_1\overline{x}_1} \to -\overline{y}_{10}\overline{x}_2\overline{y}_2 \to \overline{y}_9\overline{y}_8\overline{y}_2 \to \overline{y}_9\overline{x}_9\overline{y}_{10} \to 0 $$
    as wanted.
\end{proof}

\subsection{Hochschild cohomology of $\D_4$}\label{sec:7.3}

We recall that Theorem~\ref{intformal} tells us that $\D_4$ is intrinsically formal if
$$\text{HH}^2_q(\D_4, \D_4) = 0$$
for every $q > 0$. We will thus compute these Hochschild cohomology groups thanks to $\overline{R}_4$ via the method from Corollary~\ref{method}. Firstly, we explain how to compute easily all but finitely many such groups.

\begin{lemma}\label{lem:7.5}
    We have that
    $$\text{HH}^2_q(\D_4, \D_4) = 0$$
    for $q$ odd, or for $q \geq 10$.
\end{lemma}

\begin{proof}
    Let $\{ (s, \varphi_s + \tilde{\varphi}_s t) \}_{s \in \overline{S}_4}$ be a first order deformation of $\overline{R}_4$ corresponding to a cocycle in $\text{HH}^2_q(\D_4,\D_4)$. Then $\tilde{\varphi} \in \text{Hom}_{\Bbbk (\overline{Q}_4)_0^e}(\Bbbk \overline{S}_4, \Bbbk \text{Irr}_{\overline{S}_4})_q$, so in other words $\tilde{\varphi}$ is a degree $q$ map that fixes the endpoints, and hence sends elements of $\overline{S}_4$ to a linear combination of parallel irreducible paths of length $q+2$. \\
    Suppose that $q$ is odd. Then we recall from Lemma~\ref{lem:5.5} that $\overline{Q}_4$ is bipartite, and hence $\tilde{\varphi} = 0$, so that $\text{HH}^2_q(\D_4, \D_4) = 0$. \\
    Suppose now that $q \geq 10$. Observe that the longest and unique irreducible path that is parallel to an element of $\overline{S}_4$ is $\overline{x}_2\overline{x}_3\overline{x}_5\overline{x}_7\overline{x}_9\overline{y}_9\overline{y}_7\overline{y}_5\overline{y}_3\overline{y}_2$, which is parallel to $\overline{y}_1\overline{x}_1$. Since this path has length $10$, $\tilde{\varphi} = 0$, so that $\text{HH}^2_q(\D_4, \D_4) = 0$.
\end{proof}

It remains to compute $\text{HH}^2_q(\D_4, \D_4)$ for $q = 2,4,6$ and $8$. We will compute them one by one.

\begin{lemma}\label{lem:7.6}
    We have that
    $$\text{HH}^2_2(\D_4, \D_4) = 0$$
\end{lemma}

\begin{proof}
    Let $\{ (s, \varphi_s + \tilde{\varphi}_s t) \}_{s \in \overline{S}_4}$ be a first order deformation of $\overline{R}_4$ corresponding to a cocycle in $\text{HH}^2_2(\D_4,\D_4)$. Then $\tilde{\varphi} \in \text{Hom}_{\Bbbk (\overline{Q}_4)_0^e}(\Bbbk \overline{S}_4, \Bbbk \text{Irr}_{\overline{S}_4})_2$, so in other words $\tilde{\varphi}$ is a degree $2$ map that fixes the endpoints, and thus $\tilde{\varphi}_s$ has the following form for $s \in \overline{S}_4$
    $$\tilde{\varphi}_{\overline{y}_1\overline{x}_1}
= \alpha_1\overline{x}_2\overline{x}_3\overline{y}_3\overline{y}_2
+ \alpha_2\overline{x}_2\overline{x}_4\overline{y}_4\overline{y}_2,
$$
$$\tilde{\varphi}_{\overline{y}_2\overline{x}_2}
= \alpha_3\overline{x}_3\overline{x}_5\overline{y}_5\overline{y}_3
+ \alpha_4\overline{x}_3\overline{x}_5\overline{x}_7\overline{y}_8
+ \alpha_5\overline{x}_8\overline{y}_7\overline{y}_5\overline{y}_3
+ \alpha_6\overline{x}_8\overline{x}_9\overline{y}_9\overline{y}_8,$$
$$\tilde{\varphi}_{\overline{y}_3\overline{x}_3}
= \alpha_7\overline{x}_5\overline{x}_7\overline{y}_7\overline{y}_5,
\quad
\tilde{\varphi}_{\overline{y}_4\overline{x}_4}
= \alpha_8\overline{x}_6\overline{x}_7\overline{y}_7\overline{y}_6,$$
$$\tilde{\varphi}_{\overline{y}_5\overline{x}_5}
= \alpha_9\overline{x}_7\overline{x}_9\overline{y}_9\overline{y}_7,
\quad
\tilde{\varphi}_{\overline{y}_6\overline{x}_6}
= \alpha_{10}\overline{x}_7\overline{x}_9\overline{y}_9\overline{y}_7,$$
$$\tilde{\varphi}_{\overline{y}_2\overline{x}_{10}}
= \alpha_{11}\overline{x}_3\overline{x}_5\overline{x}_7\overline{x}_9,
\quad
\tilde{\varphi}_{\overline{y}_{10}\overline{x}_2}
= \alpha_{12}\overline{y}_9\overline{y}_7\overline{y}_5\overline{y}_3,$$
$$\tilde{\varphi}_{\overline{y}_3\overline{x}_4}
= \alpha_{13}\overline{x}_5\overline{x}_7\overline{y}_7\overline{y}_6,
\quad
\tilde{\varphi}_{\overline{y}_4\overline{x}_3}
= \alpha_{14}\overline{x}_6\overline{x}_7\overline{y}_7\overline{y}_5,$$
$$\tilde{\varphi}_{\overline{y}_8\overline{x}_4}
= \alpha_{15}\overline{x}_9\overline{y}_9\overline{y}_7\overline{y}_6,
\quad
\tilde{\varphi}_{\overline{y}_4\overline{x}_8}
= \alpha_{16}\overline{x}_6\overline{x}_7\overline{x}_9\overline{y}_9,$$
$$\tilde{\varphi}_{\overline{y}_8\overline{x}_3}
= \alpha_{17}\overline{x}_9\overline{y}_9\overline{y}_7\overline{y}_5,
\quad
\tilde{\varphi}_{\overline{y}_3\overline{x}_8}
= \alpha_{18}\overline{x}_5\overline{x}_7\overline{x}_9\overline{y}_9,$$
$$\tilde{\varphi}_{\overline{y}_6\overline{y}_4}
= \alpha_{19}\overline{x}_7\overline{y}_7\overline{y}_5\overline{y}_3
+ \alpha_{20}\overline{x}_7\overline{x}_9\overline{y}_9\overline{y}_8,
$$
$$\tilde{\varphi}_{\overline{x}_4\overline{x}_6}
= \alpha_{21}\overline{x}_3\overline{x}_5\overline{x}_7\overline{y}_7
+ \alpha_{22}\overline{x}_8\overline{x}_9\overline{y}_9\overline{y}_7,
$$
$$\tilde{\varphi}_{\overline{y}_8\overline{y}_2}
= \alpha_{23}\overline{y}_7\overline{y}_5\overline{y}_3\overline{y}_2,
\quad
\tilde{\varphi}_{\overline{x}_2\overline{x}_8}
= \alpha_{24}\overline{x}_2\overline{x}_3\overline{x}_5\overline{x}_7,$$
$$\tilde{\varphi}_s = 0 \quad \text{else},$$
with some conditions on the $\alpha_i$'s that we are going to compute. We recall from Definition~\ref{defmethod}, that $\tilde{\varphi}$ is a \(2\)-cocycle if and only if for any overlap ambiguity $pqr$ of length $3$, we have 
$$(p \star q) \star r  = p \star (q \star r) \text{ mod } t^2$$
where
$$u \star v = \left\{
    \begin{array}{cc}
        uv & \mbox{if } uv \not \in S \\
        \varphi_{uv} + t\tilde{\varphi}_{uv} & \mbox{if } uv \in S
    \end{array} \right.$$
For example, note that the overlap ambiguity $\overline{y}_{10}\overline{y}_1\overline{x}_1$ admits a unique parallel irreducible path of length $5$ given by $\overline{y}_9\overline{y}_7\overline{y}_5\overline{y}_3\overline{y}_2$, that we will denote $p$. Thus we can compute
\begin{align*}
    \overline{y}_{10} \star (\overline{y}_1 \star \overline{x}_1) &= \overline{y}_{10} \star (-\overline{x}_2\overline{y}_2 + t(\alpha_1\overline{x}_2\overline{x}_3\overline{y}_3\overline{y}_2 + \alpha_2\overline{x}_2\overline{x}_4\overline{y}_4\overline{y}_2)) &\text{ mod } t^2 \\
    &= -(\overline{y}_{10} \star \overline{x}_2) \star \overline{y}_2 + t(\alpha_1 - \alpha_2)p &\text{ mod } t^2 \\
    &= -(-\overline{y}_{9}\overline{y}_{8} + t\alpha_{12}\overline{y}_9\overline{y}_7\overline{y}_5\overline{y}_3 ) \star \overline{y}_2 + t(\alpha_1 - \alpha_2)p &\text{ mod } t^2 \\
    &= \overline{y}_{9} \star (\overline{y}_8 \star \overline{y}_2) + t(- \alpha_{12} + \alpha_1 - \alpha_2)p  &\text{ mod } t^2 \\
    &= \overline{y}_{9} \star (-\overline{x}_{9}\overline{y}_{10} + t\alpha_{23}\overline{y}_7\overline{y}_5\overline{y}_3\overline{y}_2 ) + t(- \alpha_{12} + \alpha_1 - \alpha_2)p  &\text{ mod } t^2 \\
    &= t(\alpha_{23} - \alpha_{12} + \alpha_1 - \alpha_2)p  &\text{ mod } t^2 \\
    &= (\overline{y}_{10} \star \overline{y}_1) \star \overline{x}_1 = 0 &\text{ mod } t^2
\end{align*}
So we deduce that
$$\alpha_1 - \alpha_2 - \alpha_{12} + \alpha_{23} = 0$$
The other conditions can be computed similarly and we find
\begin{align*}
& \alpha_1 - \alpha_2 - \alpha_{12} + \alpha_{23} = 0 \\
& \alpha_1 - \alpha_2 - \alpha_{11} + \alpha_{24} = 0 \\
& \alpha_8 - \alpha_9 + \alpha_{13} + \alpha_{15} + \alpha_{19} + \alpha_{20} = 0 \\
& \alpha_8 - \alpha_9 + \alpha_{14} + \alpha_{16} + \alpha_{21} + \alpha_{22} = 0 \\
& \alpha_7 - \alpha_{10} + \alpha_{14} + \alpha_{17} + \alpha_{20} = 0 \\
& \alpha_7 - \alpha_{10} + \alpha_{13} + \alpha_{18} + \alpha_{22} = 0 \\
& \alpha_3 - 2\alpha_5 - \alpha_{12} - \alpha_{15} + \alpha_{17} + \alpha_{19} - \alpha_{23} = 0 \\
& \alpha_3 - 2\alpha_4 - \alpha_{11} - \alpha_{16} + \alpha_{18} + \alpha_{21} - \alpha_{24} = 0
\end{align*}
You can observe that these eight equations are linearly independant because a new variable is introduced in each equation. Thus the dimension of the cocycle space is $24-8 = 16$. \\
Now we need to compute the dimension of the subspace of the coboundaries. We recall that a coboundary is an element $\psi \in \text{Hom}_{\Bbbk (\overline{Q}_4)_0^e}(\Bbbk (\overline{Q}_4)_1, \Bbbk \text{Irr}_{\overline{S}_4})_2$, so in other words $\psi$ is a degree $2$ map that fixes the endpoints, and thus $\psi$ has the following form
    $$\psi_{\overline{y}_1}
= \beta_1\overline{x}_2\overline{y}_2\overline{y}_1,
\quad
\psi_{\overline{x}_1}
= \beta_2\overline{x}_1\overline{x}_2\overline{y}_2,$$
$$\psi_{\overline{y}_2}
= \beta_3\overline{x}_3\overline{y}_3\overline{y}_2
+ \beta_4\overline{x}_4\overline{y}_4\overline{y}_2
+ \beta_5\overline{x}_8\overline{x}_9\overline{y}_{10},
\quad
\psi_{\overline{x}_2}
= \beta_6\overline{x}_2\overline{x}_3\overline{y}_3
+ \beta_7\overline{x}_2\overline{x}_4\overline{y}_4
+ \beta_8\overline{x}_{10}\overline{y}_9\overline{y}_8,$$
$$\psi_{\overline{y}_3}
= \beta_9\overline{x}_5\overline{y}_5\overline{y}_3
+ \beta_{10}\overline{x}_5\overline{x}_7\overline{y}_8,
\quad
\psi_{\overline{x}_3}
= \beta_{11}\overline{x}_3\overline{x}_5\overline{y}_5
+ \beta_{12}\overline{x}_8\overline{y}_7\overline{y}_5,$$
$$\psi_{\overline{y}_4}
= \beta_{13}\overline{x}_6\overline{y}_5\overline{y}_3
+ \beta_{14}\overline{x}_6\overline{x}_7\overline{y}_8,
\qquad
\psi_{\overline{x}_4}
= \beta_{15}\overline{x}_3\overline{x}_5\overline{y}_6
+ \beta_{16}\overline{x}_8\overline{y}_7\overline{y}_6,$$
$$\psi_{\overline{y}_5}
= \beta_{17}\overline{x}_7\overline{y}_7\overline{y}_5,
\qquad
\psi_{\overline{x}_5}
= \beta_{18}\overline{x}_5\overline{x}_7\overline{y}_7,$$
$$\psi_{\overline{y}_6} = \beta_{19}\overline{x}_7\overline{y}_7\overline{y}_6,
\qquad\quad
\psi_{\overline{x}_6} = \beta_{20}\overline{x}_6\overline{x}_7\overline{y}_7,$$
$$\psi_{\overline{y}_7}
= \beta_{21}\overline{x}_9\overline{y}_9\overline{y}_7,
\qquad
\psi_{\overline{x}_7}
= \beta_{22}\overline{x}_7\overline{x}_9\overline{y}_9,$$
$$\psi_{\overline{y}_8}
= \beta_{23}\overline{x}_9\overline{y}_9\overline{y}_8 + \beta_{24}\overline{y}_7\overline{y}_5\overline{y}_3,
\qquad
\psi_{\overline{x}_8}
= \beta_{25}\overline{x}_8\overline{x}_9\overline{y}_9 + \beta_{26}\overline{x}_3\overline{x}_5\overline{x}_7,$$
$$\psi_{\overline{y}_9}
= 0,
\qquad
\psi_{\overline{x}_9}
= 0,\qquad\psi_{\overline{y}_{10}}
= 0,
\qquad
\psi_{\overline{x}_{10}}
= 0,$$
    where $\beta_1, \ldots, \beta_{26} \in \Bbbk$. If $\delta(\psi)_s = \tilde{\varphi}_s$ for every $s \in \overline{S}_4$, then using Definition~\ref{defmethod} to compute $\delta(\psi)_s$, we obtain
    \begin{align*}
& \alpha_1 = \beta_1 + \beta_2 + \beta_3 + \beta_6 \\
& \alpha_2 = \beta_1 + \beta_2 + \beta_4 + \beta_7 \\
& \alpha_3 = -\beta_3 - \beta_4 - \beta_6 - \beta_7 + \beta_9 + \beta_{11} - \beta_{13} - \beta_{15} \\
& \alpha_4 = -\beta_3 - \beta_7 + \beta_{10} - \beta_{14} - \beta_{15} \\
& \alpha_5 = -\beta_4 - \beta_6 + \beta_{12} - \beta_{13} - \beta_{16} \\
& \alpha_6 = -\beta_5 - \beta_8 + \beta_{14} + \beta_{16} \\
& \alpha_7 = -\beta_{10} - \beta_{12} \\
& \alpha_8 = \beta_{13} - \beta_{14} + \beta_{15} - \beta_{16} \\
& \alpha_9 = \beta_{17} + \beta_{18} + \beta_{21} + \beta_{22} \\
& \alpha_{10} = \beta_{19} + \beta_{20} + \beta_{21} + \beta_{22} \\
& \alpha_{11} = \beta_3 - \beta_4 + \beta_{26} \\
& \alpha_{12} = \beta_6 - \beta_7 + \beta_{24} \\
& \alpha_{13} = \beta_9 - \beta_{10} - \beta_{16} + \beta_{18} + \beta_{19} \\
& \alpha_{14} = \beta_{11} - \beta_{12} - \beta_{14} + \beta_{17} + \beta_{20} \\
& \alpha_{15} = -\beta_{15} + 2\beta_{16} - \beta_{19} + \beta_{21} - \beta_{23} - \beta_{24} \\
& \alpha_{16} = -\beta_{13} + 2\beta_{14} - \beta_{20} + \beta_{22} - \beta_{25} - \beta_{26} \\
& \alpha_{17} = -\beta_{11} + 2\beta_{12} - \beta_{17} + \beta_{21} - \beta_{23} \\
& \alpha_{18} = -\beta_9 + 2\beta_{10} - \beta_{18} + \beta_{22} - \beta_{25} \\
& \alpha_{19} = -\beta_9 - \beta_{13} + \beta_{17} - \beta_{19} + \beta_{24} \\
& \alpha_{20} = \beta_{10} + \beta_{14} + \beta_{19} + \beta_{22} + \beta_{23} \\
& \alpha_{21} = -\beta_{11} - \beta_{15} + \beta_{18} - \beta_{20} + \beta_{26} \\
& \alpha_{22} = \beta_{12} + \beta_{16} + \beta_{20} + \beta_{21} + \beta_{25} \\
& \alpha_{23} = -\beta_3 + \beta_4 + \beta_{24} \\
& \alpha_{24} = -\beta_6 + \beta_7 + \beta_{26}
\end{align*}
    Indeed, let's show an equation. We compute
    \begin{align*}
\delta(\psi)_{\overline{y}_3\overline{x}_3}
&= T(\psi)(\overline{y}_3\overline{x}_3)
= \sigma\pi(\psi_{\overline{y}_3}\overline{x}_3
+ \overline{y}_3\psi_{\overline{x}_3}) \\
&= \sigma\pi(
\beta_9\overline{x}_5\overline{y}_5\overline{y}_3\overline{x}_3
+ \beta_{10}\overline{x}_5\overline{x}_7\overline{y}_8\overline{x}_3
+ \beta_{11}\overline{y}_3\overline{x}_3\overline{x}_5\overline{y}_5
+ \beta_{12}\overline{y}_3\overline{x}_8\overline{y}_7\overline{y}_5) \\
&= -\beta_{10}\overline{x}_5\overline{x}_7\overline{y}_7\overline{y}_5
-\beta_{12}\overline{x}_5\overline{x}_7\overline{y}_7\overline{y}_5 \\
&= \tilde{\varphi}_{\overline{y}_3\overline{x}_3}
= \alpha_7\overline{x}_5\overline{x}_7\overline{y}_7\overline{y}_5
\end{align*}
    So we deduce the equation
    $$\alpha_7 = - \beta_{10} - \beta_{12}$$
    Similarly we may verify the other equations. \\ We can observe that this system is compatible with the other equations on the $\alpha_i$'s, which is coherent because the space of coboundaries lies inside the space of cocycles, by definition. This system can be viewed as the following matrix
    \vspace{1cm}
\[
\scalebox{0.67}{$
\begin{pmatrix}
1 & 1 & 1 & 0 & 0 & 1 & 0 & 0 & 0 & 0 & 0 & 0 & 0 & 0 & 0 & 0 & 0 & 0 & 0 & 0 & 0 & 0 & 0 & 0 & 0 & 0 \\
1 & 1 & 0 & 1 & 0 & 0 & 1 & 0 & 0 & 0 & 0 & 0 & 0 & 0 & 0 & 0 & 0 & 0 & 0 & 0 & 0 & 0 & 0 & 0 & 0 & 0 \\
0 & 0 & -1 & -1 & 0 & -1 & -1 & 0 & 1 & 0 & 1 & 0 & -1 & 0 & -1 & 0 & 0 & 0 & 0 & 0 & 0 & 0 & 0 & 0 & 0 & 0 \\
0 & 0 & -1 & 0 & 0 & 0 & -1 & 0 & 0 & 1 & 0 & 0 & 0 & -1 & -1 & 0 & 0 & 0 & 0 & 0 & 0 & 0 & 0 & 0 & 0 & 0 \\
0 & 0 & 0 & -1 & 0 & -1 & 0 & 0 & 0 & 0 & 0 & 1 & -1 & 0 & 0 & -1 & 0 & 0 & 0 & 0 & 0 & 0 & 0 & 0 & 0 & 0 \\
0 & 0 & 0 & 0 & -1 & 0 & 0 & -1 & 0 & 0 & 0 & 0 & 0 & 1 & 0 & 1 & 0 & 0 & 0 & 0 & 0 & 0 & 0 & 0 & 0 & 0 \\
0 & 0 & 0 & 0 & 0 & 0 & 0 & 0 & 0 & -1 & 0 & -1 & 0 & 0 & 0 & 0 & 0 & 0 & 0 & 0 & 0 & 0 & 0 & 0 & 0 & 0 \\
0 & 0 & 0 & 0 & 0 & 0 & 0 & 0 & 0 & 0 & 0 & 0 & 1 & -1 & 1 & -1 & 0 & 0 & 0 & 0 & 0 & 0 & 0 & 0 & 0 & 0 \\
0 & 0 & 0 & 0 & 0 & 0 & 0 & 0 & 0 & 0 & 0 & 0 & 0 & 0 & 0 & 0 & 1 & 1 & 0 & 0 & 1 & 1 & 0 & 0 & 0 & 0 \\
0 & 0 & 0 & 0 & 0 & 0 & 0 & 0 & 0 & 0 & 0 & 0 & 0 & 0 & 0 & 0 & 0 & 0 & 1 & 1 & 1 & 1 & 0 & 0 & 0 & 0 \\
0 & 0 & 1 & -1 & 0 & 0 & 0 & 0 & 0 & 0 & 0 & 0 & 0 & 0 & 0 & 0 & 0 & 0 & 0 & 0 & 0 & 0 & 0 & 0 & 0 & 1 \\
0 & 0 & 0 & 0 & 0 & 1 & -1 & 0 & 0 & 0 & 0 & 0 & 0 & 0 & 0 & 0 & 0 & 0 & 0 & 0 & 0 & 0 & 0 & 1 & 0 & 0 \\
0 & 0 & 0 & 0 & 0 & 0 & 0 & 0 & 1 & -1 & 0 & 0 & 0 & 0 & 0 & -1 & 0 & 1 & 1 & 0 & 0 & 0 & 0 & 0 & 0 & 0 \\
0 & 0 & 0 & 0 & 0 & 0 & 0 & 0 & 0 & 0 & 1 & -1 & 0 & -1 & 0 & 0 & 1 & 0 & 0 & 1 & 0 & 0 & 0 & 0 & 0 & 0 \\
0 & 0 & 0 & 0 & 0 & 0 & 0 & 0 & 0 & 0 & 0 & 0 & 0 & 0 & -1 & 2 & 0 & 0 & -1 & 0 & 1 & 0 & -1 & -1 & 0 & 0 \\
0 & 0 & 0 & 0 & 0 & 0 & 0 & 0 & 0 & 0 & 0 & 0 & -1 & 2 & 0 & 0 & 0 & 0 & 0 & -1 & 0 & 1 & 0 & 0 & -1 & -1 \\
0 & 0 & 0 & 0 & 0 & 0 & 0 & 0 & 0 & 0 & -1 & 2 & 0 & 0 & 0 & 0 & -1 & 0 & 0 & 0 & 1 & 0 & -1 & 0 & 0 & 0 \\
0 & 0 & 0 & 0 & 0 & 0 & 0 & 0 & -1 & 2 & 0 & 0 & 0 & 0 & 0 & 0 & 0 & -1 & 0 & 0 & 0 & 1 & 0 & 0 & -1 & 0 \\
0 & 0 & 0 & 0 & 0 & 0 & 0 & 0 & -1 & 0 & 0 & 0 & -1 & 0 & 0 & 0 & 1 & 0 & -1 & 0 & 0 & 0 & 0 & 1 & 0 & 0 \\
0 & 0 & 0 & 0 & 0 & 0 & 0 & 0 & 0 & 1 & 0 & 0 & 0 & 1 & 0 & 0 & 0 & 0 & 1 & 0 & 0 & 1 & 1 & 0 & 0 & 0 \\
0 & 0 & 0 & 0 & 0 & 0 & 0 & 0 & 0 & 0 & -1 & 0 & 0 & 0 & -1 & 0 & 0 & 1 & 0 & -1 & 0 & 0 & 0 & 0 & 0 & 1 \\
0 & 0 & 0 & 0 & 0 & 0 & 0 & 0 & 0 & 0 & 0 & 1 & 0 & 0 & 0 & 1 & 0 & 0 & 0 & 1 & 1 & 0 & 0 & 0 & 1 & 0 \\
0 & 0 & -1 & 1 & 0 & 0 & 0 & 0 & 0 & 0 & 0 & 0 & 0 & 0 & 0 & 0 & 0 & 0 & 0 & 0 & 0 & 0 & 0 & 1 & 0 & 0 \\
0 & 0 & 0 & 0 & 0 & -1 & 1 & 0 & 0 & 0 & 0 & 0 & 0 & 0 & 0 & 0 & 0 & 0 & 0 & 0 & 0 & 0 & 0 & 0 & 0 & 1 \\
\end{pmatrix}
$}
\]
    whose reduced row echelon form is the matrix
\[
\scalebox{0.65}{$
\begin{pmatrix}
1 & 1 & 0 & 0 & 0 & 0 & 0 & 0 & 0 & 0 & 0 & 0 & 0 & 0 & 0 & 0 & 0 & 0 & 0 & 0 & 0 & 0 & 1 & 0 & 1 & 0 \\
0 & 0 & 1 & 0 & 0 & 0 & 1 & 0 & 0 & 0 & 0 & 0 & 0 & 0 & 0 & 0 & 0 & 0 & 0 & 0 & 0 & 0 & -1 & 0 & -1 & 1 \\
0 & 0 & 0 & 1 & 0 & 0 & 1 & 0 & 0 & 0 & 0 & 0 & 0 & 0 & 0 & 0 & 0 & 0 & 0 & 0 & 0 & 0 & -1 & 0 & -1 & 0 \\
0 & 0 & 0 & 0 & 1 & 0 & 0 & 1 & 0 & 0 & 0 & 0 & 0 & 0 & 0 & 0 & 0 & 0 & 0 & 0 & 0 & 0 & 1 & 0 & 1 & 0 \\
0 & 0 & 0 & 0 & 0 & 1 & -1 & 0 & 0 & 0 & 0 & 0 & 0 & 0 & 0 & 0 & 0 & 0 & 0 & 0 & 0 & 0 & 0 & 0 & 0 & -1 \\
0 & 0 & 0 & 0 & 0 & 0 & 0 & 0 & 1 & 0 & 0 & 0 & 0 & 0 & 0 & -2 & 0 & 1 & 0 & -2 & 0 & 1 & -2 & 0 & -3 & 0 \\
0 & 0 & 0 & 0 & 0 & 0 & 0 & 0 & 0 & 1 & 0 & 0 & 0 & 0 & 0 & -1 & 0 & 0 & 0 & -1 & 0 & 1 & -1 & 0 & -2 & 0 \\
0 & 0 & 0 & 0 & 0 & 0 & 0 & 0 & 0 & 0 & 1 & 0 & 0 & 0 & 0 & 2 & 0 & -1 & 0 & 2 & 0 & -1 & 1 & 0 & 2 & 0 \\
0 & 0 & 0 & 0 & 0 & 0 & 0 & 0 & 0 & 0 & 0 & 1 & 0 & 0 & 0 & 1 & 0 & 0 & 0 & 1 & 0 & -1 & 1 & 0 & 2 & 0 \\
0 & 0 & 0 & 0 & 0 & 0 & 0 & 0 & 0 & 0 & 0 & 0 & 1 & 0 & 0 & 2 & 0 & 0 & 0 & 1 & 0 & -1 & 2 & 0 & 3 & 1 \\
0 & 0 & 0 & 0 & 0 & 0 & 0 & 0 & 0 & 0 & 0 & 0 & 0 & 1 & 0 & 1 & 0 & 0 & 0 & 0 & 0 & 0 & 1 & 0 & 1 & 0 \\
0 & 0 & 0 & 0 & 0 & 0 & 0 & 0 & 0 & 0 & 0 & 0 & 0 & 0 & 1 & -2 & 0 & 0 & 0 & -1 & 0 & 1 & -1 & 0 & -2 & -1 \\
0 & 0 & 0 & 0 & 0 & 0 & 0 & 0 & 0 & 0 & 0 & 0 & 0 & 0 & 0 & 0 & 1 & 1 & 0 & 0 & 0 & 0 & 1 & 0 & 1 & 0 \\
0 & 0 & 0 & 0 & 0 & 0 & 0 & 0 & 0 & 0 & 0 & 0 & 0 & 0 & 0 & 0 & 0 & 0 & 1 & 1 & 0 & 0 & 1 & 0 & 1 & 0 \\
0 & 0 & 0 & 0 & 0 & 0 & 0 & 0 & 0 & 0 & 0 & 0 & 0 & 0 & 0 & 0 & 0 & 0 & 0 & 0 & 1 & 1 & -1 & 0 & -1 & 0 \\
0 & 0 & 0 & 0 & 0 & 0 & 0 & 0 & 0 & 0 & 0 & 0 & 0 & 0 & 0 & 0 & 0 & 0 & 0 & 0 & 0 & 0 & 0 & 1 & 0 & 1 \\
0 & 0 & 0 & 0 & 0 & 0 & 0 & 0 & 0 & 0 & 0 & 0 & 0 & 0 & 0 & 0 & 0 & 0 & 0 & 0 & 0 & 0 & 0 & 0 & 0 & 0 \\
0 & 0 & 0 & 0 & 0 & 0 & 0 & 0 & 0 & 0 & 0 & 0 & 0 & 0 & 0 & 0 & 0 & 0 & 0 & 0 & 0 & 0 & 0 & 0 & 0 & 0 \\
0 & 0 & 0 & 0 & 0 & 0 & 0 & 0 & 0 & 0 & 0 & 0 & 0 & 0 & 0 & 0 & 0 & 0 & 0 & 0 & 0 & 0 & 0 & 0 & 0 & 0 \\
0 & 0 & 0 & 0 & 0 & 0 & 0 & 0 & 0 & 0 & 0 & 0 & 0 & 0 & 0 & 0 & 0 & 0 & 0 & 0 & 0 & 0 & 0 & 0 & 0 & 0 \\
0 & 0 & 0 & 0 & 0 & 0 & 0 & 0 & 0 & 0 & 0 & 0 & 0 & 0 & 0 & 0 & 0 & 0 & 0 & 0 & 0 & 0 & 0 & 0 & 0 & 0 \\
0 & 0 & 0 & 0 & 0 & 0 & 0 & 0 & 0 & 0 & 0 & 0 & 0 & 0 & 0 & 0 & 0 & 0 & 0 & 0 & 0 & 0 & 0 & 0 & 0 & 0 \\
0 & 0 & 0 & 0 & 0 & 0 & 0 & 0 & 0 & 0 & 0 & 0 & 0 & 0 & 0 & 0 & 0 & 0 & 0 & 0 & 0 & 0 & 0 & 0 & 0 & 0 \\
0 & 0 & 0 & 0 & 0 & 0 & 0 & 0 & 0 & 0 & 0 & 0 & 0 & 0 & 0 & 0 & 0 & 0 & 0 & 0 & 0 & 0 & 0 & 0 & 0 & 0 \\
\end{pmatrix}
$}
\]
    This $(24 \times 26)$ matrix has rank $16$. Since we know that the space of cocycles has also dimension $16$, this implies that
    $$\text{HH}^2_2(\D_4, \D_4) = 0$$
    as wanted.
\end{proof}

\begin{lemma}\label{lem:7.7}
    We have that
    $$\text{HH}^2_4(\D_4, \D_4) = 0$$
\end{lemma}

\begin{proof}
    Let $\{ (s, \varphi_s + \tilde{\varphi}_s t) \}_{s \in \overline{S}_4}$ be a first order deformation of $\overline{R}_4$ corresponding to a cocycle in $\text{HH}^2_4(\D_4,\D_4)$. Then $\tilde{\varphi} \in \text{Hom}_{\Bbbk (\overline{Q}_4)_0^e}(\Bbbk \overline{S}_4, \Bbbk \text{Irr}_{\overline{S}_4})_4$, so in other words $\tilde{\varphi}$ is a degree $4$ map that fixes the endpoints, and thus $\tilde{\varphi}_s$ has the following form for $s \in \overline{S}_4$
    $$\tilde{\varphi}_{\overline{y}_1\overline{x}_1}
= \alpha_1\overline{x}_2\overline{x}_3\overline{x}_5\overline{y}_5\overline{y}_3\overline{y}_2
+ \alpha_2\overline{x}_{10}\overline{y}_9\overline{y}_7\overline{y}_5\overline{y}_3\overline{y}_2
+ \alpha_3\overline{x}_2\overline{x}_3\overline{x}_5\overline{x}_7\overline{x}_9\overline{y}_{10},
$$
$$\tilde{\varphi}_{\overline{y}_2\overline{x}_2}
= \alpha_4\overline{x}_3\overline{x}_5\overline{x}_7\overline{y}_7\overline{y}_5\overline{y}_3
+ \alpha_5\overline{x}_8\overline{x}_9\overline{y}_9\overline{y}_7\overline{y}_5\overline{y}_3
+ \alpha_6\overline{x}_3\overline{x}_5\overline{x}_7\overline{x}_9\overline{y}_9\overline{y}_8,$$
$$\tilde{\varphi}_{\overline{y}_3\overline{x}_3}
= \alpha_7\overline{x}_5\overline{x}_7\overline{x}_9\overline{y}_9\overline{y}_7\overline{y}_5,
\quad
\tilde{\varphi}_{\overline{y}_4\overline{x}_4}
= \alpha_8\overline{x}_6\overline{x}_7\overline{x}_9\overline{y}_9\overline{y}_7\overline{y}_6,$$
$$\tilde{\varphi}_{\overline{y}_3\overline{x}_4}
= \alpha_{9}\overline{x}_5\overline{x}_7\overline{x}_9\overline{y}_9\overline{y}_7\overline{y}_6,
\quad
\tilde{\varphi}_{\overline{y}_4\overline{x}_3}
= \alpha_{10}\overline{x}_6\overline{x}_7\overline{x}_9\overline{y}_9\overline{y}_7\overline{y}_5,$$
$$\tilde{\varphi}_{\overline{y}_{10}\overline{y}_1}
= \alpha_{11}\overline{y}_9\overline{y}_7\overline{y}_5\overline{y}_3\overline{y}_2\overline{y}_1,
\quad
\tilde{\varphi}_{\overline{x}_1\overline{x}_{10}}
= \alpha_{12}\overline{x}_1\overline{x}_2\overline{x}_3\overline{x}_5\overline{x}_7\overline{x}_9,$$
$$\tilde{\varphi}_{\overline{y}_6\overline{y}_4}
= \alpha_{13}\overline{x}_7\overline{x}_9\overline{y}_9\overline{y}_7\overline{y}_5\overline{y}_3,
\quad \tilde{\varphi}_{\overline{x}_4\overline{x}_6}
= \alpha_{14}\overline{x}_3\overline{x}_5\overline{x}_7\overline{x}_9\overline{y}_9\overline{y}_7,
$$
$$\tilde{\varphi}_{\overline{y}_8\overline{y}_2}
= \alpha_{15}\overline{x}_9\overline{y}_9\overline{y}_7\overline{y}_5\overline{y}_3\overline{y}_2,
\quad
\tilde{\varphi}_{\overline{x}_2\overline{x}_8}
= \alpha_{16}\overline{x}_2\overline{x}_3\overline{x}_5\overline{x}_7\overline{x}_9\overline{y}_9,$$
$$\tilde{\varphi}_s = 0 \quad \text{else},$$
for $\alpha_1, \ldots, \alpha_{16} \in \Bbbk$ without any condition on them because none of the overlaps admit a parallel irreducible path of length $7$. \\
Now we need to compute the dimension of the subspace of the coboundaries. We recall that a coboundary is an element $\psi \in \text{Hom}_{\Bbbk (\overline{Q}_4)_0^e}(\Bbbk (\overline{Q}_4)_1, \Bbbk \text{Irr}_{\overline{S}_4})_4$, so in other words $\psi$ is a degree $4$ map that fixes the endpoints, and thus $\psi$ has the following form
    $$\psi_{\overline{y}_1}
= \beta_1\overline{x}_2\overline{x}_3\overline{y}_3\overline{y}_2\overline{y}_1 + \beta_2\overline{x}_2\overline{x}_4\overline{y}_4\overline{y}_2\overline{y}_1 ,
\quad
\psi_{\overline{x}_1}
= \beta_3\overline{x}_1\overline{x}_2\overline{x}_3\overline{y}_3\overline{y}_2 + \beta_4\overline{x}_1\overline{x}_2\overline{x}_4\overline{y}_4\overline{y}_2,$$
$$\psi_{\overline{y}_2}
= \beta_5\overline{x}_3\overline{x}_5\overline{y}_5\overline{y}_3\overline{y}_2
+ \beta_6\overline{x}_8\overline{y}_7\overline{y}_5\overline{y}_3\overline{y}_2
+ \beta_7\overline{x}_3\overline{x}_5\overline{x}_{7}\overline{x}_9\overline{y}_{10},$$
$$\psi_{\overline{x}_2}
= \beta_8\overline{x}_2\overline{x}_3\overline{x}_5\overline{y}_5\overline{y}_3
+ \beta_9\overline{x}_2\overline{x}_3\overline{x}_5\overline{x}_7\overline{y}_8
+ \beta_{10}\overline{x}_{10}\overline{y}_9\overline{y}_7\overline{y}_5\overline{y}_3,$$
$$\psi_{\overline{y}_3}
= \beta_{11}\overline{x}_5\overline{x}_7\overline{y}_7\overline{y}_5\overline{y}_3
+ \beta_{12}\overline{x}_5\overline{x}_7\overline{x}_9\overline{y}_9\overline{y}_8,
\quad
\psi_{\overline{x}_3}
= \beta_{13}\overline{x}_3\overline{x}_5\overline{x}_7\overline{y}_7\overline{y}_5
+ \beta_{14}\overline{x}_8\overline{x}_9\overline{y}_9\overline{y}_7\overline{y}_5,$$
$$\psi_{\overline{y}_4}
= \beta_{15}\overline{x}_6\overline{x}_7\overline{y}_7\overline{y}_5\overline{y}_3
+ \beta_{16}\overline{x}_6\overline{x}_7\overline{x}_9\overline{y}_9\overline{y}_8,
\qquad
\psi_{\overline{x}_4}
= \beta_{17}\overline{x}_3\overline{x}_5\overline{x}_7\overline{y}_7\overline{y}_6
+ \beta_{18}\overline{x}_8\overline{x}_9\overline{y}_9\overline{y}_7\overline{y}_6,$$
$$\psi_{\overline{y}_5}
= \beta_{19}\overline{x}_7\overline{x}_9\overline{y}_9\overline{y}_7\overline{y}_5,
\qquad
\psi_{\overline{x}_5}
= \beta_{20}\overline{x}_5\overline{x}_7\overline{x}_9\overline{y}_9\overline{y}_7,$$
$$\psi_{\overline{y}_6} = \beta_{21}\overline{x}_7\overline{x}_9\overline{y}_9\overline{y}_7\overline{y}_6,
\qquad\quad
\psi_{\overline{x}_6} = \beta_{22}\overline{x}_6\overline{x}_7\overline{x}_9\overline{y}_9\overline{y}_7,$$
$$\psi_{\overline{y}_8}
= \beta_{23}\overline{x}_9\overline{y}_9\overline{y}_7\overline{y}_5\overline{y}_3,
\qquad
\psi_{\overline{x}_8}
= \beta_{24}\overline{x}_3\overline{x}_5\overline{x}_7\overline{x}_9\overline{y}_9,$$
$$\psi_{\overline{y}_{10}}
= \beta_{25}\overline{y}_9\overline{y}_7\overline{y}_5\overline{y}_3\overline{y}_2,
\qquad
\psi_{\overline{x}_{10}}
= \beta_{26}\overline{x}_2\overline{x}_3\overline{x}_5\overline{x}_7\overline{x}_9,$$
$$\psi_{\overline{y}_7}
= 0,
\qquad
\psi_{\overline{x}_7}
= 0,\qquad\psi_{\overline{y}_{9}}
= 0,
\qquad
\psi_{\overline{x}_{9}}
= 0,$$
    where $\beta_1, \ldots, \beta_{26} \in \Bbbk$. If $\delta(\psi)_s = \tilde{\varphi}_s$ for every $s \in \overline{S}_4$, then using Definition~\ref{defmethod} to compute $\delta(\psi)_s$, we obtain
    \begin{align*}
& \alpha_1 = \beta_1 + \beta_2 + \beta_3 + \beta_4 + \beta_5 + \beta_8 \\
& \alpha_2 = -\beta_3 - \beta_4 - \beta_6 + \beta_{10} \\
& \alpha_3 = -\beta_1 - \beta_2 + \beta_7 - \beta_9 \\
& \alpha_4 = \beta_5 + \beta_8 + \beta_{11} + \beta_{13} - \beta_{15} - \beta_{17} \\
& \alpha_5 = -\beta_6 - \beta_8 - \beta_{10} + \beta_{14} + \beta_{15} - \beta_{18} \\
& \alpha_6 = -\beta_5 - \beta_7 - \beta_9 + \beta_{12} - \beta_{16} + \beta_{17} \\
& \alpha_7 = -\beta_{12} - \beta_{14} \\
& \alpha_8 = -\beta_{15} - \beta_{16} - \beta_{17} - \beta_{18} \\
& \alpha_9 = -\beta_{11} - \beta_{12} - \beta_{18} + \beta_{20} + \beta_{21} \\
& \alpha_{10} = -\beta_{13} - \beta_{14} - \beta_{16} + \beta_{19} + \beta_{22} \\
& \alpha_{11} = \beta_1 - \beta_2 + \beta_{25} \\
& \alpha_{12} = \beta_3 - \beta_4 + \beta_{26} \\
& \alpha_{13} = \beta_{11} + \beta_{15} + \beta_{19} - \beta_{21} + \beta_{23} \\
& \alpha_{14} = \beta_{13} + \beta_{17} + \beta_{20} - \beta_{22} + \beta_{24} \\
& \alpha_{15} = -\beta_5 + 2\beta_6 + \beta_{23} + \beta_{25} \\
& \alpha_{16} = -\beta_8 + 2\beta_9 + \beta_{24} + \beta_{26}
\end{align*}
    Indeed, let's show an equation. We compute
    \begin{align*}
\delta(\psi)_{\overline{y}_3\overline{x}_3}
&= T(\psi)(\overline{y}_3\overline{x}_3)
= \sigma\pi(\psi_{\overline{y}_3}\overline{x}_3
+ \overline{y}_3\psi_{\overline{x}_3}) \\
&= \sigma\pi(
\beta_{11}\overline{x}_5\overline{x}_7\overline{y}_7\overline{y}_5\overline{y}_3\overline{x}_3
+ \beta_{12}\overline{x}_5\overline{x}_7\overline{x}_9\overline{y}_9\overline{y}_8\overline{x}_3 \\
&\qquad + \beta_{13}\overline{y}_3\overline{x}_3\overline{x}_5\overline{x}_7\overline{y}_7\overline{y}_5
+ \beta_{14}\overline{y}_3\overline{x}_8\overline{x}_9\overline{y}_9\overline{y}_7\overline{y}_5) \\
&= -\beta_{12}\overline{x}_5\overline{x}_7\overline{x}_9\overline{y}_9\overline{y}_7\overline{y}_5
-\beta_{14}\overline{x}_5\overline{x}_7\overline{x}_9\overline{y}_9\overline{y}_7\overline{y}_5 \\
&= \tilde{\varphi}_{\overline{y}_3\overline{x}_3}
= \alpha_7\overline{x}_5\overline{x}_7\overline{x}_9\overline{y}_9\overline{y}_7\overline{y}_5
\end{align*}
    So we deduce the equation
    $$\alpha_7 = - \beta_{12} - \beta_{14}$$
    Similarly we may verify the other equations. \\
    This system can be viewed as the following matrix
    \[
\scalebox{0.7}{$
\begin{pmatrix}
1 & 1 & 1 & 1 & 1 & 0 & 0 & 1 & 0 & 0 & 0 & 0 & 0 & 0 & 0 & 0 & 0 & 0 & 0 & 0 & 0 & 0 & 0 & 0 & 0 & 0 \\
0 & 0 & -1 & -1 & 0 & -1 & 0 & 0 & 0 & 1 & 0 & 0 & 0 & 0 & 0 & 0 & 0 & 0 & 0 & 0 & 0 & 0 & 0 & 0 & 0 & 0 \\
-1 & -1 & 0 & 0 & 0 & 0 & 1 & 0 & -1 & 0 & 0 & 0 & 0 & 0 & 0 & 0 & 0 & 0 & 0 & 0 & 0 & 0 & 0 & 0 & 0 & 0 \\
0 & 0 & 0 & 0 & 1 & 0 & 0 & 1 & 0 & 0 & 1 & 0 & 1 & 0 & -1 & 0 & -1 & 0 & 0 & 0 & 0 & 0 & 0 & 0 & 0 & 0 \\
0 & 0 & 0 & 0 & 0 & -1 & 0 & -1 & 0 & -1 & 0 & 0 & 0 & 1 & 1 & 0 & 0 & -1 & 0 & 0 & 0 & 0 & 0 & 0 & 0 & 0 \\
0 & 0 & 0 & 0 & -1 & 0 & -1 & 0 & -1 & 0 & 0 & 1 & 0 & 0 & 0 & -1 & 1 & 0 & 0 & 0 & 0 & 0 & 0 & 0 & 0 & 0 \\
0 & 0 & 0 & 0 & 0 & 0 & 0 & 0 & 0 & 0 & 0 & -1 & 0 & -1 & 0 & 0 & 0 & 0 & 0 & 0 & 0 & 0 & 0 & 0 & 0 & 0 \\
0 & 0 & 0 & 0 & 0 & 0 & 0 & 0 & 0 & 0 & 0 & 0 & 0 & 0 & -1 & -1 & -1 & -1 & 0 & 0 & 0 & 0 & 0 & 0 & 0 & 0 \\
0 & 0 & 0 & 0 & 0 & 0 & 0 & 0 & 0 & 0 & -1 & -1 & 0 & 0 & 0 & 0 & 0 & -1 & 0 & 1 & 1 & 0 & 0 & 0 & 0 & 0 \\
0 & 0 & 0 & 0 & 0 & 0 & 0 & 0 & 0 & 0 & 0 & 0 & -1 & -1 & 0 & -1 & 0 & 0 & 1 & 0 & 0 & 1 & 0 & 0 & 0 & 0 \\
1 & -1 & 0 & 0 & 0 & 0 & 0 & 0 & 0 & 0 & 0 & 0 & 0 & 0 & 0 & 0 & 0 & 0 & 0 & 0 & 0 & 0 & 0 & 0 & 1 & 0 \\
0 & 0 & 1 & -1 & 0 & 0 & 0 & 0 & 0 & 0 & 0 & 0 & 0 & 0 & 0 & 0 & 0 & 0 & 0 & 0 & 0 & 0 & 0 & 0 & 0 & 1 \\
0 & 0 & 0 & 0 & 0 & 0 & 0 & 0 & 0 & 0 & 1 & 0 & 0 & 0 & 1 & 0 & 0 & 0 & 1 & 0 & -1 & 0 & 1 & 0 & 0 & 0 \\
0 & 0 & 0 & 0 & 0 & 0 & 0 & 0 & 0 & 0 & 0 & 0 & 1 & 0 & 0 & 0 & 1 & 0 & 0 & 1 & 0 & -1 & 0 & 1 & 0 & 0 \\
0 & 0 & 0 & 0 & -1 & 2 & 0 & 0 & 0 & 0 & 0 & 0 & 0 & 0 & 0 & 0 & 0 & 0 & 0 & 0 & 0 & 0 & 1 & 0 & 1 & 0 \\
0 & 0 & 0 & 0 & 0 & 0 & 0 & -1 & 2 & 0 & 0 & 0 & 0 & 0 & 0 & 0 & 0 & 0 & 0 & 0 & 0 & 0 & 0 & 1 & 0 & 1
\end{pmatrix}
$}
\]
    whose reduced row echelon form is the matrix
    \[
\scalebox{0.7}{$
\begin{pmatrix}
1 & 0 & 0 & 0 & 0 & 0 & 0 & 0 & 0 & 0 & 0 & 0 & 0 & 0 & 0 & 0 & 0 & 0 & 0 & 0 & -\frac{1}{2} & -\frac{1}{2} & 0 & 0 & \frac{1}{2} & -\frac{1}{2} \\
0 & 1 & 0 & 0 & 0 & 0 & 0 & 0 & 0 & 0 & 0 & 0 & 0 & 0 & 0 & 0 & 0 & 0 & 0 & 0 & -\frac{1}{2} & -\frac{1}{2} & 0 & 0 & -\frac{1}{2} & -\frac{1}{2} \\
0 & 0 & 1 & 0 & 0 & 0 & 0 & 0 & 0 & 0 & 0 & 0 & 0 & 0 & 0 & 0 & 0 & 0 & 0 & 0 & -\frac{1}{2} & -\frac{1}{2} & 0 & 0 & -\frac{1}{2} & \frac{1}{2} \\
0 & 0 & 0 & 1 & 0 & 0 & 0 & 0 & 0 & 0 & 0 & 0 & 0 & 0 & 0 & 0 & 0 & 0 & 0 & 0 & -\frac{1}{2} & -\frac{1}{2} & 0 & 0 & -\frac{1}{2} & -\frac{1}{2} \\
0 & 0 & 0 & 0 & 1 & 0 & 0 & 0 & 0 & -2 & 0 & 0 & 0 & 0 & 0 & 0 & 0 & 0 & 0 & 0 & 2 & 2 & -1 & 0 & 1 & 0 \\
0 & 0 & 0 & 0 & 0 & 1 & 0 & 0 & 0 & -1 & 0 & 0 & 0 & 0 & 0 & 0 & 0 & 0 & 0 & 0 & 1 & 1 & 0 & 0 & 1 & 0 \\
0 & 0 & 0 & 0 & 0 & 0 & 1 & 0 & 0 & 1 & 0 & 0 & 0 & 0 & 0 & 0 & 0 & 0 & 0 & 0 & -1 & -1 & \frac{1}{2} & \frac{1}{2} & 0 & 0 \\
0 & 0 & 0 & 0 & 0 & 0 & 0 & 1 & 0 & 2 & 0 & 0 & 0 & 0 & 0 & 0 & 0 & 0 & 0 & 0 & 0 & 0 & 1 & 0 & 0 & 1 \\
0 & 0 & 0 & 0 & 0 & 0 & 0 & 0 & 1 & 1 & 0 & 0 & 0 & 0 & 0 & 0 & 0 & 0 & 0 & 0 & 0 & 0 & \frac{1}{2} & \frac{1}{2} & 0 & 1 \\
0 & 0 & 0 & 0 & 0 & 0 & 0 & 0 & 0 & 0 & 1 & 0 & 0 & 0 & 0 & 0 & -1 & 0 & 0 & -1 & -1 & 0 & \frac{1}{2} & -\frac{1}{2} & 0 & 0 \\
0 & 0 & 0 & 0 & 0 & 0 & 0 & 0 & 0 & 0 & 0 & 1 & 0 & 0 & 0 & 0 & 1 & 1 & 0 & 0 & 0 & 0 & -\frac{1}{2} & \frac{1}{2} & 0 & 0 \\
0 & 0 & 0 & 0 & 0 & 0 & 0 & 0 & 0 & 0 & 0 & 0 & 1 & 0 & 0 & 0 & 1 & 0 & 0 & 1 & 0 & -1 & 0 & 1 & 0 & 0 \\
0 & 0 & 0 & 0 & 0 & 0 & 0 & 0 & 0 & 0 & 0 & 0 & 0 & 1 & 0 & 0 & -1 & -1 & 0 & 0 & 0 & 0 & \frac{1}{2} & -\frac{1}{2} & 0 & 0 \\
0 & 0 & 0 & 0 & 0 & 0 & 0 & 0 & 0 & 0 & 0 & 0 & 0 & 0 & 1 & 0 & 1 & 0 & 0 & 0 & 1 & 1 & \frac{1}{2} & \frac{1}{2} & 1 & 1 \\
0 & 0 & 0 & 0 & 0 & 0 & 0 & 0 & 0 & 0 & 0 & 0 & 0 & 0 & 0 & 1 & 0 & 1 & 0 & 0 & -1 & -1 & -\frac{1}{2} & -\frac{1}{2} & -1 & -1 \\
0 & 0 & 0 & 0 & 0 & 0 & 0 & 0 & 0 & 0 & 0 & 0 & 0 & 0 & 0 & 0 & 0 & 0 & 1 & 1 & -1 & -1 & 0 & 0 & -1 & -1
\end{pmatrix}
$}\]
    This $(16 \times 26)$ matrix has rank $16$. Since we know that the space of cocycles has also dimension $16$, this implies that
    $$\text{HH}^2_4(\D_4, \D_4) = 0$$
    as wanted.
\end{proof}

\begin{lemma}\label{lem:7.8}
    We have that
    $$\text{HH}^2_6(\D_4, \D_4) = 0$$
\end{lemma}

\begin{proof}
    Let $\{ (s, \varphi_s + \tilde{\varphi}_s t) \}_{s \in \overline{S}_4}$ be a first order deformation of $\overline{R}_4$ corresponding to a cocycle in $\text{HH}^2_6(\D_4,\D_4)$. Then $\tilde{\varphi} \in \text{Hom}_{\Bbbk (\overline{Q}_4)_0^e}(\Bbbk \overline{S}_4, \Bbbk \text{Irr}_{\overline{S}_4})_6$, so in other words $\tilde{\varphi}$ is a degree $6$ map that fixes the endpoints, and thus $\tilde{\varphi}_s$ has the following form for $s \in \overline{S}_4$
    $$\tilde{\varphi}_{\overline{y}_1\overline{x}_1}
= \alpha_1\overline{x}_2\overline{x}_3\overline{x}_5\overline{x}_7\overline{y}_7\overline{y}_5\overline{y}_3\overline{y}_2,
$$
$$\tilde{\varphi}_{\overline{y}_2\overline{x}_2}
= \alpha_2\overline{x}_3\overline{x}_5\overline{x}_7\overline{x}_9\overline{y}_9\overline{y}_7\overline{y}_5\overline{y}_3,$$
$$\tilde{\varphi}_s = 0 \quad \text{else},$$
for $\alpha_1, \alpha_{2} \in \Bbbk$ without any condition on them because none of the overlaps admit a parallel irreducible path of length $9$. \\
Consider now the coboundary $\psi \in \text{Hom}_{\Bbbk (\overline{Q}_4)_0^e}(\Bbbk (\overline{Q}_4)_1, \Bbbk \text{Irr}_{\overline{S}_4})_6$ defined by
$$\psi_{\overline{y}_1}
= -\alpha_1\overline{x}_2\overline{x}_3\overline{x}_5\overline{y}_5\overline{y}_3\overline{y}_2\overline{y}_1,
$$
$$\psi_{\overline{y}_3}
= \alpha_2\overline{x}_5\overline{x}_7\overline{x}_9\overline{y}_9\overline{y}_7\overline{y}_5\overline{y}_3,$$
$$\psi_{y_i} = \psi_{x_i} = 0 \quad \text{else},$$
Then if we compute
\begin{align*}
\delta(\psi)_{\overline{y}_1\overline{x}_1}
&= T(\psi)(\overline{y}_1\overline{x}_1 + \overline{x}_2\overline{y}_2) \\
&= \sigma\pi(\psi_{\overline{y}_1}\overline{x}_1 + \overline{y}_1\psi_{\overline{x}_1} + \psi_{\overline{x}_2}\overline{y}_2 + \overline{x}_2\psi_{\overline{y}_2} ) \\
&= \sigma\pi(-\alpha_1\overline{x}_2\overline{x}_3\overline{x}_5\overline{y}_5\overline{y}_3\overline{y}_2\overline{y}_1\overline{x}_1) \\
&= \alpha_1\overline{x}_2\overline{x}_3\overline{x}_5\overline{x}_7\overline{y}_7\overline{y}_5\overline{y}_3\overline{y}_2 = \tilde{\varphi}_{\overline{y}_1\overline{x}_1}
\end{align*}
and similarly
$$\delta(\psi)_{\overline{y}_2\overline{x}_2} = \tilde{\varphi}_{\overline{y}_2\overline{x}_2}$$
Thus we have that $\delta(\psi) = \tilde{\varphi}$ and hence $\tilde{\varphi}$ is trivial. So we deduce that
$$\text{HH}^2_6(\D_4, \D_4) = 0$$
as wanted.
\end{proof}

\begin{lemma}\label{lem:7.9}
    We have that
    $$\text{HH}^2_8(\D_4, \D_4) = 0$$
\end{lemma}

\begin{proof}
    Let $\{ (s, \varphi_s + \tilde{\varphi}_s t) \}_{s \in \overline{S}_4}$ be a first order deformation of $\overline{R}_4$ corresponding to a cocycle in $\text{HH}^2_8(\D_4,\D_4)$. Then $\tilde{\varphi} \in \text{Hom}_{\Bbbk (\overline{Q}_4)_0^e}(\Bbbk \overline{S}_4, \Bbbk \text{Irr}_{\overline{S}_4})_8$, so in other words $\tilde{\varphi}$ is a degree $8$ map that fixes the endpoints, and thus $\tilde{\varphi}_s$ has the following form for $s \in \overline{S}_4$
    $$\tilde{\varphi}_{\overline{y}_1\overline{x}_1}
= \alpha_1\overline{x}_2\overline{x}_3\overline{x}_5\overline{x}_7\overline{x}_9\overline{y}_9\overline{y}_7\overline{y}_5\overline{y}_3\overline{y}_2,
$$
$$\tilde{\varphi}_s = 0 \quad \text{else},$$
for $\alpha_1 \in \Bbbk$ without any condition on it because none of the overlaps admit a parallel irreducible path of length $11$. \\
Consider now the coboundary $\psi \in \text{Hom}_{\Bbbk (\overline{Q}_4)_0^e}(\Bbbk (\overline{Q}_4)_1, \Bbbk \text{Irr}_{\overline{S}_4})_8$ defined by
$$\psi_{\overline{y}_1}
= \alpha_1\overline{x}_2\overline{x}_3\overline{x}_5\overline{x}_7\overline{y}_7\overline{y}_5\overline{y}_3\overline{y}_2\overline{y}_1,
$$
$$\psi_{y_i} = \psi_{x_i} = 0 \quad \text{else},$$
Then if we compute
\begin{align*}
\delta(\psi)_{\overline{y}_1\overline{x}_1}
&= T(\psi)(\overline{y}_1\overline{x}_1 + \overline{x}_2\overline{y}_2) \\
&= \sigma\pi(\psi_{\overline{y}_1}\overline{x}_1 + \overline{y}_1\psi_{\overline{x}_1} + \psi_{\overline{x}_2}\overline{y}_2 + \overline{x}_2\psi_{\overline{y}_2} ) \\
&= \sigma\pi(\alpha_1\overline{x}_2\overline{x}_3\overline{x}_5\overline{x}_7\overline{y}_7\overline{y}_5\overline{y}_3\overline{y}_2\overline{y}_1\overline{x}_1) \\
&= \alpha_1\overline{x}_2\overline{x}_3\overline{x}_5\overline{x}_7\overline{x}_9\overline{y}_9\overline{y}_7\overline{y}_5\overline{y}_3\overline{y}_2 = \tilde{\varphi}_{\overline{y}_1\overline{x}_1}
\end{align*}
Thus we have that $\delta(\psi) = \tilde{\varphi}$ and hence $\tilde{\varphi}$ is trivial. So we deduce that
$$\text{HH}^2_8(\D_4, \D_4) = 0$$
as wanted.
\end{proof}

\begin{corollary}\label{cor:7.10}
    The algebra $\D_4$ is intrinsically formal.
\end{corollary}

\newpage

\nocite{*}
\printbibliography[heading=bibintoc]

\end{document}

%% file: arcDiagram.tex
\newcommand{\colorFor}[1]{%
  \ifcase#1
    blue!95!cyan
  \or
    red!90!black
  \or
    green!60!black
  \or
    yellow!90!black
  \or
    orange
  \or
    magenta
  \or
    violet
  \or
    cyan!95!blue
  \or
    green!40!yellow
  \or
    black
  \or
    teal
  \else
    black%
  \fi
}

\tikzset{
  diagram/.style={x=1em,y=1em}
}

\tikzset{
  strand/.style={
    line width=1pt,
    line cap=round
  }
}

\tikzset{
  dot/.style={
    fill,
    circle,
    inner sep=1.5pt
  }
}

\tikzset{
  bigdot/.style={
    fill,
    circle,
    inner sep=2.5pt
  }
}

\tikzset{
  sdot/.style={
    fill,
    circle,
    inner sep=1pt
  }
}

\tikzset{
  colored strand/.style args={#1}{
    strand,
    draw=\colorFor{#1}
  }
}

\tikzset{
  colored dot/.style={
    dot,
    fill=\colorFor{#1},
    draw=\colorFor{#1}
  }
}

\newcommand{\UP}[1]{%
  \fill
    (#1,0.16)
    -- (#1-0.14,-0.16)
    -- (#1,-0.08)
    -- (#1+0.14,-0.16)
    -- cycle;
}

\newcommand{\DN}[1]{%
  \fill
    (#1,-0.16)
    -- (#1-0.14,0.16)
    -- (#1,0.08)
    -- (#1+0.14,0.16)
    -- cycle;
}

\newcommand{\BUP}[1]{%
  \draw[strand] (#1,-0.1) -- (#1,0.03);
  \fill
    (#1,0.32)
    -- (#1-0.24,-0.32)
    -- (#1,-0.16)
    -- (#1+0.24,-0.32)
    -- cycle;
}

\newcommand{\BDN}[1]{%
  \draw[strand] (#1,0.1) -- (#1,-0.03);
  \fill
    (#1,-0.32)
    -- (#1-0.24,0.32)
    -- (#1,0.16)
    -- (#1+0.24,0.32)
    -- cycle;
}

\newcommand{\cBUP}[2]{%
  \draw[colored strand={#2}] (#1,-0.1) -- (#1,0.03);
  \fill[\colorFor{#2}]
    (#1,0.32)
    -- (#1-0.24,-0.32)
    -- (#1,-0.16)
    -- (#1+0.24,-0.32)
    -- cycle;
}

\newcommand{\cBDN}[2]{%
  \draw[colored strand={#2}] (#1,0.1) -- (#1,-0.03);
  \fill[\colorFor{#2}]
    (#1,-0.32)
    -- (#1-0.24,0.32)
    -- (#1,0.16)
    -- (#1+0.24,0.32)
    -- cycle;
}

\newcommand{\cUP}[2]{%
  \draw[colored strand={#2}] (#1,-0.1) -- (#1,0.03);
  \fill[\colorFor{#2}]
    (#1,0.16)
    -- (#1-0.14,-0.16)
    -- (#1,-0.08)
    -- (#1+0.14,-0.16)
    -- cycle;
}

\newcommand{\cDN}[2]{%
  \draw[colored strand={#2}] (#1,0.1) -- (#1,-0.03);
  \fill[\colorFor{#2}]
    (#1,-0.16)
    -- (#1-0.14,0.16)
    -- (#1,0.08)
    -- (#1+0.14,0.16)
    -- cycle;
}

\newcommand{\weight}[1]{%
  \foreach[count=\i from 0] \w in {#1} {%
    \ifnum\w=1
      \UP{\i}
    \else
      \DN{\i}
    \fi
  }%
}

\newcommand{\Bweight}[1]{%
  \foreach[count=\i from 0] \w in {#1} {%
    \ifnum\w=1
      \BUP{\i}
    \else
      \BDN{\i}
    \fi
  }%
}

\newcommand{\Cweight}[2]{%
  \foreach[count=\i from 0] \w [count=\j from 0] in {#1} {%
    \pgfmathparse{{#2}[\j]}%
    \let\col\pgfmathresult
    \ifnum\w=1
      \cUP{\i}{\col}
    \else
      \cDN{\i}{\col}
    \fi
  }%
}

\newcommand{\CBweight}[2]{%
  \foreach[count=\i from 0] \w [count=\j from 0] in {#1} {%
    \pgfmathparse{{#2}[\j]}%
    \let\col\pgfmathresult
    \ifnum\w=1
      \cBUP{\i}{\col}%
    \else
      \cBDN{\i}{\col}%
    \fi
  }%
}

\newcommand{\CAP}[3]{%
  \pgfmathsetmacro{\mid}{(#1+#2)/2}%
  \pgfmathsetmacro{\rad}{(#2-#1)/2}%
  \draw[strand] (\mid-\rad,0.21)
    arc[start angle=180,end angle=0,
        x radius=\rad,
        y radius=#3*\rad];
}

\newcommand{\CUP}[3]{%
  \pgfmathsetmacro{\mid}{(#1+#2)/2}%
  \pgfmathsetmacro{\rad}{(#2-#1)/2}%
  \draw[strand] (\mid-\rad,-0.21)
    arc[start angle=-180,end angle=0,
        x radius=\rad,
        y radius=#3*\rad];
}

\newcommand{\DCAP}[3]{%
  \pgfmathsetmacro{\mid}{(#1+#2)/2}%
  \pgfmathsetmacro{\rad}{(#2-#1)/2}%
  \draw[strand, dotted] (\mid-\rad,0.21)
    arc[start angle=180,end angle=0,
        x radius=\rad,
        y radius=#3*\rad];
}

\newcommand{\DCUP}[3]{%
  \pgfmathsetmacro{\mid}{(#1+#2)/2}%
  \pgfmathsetmacro{\rad}{(#2-#1)/2}%
  \draw[strand, dotted] (\mid-\rad,-0.21)
    arc[start angle=-180,end angle=0,
        x radius=\rad,
        y radius=#3*\rad];
}

\newcommand{\CCAP}[4]{%
  \pgfmathsetmacro{\mid}{(#1+#2)/2}%
  \pgfmathsetmacro{\rad}{(#2-#1)/2}%
  \draw[colored strand={#4}] (\mid-\rad,0.21)
    arc[start angle=180,end angle=0,
        x radius=\rad,
        y radius=#3*\rad];
}

\newcommand{\CCUP}[4]{%
  \pgfmathsetmacro{\mid}{(#1+#2)/2}%
  \pgfmathsetmacro{\rad}{(#2-#1)/2}%
  \draw[colored strand={#4}] (\mid-\rad,-0.21)
    arc[start angle=-180,end angle=0,
        x radius=\rad,
        y radius=#3*\rad];
}

\newcommand{\DOWNLINE}[2]{%
  \draw[strand] (#1,-0.21) -- (#1,-#2);
}

\newcommand{\UPLINE}[2]{%
  \draw[strand] (#1,0.21) -- (#1,#2);
}

\newcommand{\LINE}[2]{%
  \draw[strand] (#1,0.21) -- (#1,#2);
  \draw[strand] (#1,-0.21) -- (#1,-#2);
}

\newcommand{\RLINE}[2]{%
  \draw[strand] (#1,0) arc[start angle=180, end angle=130, x radius=0.5, y radius=#2];
  \draw[strand] (#1,0) arc[start angle=180, end angle=230, x radius=0.5, y radius=#2];
  \draw[strand, dotted] (#1,0) arc[start angle=180, end angle=100, x radius=0.5, y radius=#2];
  \draw[strand, dotted] (#1,0) arc[start angle=180, end angle=260, x radius=0.5, y radius=#2];
}

\newcommand{\TRLINE}[2]{%
  \draw[strand] (#1,0) arc[start angle=180, end angle=130, x radius=0.5, y radius=#2];
  \draw[strand, dotted] (#1,0) arc[start angle=180, end angle=100, x radius=0.5, y radius=#2];
}

\newcommand{\BRLINE}[2]{%
  \draw[strand] (#1,0) arc[start angle=180, end angle=230, x radius=0.5, y radius=#2];
  \draw[strand, dotted] (#1,0) arc[start angle=180, end angle=260, x radius=0.5, y radius=#2];
}

\newcommand{\CRLINE}[3]{%
  \draw[colored strand={#3}] (#1,0) arc[start angle=180, end angle=130, x radius=0.5, y radius=#2];
  \draw[colored strand={#3}] (#1,0) arc[start angle=180, end angle=230, x radius=0.5, y radius=#2];
  \draw[colored strand={#3}, dotted] (#1,0) arc[start angle=180, end angle=100, x radius=0.5, y radius=#2];
  \draw[colored strand={#3}, dotted] (#1,0) arc[start angle=180, end angle=260, x radius=0.5, y radius=#2];
}

\newcommand{\FILL}[1]{%
  \foreach \i in {0,...,\numexpr#1-1\relax} {
    \draw[strand] (\i,0.21) -- (\i,-0.21);
  }
}

\newcommand{\BFILL}[1]{%
  \foreach \i in {0,...,\numexpr#1-1\relax} {
    \draw[strand] (\i,0) -- (\i,-0.21);
  }
}

\newcommand{\BBFILL}[1]{%
    \draw[strand] (#1,0) -- (#1,-0.21);
}

\newcommand{\TFILL}[1]{%
  \foreach \i in {0,...,\numexpr#1-1\relax} {
    \draw[strand] (\i,0.21) -- (\i,0);
  }
}

\newcommand{\CFILL}[2]{%
  \foreach[count=\i from 0] \c in {#2} {%
    \ifnum\i<\numexpr#1\relax
      \draw[colored strand={\c}] (\i,0.21) -- (\i,-0.21);
    \fi
  }%
}

\newcommand{\PFILL}[1]{%
    \draw[strand] (#1,0.21) -- (#1,-0.21);
}

\newcommand{\CPFILL}[2]{%
    \draw[colored strand = {#2}] (#1,0.21) -- (#1,-0.21);
}

\newcommand{\dCAP}[3]{%
  \pgfmathsetmacro{\mid}{(#1+#2)/2}%
  \pgfmathsetmacro{\rad}{(#2-#1)/2}%
  \draw[strand] (\mid-\rad,0.21)
    arc[start angle=180,end angle=0,
        x radius=\rad,
        y radius=#3*\rad];
  \node[dot] at (\mid, 0.21+#3*\rad) {};%
}

\newcommand{\bigdCAP}[3]{%
  \pgfmathsetmacro{\mid}{(#1+#2)/2}%
  \pgfmathsetmacro{\rad}{(#2-#1)/2}%
  \draw[strand] (\mid-\rad,0.21)
    arc[start angle=180,end angle=0,
        x radius=\rad,
        y radius=#3*\rad];
  \node[bigdot] at (\mid, 0.21+#3*\rad) {};%
}

\newcommand{\dCUP}[3]{%
  \pgfmathsetmacro{\mid}{(#1+#2)/2}%
  \pgfmathsetmacro{\rad}{(#2-#1)/2}%
  \draw[strand] (\mid-\rad,-0.21)
    arc[start angle=-180,end angle=0,
        x radius=\rad,
        y radius=#3*\rad];
  \node[dot] at (\mid, -0.21-#3*\rad) {};%
}

\newcommand{\bigdCUP}[3]{%
  \pgfmathsetmacro{\mid}{(#1+#2)/2}%
  \pgfmathsetmacro{\rad}{(#2-#1)/2}%
  \draw[strand] (\mid-\rad,-0.21)
    arc[start angle=-180,end angle=0,
        x radius=\rad,
        y radius=#3*\rad];
  \node[bigdot] at (\mid, -0.21-#3*\rad) {};%
}

\newcommand{\dDOWNLINE}[2]{%
  \draw[strand] (#1,-0.21) -- (#1,-#2);
  \node[dot] at (#1, {(-0.21-#2)/2}) {};%
}

\newcommand{\bigdDOWNLINE}[2]{%
  \draw[strand] (#1,-0.21) -- (#1,-#2);
  \node[bigdot] at (#1, {(-0.21-#2)/2}) {};%
}

\newcommand{\dUPLINE}[2]{%
  \draw[strand] (#1,0.21) -- (#1,#2);
  \node[dot] at (#1, {(0.21+#2)/2}) {};%
}

\newcommand{\bigdUPLINE}[2]{%
  \draw[strand] (#1,0.21) -- (#1,#2);
  \node[bigdot] at (#1, {(0.21+#2)/2}) {};%
}

\newcommand{\dLINE}[2]{%
  \draw[strand] (#1,0.21) -- (#1,#2);
  \node[dot] at (#1, {(0.21+#2)/2}) {};%
  \draw[strand] (#1,-0.21) -- (#1,-#2);
  \node[dot] at (#1, {(-0.21-#2)/2}) {};%
}

\newcommand{\sdCAP}[3]{%
  \pgfmathsetmacro{\mid}{(#1+#2)/2}%
  \pgfmathsetmacro{\rad}{(#2-#1)/2}%
  \draw[strand] (\mid-\rad,0.21)
    arc[start angle=180,end angle=0,
        x radius=\rad,
        y radius=#3*\rad];
  \node[sdot] at (\mid, 0.21+#3*\rad) {};%
}

\newcommand{\sdCUP}[3]{%
  \pgfmathsetmacro{\mid}{(#1+#2)/2}%
  \pgfmathsetmacro{\rad}{(#2-#1)/2}%
  \draw[strand] (\mid-\rad,-0.21)
    arc[start angle=-180,end angle=0,
        x radius=\rad,
        y radius=#3*\rad];
  \node[sdot] at (\mid, -0.21-#3*\rad) {};%
}

\newcommand{\sdDOWNLINE}[2]{%
  \draw[strand] (#1,-0.21) -- (#1,-#2);
  \node[sdot] at (#1, {(-0.21-#2)/2}) {};%
}

\newcommand{\sdUPLINE}[2]{%
  \draw[strand] (#1,0.21) -- (#1,#2);
  \node[sdot] at (#1, {(0.21+#2)/2}) {};%
}

\newcommand{\sdLINE}[2]{%
  \draw[strand] (#1,0.21) -- (#1,#2);
  \node[sdot] at (#1, {(0.21+#2)/2}) {};%
  \draw[strand] (#1,-0.21) -- (#1,-#2);
  \node[sdot] at (#1, {(-0.21-#2)/2}) {};%
}

\newcommand{\CdCUP}[4]{%
  \pgfmathsetmacro{\mid}{(#1+#2)/2}%
  \pgfmathsetmacro{\rad}{(#2-#1)/2}%
  \draw[colored strand={#4}] (\mid-\rad,-0.21)
    arc[start angle=-180,end angle=0,
        x radius=\rad,
        y radius=#3*\rad];
  \node[colored dot={#4}] at (\mid, -0.21-#3*\rad) {};%
}

\newcommand{\CdCAP}[4]{%
  \pgfmathsetmacro{\mid}{(#1+#2)/2}%
  \pgfmathsetmacro{\rad}{(#2-#1)/2}%
  \draw[colored strand={#4}] (\mid-\rad,0.21)
    arc[start angle=180,end angle=0,
        x radius=\rad,
        y radius=#3*\rad];
  \node[colored dot={#4}] at (\mid, 0.21+#3*\rad) {};%
}

\newcommand{\pdLINE}[2]{%
  \draw[strand] (#1,0.21) -- (#1,#2*.8);
  \draw[strand,dotted] (#1,0.21) -- (#1,#2*1.2);
  \node[dot] at (#1, {(0.21+#2)/2.2}) {};%
  \draw[strand] (#1,-0.21) -- (#1,-#2*.8);
  \draw[strand,dotted] (#1,-0.21) -- (#1,-#2*1.2);
  \node[dot] at (#1, {(-0.21-#2)/2.2}) {};%
}

\newcommand{\dRLINE}[2]{%
  \draw[strand] (#1,0) arc[start angle=180, end angle=130, x radius=0.5, y radius=#2]
    node[dot, pos=0.75] {};
  \draw[strand] (#1,0) arc[start angle=180, end angle=230, x radius=0.5, y radius=#2]
    node[dot, pos=0.75] {};
  \draw[strand, dotted] (#1,0) arc[start angle=180, end angle=100, x radius=0.5, y radius=#2];
  \draw[strand, dotted] (#1,0) arc[start angle=180, end angle=260, x radius=0.5, y radius=#2];
}

\newcommand{\dBRLINE}[2]{%
  \draw[strand] (#1,0) arc[start angle=180, end angle=230, x radius=0.5, y radius=#2]
    node[dot, pos=0.75] {};
  \draw[strand, dotted] (#1,0) arc[start angle=180, end angle=260, x radius=0.5, y radius=#2];
}

\newcommand{\dLLINE}[2]{%
  \draw[strand] (#1,0) arc[start angle=0, end angle=50, x radius=0.5, y radius=#2]
    node[dot, pos=0.75] {};
  \draw[strand] (#1,0) arc[start angle=0, end angle=-50, x radius=0.5, y radius=#2]
    node[dot, pos=0.75] {};
  \draw[strand, dotted] (#1,0) arc[start angle=0, end angle=80, x radius=0.5, y radius=#2];
  \draw[strand, dotted] (#1,0) arc[start angle=0, end angle=-80, x radius=0.5, y radius=#2];
}

\newcommand{\CpdLINE}[3]{%
  \draw[colored strand={#3}] (#1,0.21) -- (#1,#2*.8);
  \draw[colored strand={#3},dotted] (#1,0.21) -- (#1,#2*1.2);
  \node[colored dot = {#3}] at (#1, {(0.21+#2)/2.2}) {};%
  \draw[colored strand={#3}] (#1,-0.21) -- (#1,-#2*.8);
  \draw[colored strand={#3},dotted] (#1,-0.21) -- (#1,-#2*1.2);
  \node[colored dot = {#3}] at (#1, {(-0.21-#2)/2.2}) {};%
}

\newcommand{\BpdLINE}[2]{%
  \draw[strand] (#1,-0.21) -- (#1,-#2*.8);
  \draw[strand,dotted] (#1,-0.21) -- (#1,-#2*1.2);
  \node[dot] at (#1, {(-0.21-#2)/2.2}) {};%
}

\newcommand{\TpdLINE}[2]{%
  \draw[strand] (#1,0.21) -- (#1,#2*.8);
  \draw[strand,dotted] (#1,0.21) -- (#1,#2*1.2);
  \node[dot] at (#1, {(0.21+#2)/2.2}) {};%
}